\documentclass{amsart}
\usepackage{amsfonts,amssymb,amsmath,amsthm}
\usepackage{url}
\usepackage{enumerate}
\usepackage{hyperref}
\usepackage{fancyhdr}
\newtheorem{theorem}{Theorem}[section]
\newtheorem{lemma}[theorem]{Lemma}
\newtheorem{conjecture}[theorem]{Conjecture}
\newtheorem{proposition}[theorem]{Proposition}
\newtheorem{corollary}[theorem]{Corollary}
\newtheorem{example}[theorem]{Example}
\newtheorem{examples}[theorem]{Examples}
\theoremstyle{definition}
\newtheorem{definition}[theorem]{Definition}
\newtheorem{definitions}[theorem]{Definitions}
\newtheorem{heuristic}[theorem]{Heuristic}
\newtheorem{remark}[theorem]{Remark}
\newtheorem{remarks}[theorem]{Remarks}
\numberwithin{equation}{section}
\def\dsfrac{\displaystyle \frac}
\def\notdiv{\nmid}
\def\div{\,\vert\,}
\let\sst=\scriptscriptstyle
\def\mb{\,{\sst \bullet}\,}
\def\lien{\mathrel{\mkern-4mu}}
\def\too{\relbar\lien\rightarrow}
\def\tooo{\relbar\lien\relbar\lien\too}

\def\N{\mathbb{N}}

\def\C{\mathbb{C}}
\def\Q{\mathbb{Q}}
\def\Z{\mathbb{Z}}
\def\F{\mathbb{F}}
\def\No{{\rm N}}
\def\Cl{{\mathcal C}\hskip-2pt{\ell}}
\def\Pl{{\mathcal P}\hskip-2pt{\ell}}
\def\Frac#1#2{\hbox{\footnotesize $\displaystyle \frac{#1}{#2}$}}
\def\plus{\displaystyle\mathop{\raise 2.0pt \hbox{$\bigoplus $}}\limits}
\def\prd{\displaystyle\mathop{\raise 2.0pt \hbox{$\prod$}}\limits}
\def\sm{\displaystyle\mathop{\raise 2.0pt \hbox{$\sum$}}\limits}

\author[Georges Gras]{Georges Gras}
\address{Villa la Gardette \\ chemin Ch\^ateau Gagni\`ere \\ F--38520 Le Bourg d'Oisans.}
\email{g.mn.gras@wanadoo.fr {\it url\,:\,}\url{http://www.researchgate.net/profile/Georges_Gras} }



\begin{document}
 
\title[Local $\theta$-regulators -- $p$-adic Conjectures] {Local $\theta$-regulators of 
an algebraic number \\ $p$-adic Conjectures }

\date{January 7, 2017}

\begin{abstract}
\vspace{-0.2cm}
Let $K/\Q$ be Galois and let $\eta \in K^\times$ be such that 
the multiplicative $\Z[G]$-module generated by $\eta$ is of $\Z$-rank $n$.
We define the local $\theta$-regulators $\Delta_p^\theta(\eta) \in \F_p$ for the 
$\Q_p\,$-irreducible characters $\theta$ of $G={\rm Gal}(K/\Q)$. 
Let $V_\theta$ be the $\theta$-irreducible representation.
A linear representation ${\mathcal L}^\theta\simeq \delta \, V_\theta$ is associated 
with $\Delta_p^\theta (\eta)$ whose nullity is equivalent to $\delta \geq 1$
(Theorem \ref{theo24}).
Each $\Delta_p^\theta (\eta)$ yields ${\rm Reg}_p^\theta (\eta)$ modulo $p$ 
in the factorization $\prd_{\theta} \big({\rm Reg}_p^\theta (\eta) \big)^{\varphi(1)}$ 
of ${\rm Reg}_p^G (\eta) := \frac{ {\rm Reg}_p(\eta)}{p^{[K : \Q\,]} }$ 
(normalized $p$-adic regulator of $\eta$), 
where $\varphi \mid \theta$ is absolutely irreducible.
From the probability ${\rm Prob}\big (\Delta_p^\theta(\eta) = 0 \  \& \  
{\mathcal L}^\theta \simeq \delta \, V_\theta\big ) \leq p^{- f \delta^2}$ 
($f=$ residue degree of $p$ in the field of values of $\varphi$) 
and the Borel--Cantelli heuristic, we conjecture that, for $p$ large enough, 
${\rm Reg}_p^G (\eta)$ is a $p$-adic unit or that 
$p^{\varphi(1)} \parallel {\rm Reg}_p^G (\eta)$ 
(existence of a single $\theta$ of $G$ with $f=\delta=1$ 
and no extra $p$-divisibility); 
this obstruction may be lifted assuming the existence of a 
binomial probability law (Sec. \ref{section7}) confirmed 
through numerical studies (with groups $G = C_3$, $C_5$, $D_6$).
This conjecture would imply that, {\it for all $p$ large enough}, Fermat quotients of rationals and
normalized $p$-adic regulators are $p$-adic units (Theorem. \ref{heur3}), whence the fact
that number fields are $p$-rational for $p \gg 0$. We recall \S\,\ref{subBK} some deep
cohomological results, which may strengthen such conjectures.
\end{abstract}

\centerline{English\ translation  of  the  original  article:}

\medskip
\centerline {Les\  $\theta$-r\'egulateurs locaux d'un nombre alg\'ebrique  -- Conjectures  $p$-adiques }

\medskip
\centerline{\it {Canadian Journal of Mathematics}, Vol. {\bf 68}, 3 (2016), 571--624}

\medskip
\maketitle

\vspace{-0.4cm}
\section{Introduction}\label{section1}
Let $K/\Q$ be a Galois extension of degree $n$ of Galois group $G$.
Let $\eta \in K^\times$. An exponential notation is used for conjugation of $\eta$ by $\sigma \in G$, 
which implies the writing $(\eta^\sigma)^\tau =: \eta^{\tau \sigma}$ for all $\sigma, \tau \in G$
(law of left $G$-module). 
We assume that the multiplicative $\Z[G]$-module generated by $\eta$ is of $\Z$-rank $n$
(i.e., $\langle \eta \rangle_G^{} \otimes \Q \simeq \Q[G]$).
For $p$ large enough, we put ${\rm Reg}_p^G(\eta) :=\hbox{det} \big(\hbox{$\frac{-1}{p}$} 
{\rm log}_p(\eta^{\tau \sigma^{-1}}) \big)_{\sigma, \tau \in G}$ (normalized $p$-adic regulator of $\eta$).

\smallskip
We shall see that the unique obstruction, to apply the heuristic principle of Borel--Cantelli leading 
(conjecturally) to a finite 
number of $p$ such that ${\rm Reg}_p^G(\eta) \equiv 0 \pmod p$, is related to primes $p$ such that 
${\rm Reg}_p^G(\eta)$ is exactely divisible by a minimal power of $p$; this is equivalent to 
${\rm Reg}_p^G(\eta) \sim p^{\varphi(1)}$ (equality up to a $p$-adic unit), where the character $\varphi$ 
(absolutely irreducible) defines a $p$-adic character $\theta$ satisfying certain conditions 
(Definition \ref{defidec}). 

Such a situation is a priori of probability at most $\Frac{O(1)}{p}$,
only when $\eta$ is considered as a random variable; it is the unique case where the
Borel--Cantelli principle does not apply (see Section \ref{concl} for some enlightenment). 
We intend, from heuristics and numerical experiments, to remove this obstruction and 
to reach the following probabilistic result, {\it when $\eta$ is fixed and $p \to \infty$}:

\begin{theorem}\label{heur3} {\it Let $K/\Q$ be a Galois extension of degree $n$ and of Galois group $G$.
Let $\eta\in K^\times$ be fixed, $\eta$ generating a multiplicative $\Z[G]$-module of $\Z$-rank $n$.

\smallskip  
(i) Under the Heuristic \ref{heur2} (existence of a classical binomial law of probability), 
the probability to have ${\rm Reg}_p^G(\eta) \equiv 0 \pmod p$ is at most
$\dsfrac{C_\infty(\eta)}{p^{ {\rm log}_2(p)/ {\rm log}(c_0(\eta))-O(1) }}$ for $p\to \infty$, where 
$c_0(\eta) = {\rm max}_{\sigma \in G} (\vert \eta^\sigma \vert)$, $e^{-1} \leq C_\infty(\eta) \leq 1$,
and ${\rm log}_2={\rm log} \circ {\rm log}$.

\smallskip  
(ii) Under the previous heuristic \ref{heur2} and the principle of Borel--Cantelli, the number of 
primes $p$ such that ${\rm Reg}_p^G(\eta) \equiv 0 \pmod p$ is finite. }
\end{theorem}

 We shall always suppose that the prime number $p$ that we consider is large enough, in particular odd, 
not divisor of $ n $, unramified in $K$, and prime to $\eta$, so that the normalized $p$-adic regulator
${\rm Reg}_p^G (\eta) := p^{-[K : \Q\,]} \cdot {\rm Reg}_p(\eta)$ makes sense in $\Z_p$, 
where ${\rm Reg}_p(\eta)$ is the usual $p$-adic regulator of $\eta$ (see \S\,\ref{reg}).

\smallskip
Denote by $Z_K$ (resp. $Z_{K,(p)}$) the ring of integers (resp. of $p$-integers) of $K$; for $K=\Q$, one gets
$\Z$ (resp. $\Z_{(p)}$).
For all place $v \div p$ of $K$, we denote by ${\mathfrak p}_v \div p$ the prime ideal associated with $v$. 

\smallskip
If $n_p$ is the common residue degree of the places $v \div p$ in $K/\Q$, the multiplicative groups 
of the residue fields are of order $p^{n_p}-1$ and  for all $v \div p$ we have the congruence
$\eta^{p^{n_p}-1}\! \equiv 1 \!\!\pmod{ {\mathfrak p}_v}$; hence finally, since 
$\prd_{{\mathfrak p}_v \div p} {\mathfrak p}_v = p\, Z_K$,
$$\eta^{p^{n_p}-1} = 1 + p\,\alpha_p(\eta), \ \ \alpha_p(\eta)\in Z_{K,(p)}, $$

which leads, by Galois, to the relations
$$\hbox{ $\alpha_p(\eta^\sigma) = \alpha_p(\eta)^\sigma$ for all $\sigma \in G$, }$$

and to the ``logarithmic'' properties
$$\hbox{ $\alpha_p(\eta\,\eta') \equiv \alpha_p(\eta) + \alpha_p(\eta') \!\!\! \pmod {p\,Z_{K,(p)}} \ \ \&\ \ 
\alpha_p(\eta^\lambda) \equiv \lambda\, \alpha_p(\eta)\! \!\! \pmod {p\,Z_{K,(p)}}$} $$
(for $\eta, \eta' \in K^\times$, $\lambda \in \Z$).

\medskip
This {\it generalized Fermat quotient} $\alpha_p(\eta)$ of $\eta$ is the key element of our study. More precisely,
the properties of the $G$-module generated by $\alpha_p(\eta)$ modu\-lo ${p\,Z_{K,(p)} }$ shall precise 
the properties of the {\it normalized $p$-adic regulators} of $\eta$, in particular for the search of the (rares) 
solutions $p$ giving their divisibility by $p$.
The numerical illustrations are obtained by means of PARI programs (from \cite{P}).

\section{Regulators and Representations -- Local regulators}\label{section2}

\subsection{\texorpdfstring{$p$}{Lg}-adic logarithm -- 
\texorpdfstring{$p$}{Lg}-adic regulators} \label{sub1}
Let $p$ be a fixed prime number satisfying the hypothesis given in the Introduction. 
We suppose that the number field $K$ is considered as a subfield of $\C_p$. 
Thus, any ``embedding'' of $K$ into $\C_p$ is nothing else than a $\Q$-automorphism 
$\sigma \in G$. 

\smallskip
Let ${\mathfrak p}_0={\mathfrak p}_{v_0}$ be a prime ideal of $K$ above $p$
and let $D_{{\mathfrak p}_0}$ be its decomposition group.
The places $v\div p$, conjugates of $v_0$, correspond to the $(G : D_{{\mathfrak p}_0})$ 
distinct prime ideals ${\mathfrak p}_v := {\mathfrak p}_0^{\sigma_v}$, where 
$(\sigma_v)_{v\div p}$ is an exact system of representatives of $G/D_{{\mathfrak p}_0}$.

\smallskip
We consider the $\Q_p[G]$-module $\prd_{v \div p} K_v$ where 
$K_v = \sigma_v (K)\,\Q_p \subset \C_p$,
is the completion of $K$ at $v$; as $K/\Q$ is Galois, $K_v/\Q_p$ is 
independent of $v\div p$ but the notation $K_v$ recalls that this local 
extension is {\it provided} with the embedding $\sigma_v\, : \,K \to K_v \subset \C_p$, 
which allows the diagonal embedding with dense image

\smallskip
\centerline{$i_p := (\sigma_v)_{v\div p} : K \tooo \prd_{v \div p} K_v $}

where $i_p(x) :=(\sigma_v(x))_{v\div p}$, and which implies that 
$K \otimes \Q_p \simeq \prd_{v \div p} K_v \simeq 
\Q_p[G]$ (semi-local theory). By abuse, if $x\in K$, we shall write 
$x \in \prd_{v \div p} K_v $, $i_p$ being understood.

\subsubsection{\texorpdfstring{$p$}{Lg}-adic logarithm 
on \texorpdfstring{$K^\times$}{Lg}}\label{log}
The $p$-adic logarithm ${\rm log}_p\! :  K^\times \to K\Q_p$ is defined on the set 
$\big \{1+p\,x ,  \ x\in Z_{K,(p)} \big\}$, by means of the usual series ($p>2$)

\smallskip
\centerline{${\rm log}_p (1+p\, x) = \sm_{i \ge 1}\, (-1)^{i+1} \, \Frac{(px)^i}{i} \equiv p\,x \pmod{p^2}$, }

noting that $\sum_{i = 1}^N\, (-1)^{i+1} \, \Frac{(px)^i}{i} \in K$ for any $N \geq 1$.
In the case of $\gamma \in K^\times_{(p)}$, we use the functional relation
$${\rm log}_p(\gamma) = \Frac{1}{p^{n_p}-1} {\rm log}_p(\gamma^{p^{n_p}-1}) = 
\Frac{1}{p^{n_p}-1}{\rm log}_p(1 + p\,\alpha_p(\gamma) ) \equiv - p\,\alpha_p(\gamma) \pmod {p^2} . $$

More generally, this ${\rm log}_p$ function, seen modulo $p^{N+1}$, $N\geq 1$, is represented by elements
of $Z_{K,(p)}$ and is an homomorphism of $G$-modules for the law defined, for all $\sigma \in G$, by
$\sigma \big ({\rm log}_p(\gamma)\! \pmod {p^{N+1}} \big) := {\rm log}_p(\gamma^{\sigma})\! \pmod {p^{N+1}}$,
using the congruence (where $N'$ is an obvious function of $N$)
$$\sigma\big ({\rm log}_p(\gamma)\!\!\!\! \pmod {p^{N+1}} \big) 
\equiv \Frac{1}{p^{n_p}-1} \sm_{1 \leq i \leq N'} \, (-1)^{i+1} \, 
\Frac{(p\,\alpha_p(\gamma)^\sigma)^i}{i} \! \pmod {p^{N+1}} , $$

defining an element of $Z_{K,(p)}$ which approximates 
${\rm log}_p(\gamma^\sigma)$ modulo $p^{N+1}$. So $\sigma_v({\rm log}_p(\gamma))$
makes sense in $K_v$ for all $v \mid p$.

\subsubsection{\texorpdfstring{$p$}{Lg}-adic rank} \label{rangp}
Let ${\rm Log}_{p} :={\rm log}_p \circ \, i_p = (\sigma_v)_{v \div p}\,\circ\,{\rm log}_p$ 
be the homomorphism of $G$-modules defined, on the subgroup of 
elements of $K^\times$ prime to $p$, by 
$${\rm Log}_{p} (\gamma) = \big ( {\rm log}_p (\gamma ^{\sigma_v}) \big)_{v \div p}^{} 
\!\!\in \prd_{v \div p} K_v. $$

 Let $\eta \in K^\times$, prime to $p$, and let $F$ be the 
$\Z[G]$-module generated by $\eta$.
We call $p$-adic rank of $F$, the integer
$${\rm rg}_p(F):= {\rm dim}^{}_{\, \Q_p}(\Q_p {\rm Log}_{p}(F)). $$

The use of ${\rm Log}_{p}$ is a commodity since by conjugation by the elements of $G$, the 
know\-ledge of ${\rm log}_p$ implies that of ${\rm Log}_{p}$ and conversely by projections
$\prd_{v' \div p} K_{v'} \to K_{v}$.

To make a link with the concept of $p$-adic regulator, we shall 
prove first the following two technical results:

\begin{lemma}\label{lemm01} Let $p$ be an odd prime, unramified in $K$, and let $\lambda \in Z_{K,(p)}$. 
If $\lambda \notin p\, Z_{K,(p)}$, there exists $u \in K^\times$, prime to $p$,  such that 
${\rm Tr}_{K/\Q}(\lambda u) \not\equiv 0 \pmod p$.
\end{lemma}

\begin{proof} For all $u \in K^\times$, prime to $p$, consider the diagonal embedding of 
$\lambda u$ in $\prd_{v \div p} K_v$, 
and let ${\rm Tr}_v$ be the local traces ${\rm Tr}_{K_v/\Q_p}$ for $v \div p$. Then 
$${\rm Tr}_{K/\Q}(\lambda u) = \sm_{v \div p} {\rm Tr}_v(\sigma_v(\lambda u)).$$

By assumption, there exists a non-empty set $\Sigma$ of places $v \div p$
such that $\sigma_v(\lambda)$ (hence $\sigma_v(\lambda u) = \sigma_v(\lambda)\,\sigma_v(u)$ for all 
$u$ prime to $p$) is a unit of $K_v$. 

\smallskip
For $v_1 \in\,\Sigma$, write

\medskip
\centerline{ ${\rm Tr}_{K/\Q}(\lambda u) = \sm_{v \div p,\, v\ne v_1} {\rm Tr}_v(\sigma_v(\lambda u)) +
 {\rm Tr}_{v_1}(\sigma_{v_1}(\lambda u)) =: a + {\rm Tr}_{v_1}(\sigma_{v_1}(\lambda u))$. }

As $p$ is unramified in $K$, the residue traces at $p$ are surjective and since 
$\sigma_{v_1}(\lambda u)$ is a unit,
it is sufficient to take a suitable $u \equiv 1 \pmod {\prod_{v ,\,v \ne v_1} {\mathfrak p}_v}$ 
(in which case $a \in \Z_p \pmod p$ does not depend on $u$) and 
$u \equiv u_1 \pmod{ {\mathfrak p}_{v_1}}$ such that for instance 
${\rm Tr}_{v_1}(\sigma_{v_1}(\lambda u)) \equiv 1 - a \!\pmod p$ if $a\not\equiv 1\!\! \pmod p$
 (resp. $1\!\! \pmod p$ if $a \equiv 1\pmod p$). 
Whence ${\rm Tr}_{K/\Q}(\lambda u) \equiv 1 \hbox{\ (resp. 2)} \pmod p$.\,\footnote{\,For $p=2$, 
$K = \Q(\sqrt{17})$, $\lambda = 1+2 \sqrt {17}$, there is no solution $u$ prime to $2$.}
\end{proof}

The following lemma, valid for any $p>2$ unramified, prime to $\eta$, will 
be especially useful to us (from \cite[\S\,5.5, proof of Theorem\,5.31]{Wa}): 

\begin{lemma} \label{lemm1} Let $\eta \in K^\times$, prime to $p$, and let $\lambda(\sigma)$, 
$\sigma \in G$, be $p$-integer coefficients of $K\Q_p$, not all divisible by $p$. 
Suppose that we have the relation of dependence modulo $p^{N+1}$, $N \geq 1$, 
of the $n$ vectors $\ell_\sigma := (\ldots,{\rm log}_p (\eta^{\tau\sigma^{-1} }), \ldots)_\tau$, 
$\sigma \in G$,

\smallskip
\centerline{$\sm_{\sigma \in G} \lambda(\sigma){\rm log}_p (\eta^{\tau\sigma^{-1}}) \equiv 0\!\pmod {p^{N+1}}\ $ 
for all $\tau \in G$. }

Then there exist coefficients $\lambda'(\sigma) \in \Z_{(p)}$, not all divisible by $p$, fulfilling the relation 
$\sm_{\sigma \in G} \lambda'(\sigma){\rm log}_p (\eta^{\tau\sigma^{-1}}) 
\equiv 0 \pmod {p^{N+1}}$ for all $\tau \in G$.

Taking $\tau = 1$ yields the relation 
$\sm_{\sigma \in G} \lambda'(\sigma)\alpha_p(\eta)^{\sigma^{-1}} \equiv 0 \pmod {p}$.
\end{lemma}

\begin{proof}
Modulo $p^{N+1}$,  we can suppose that $\lambda(\sigma) \in Z_{K,(p)}$ for all $\sigma \in G$.
Here the ${\rm log}_p (\eta^{\tau\sigma^{-1}})$ are also represented, modulo $p^{N+1}$, 
by elements of $Z_{K,(p)}$ and the corresponding linear algebra is a priori over the field $K$. 

\smallskip
We obtain (for instance) ${\rm Tr}_{K/\Q}(\lambda(1)) \equiv 1\pmod p$ by multiplication of the congruence 
by a suitable $u \in K^\times$ prime to $p$ (Lemma \ref{lemm01}). 
By conjugation with $\nu \in G$ we obtain
$\sm_{\sigma \in G} \lambda(\sigma)^\nu {\rm log}_p (\eta^{\nu\tau\sigma^{-1}})
\equiv 0 \pmod {p^{N+1}}$ for all $\tau \in G$, which is equivalent to 
$\sm_{\sigma \in G} \lambda(\sigma)^{\nu} {\rm log}_p (\eta^{s \,\sigma^{-1}}) 
\equiv 0 \pmod {p^{N+1}}$ for all $s\in G$. 
Taking the trace in $K/\Q$ of the coefficients (summation over $\nu$), we obtain the rational $p$-integers 
$\lambda'(\sigma)$ for all $\sigma\in G$, with $\lambda'(1) \equiv 1 \pmod p$.
\end{proof}

We may suppose that such linear relations of dependence modulo $p^{N+1} Z_K \Z_p$, for $N \geq 1$, 
are with coefficients in $\Z_{(p)}$ because the two notions of rank coincide.
Taking the limit on $N$, one goes from the complete ring $Z_{K}\Z_p$ to the $p$-adic ring $\Z_p$.

\subsubsection{Regulators}\label{reg} 
Let $F$ be the $\Z[G]$-module generated by $\eta$.
Since $\Q_p{\rm Log}_{p}(F)$ is the $\Q_p[G]$-module generated by ${\rm Log}_{p}(\eta)$ 
and since $\prd_{v \div p} K_v$ is the representation of $G$ induced by the 
representation $K_{v_0}$ of the decomposition group $D_{{\mathfrak p}_0}$, 
the $p$-adic rank $r_p(F)$ of $F$ is equal to the $\Q_p$-rank of the system 
of vectors $(\ldots,{\rm log}_p(\eta^{\tau \sigma^{-1}}) ,\ldots)_\tau$, $\sigma \in G$, 
then to the rank (in the usual sense from the lemmas) of the classical $p$-adic 
regulator ${\mathcal R}_p(\eta)$ (or Frobenius determinant) of $\eta$

\medskip
\centerline{${\mathcal R}_p(\eta) := {\rm Frob}^G ( {\rm log}_p(\eta)) := 
\hbox{det}\big( {\rm log}_p(\eta^{\tau\sigma^{-1} }) \,\big)_{\sigma, \tau \in G}$.}

\medskip
The $\Z[G]$-module $F$ is monogenic in the framwork recalled in
\cite[\S\,1]{J}, or \cite[III.3.1.2 (ii)]{Grcft}, in which case the conjecture of 
Jaulent (\cite[\S\,2]{J}), asserts that the $p$-adic rank ${\rm rg}_p(F)$ 
of $F$ is equal to its $\Z$-rank ${\rm rg}(F) := {\rm dim}^{}_{\, \Q} (F \otimes \Q)$ 
(this is the natural extension of the Leopoldt conjecture on the group of units of $K$). 

\smallskip
We note that any minor of order $r$ is divisible by $p^r$ 
since ${\rm log}_p(\eta) \equiv -p\,\alpha_p(\eta) \pmod {p^2}$ in $\Z_p$.
 Hence the following definitions for $\eta \in K^\times$ prime to $p$:

\begin{definitions} \label{defimodif} (i) Consider (for $p>2$, unramified in $K$) the determinant
$${\rm Reg}_p^G (\eta) := {\rm Frob}^G \Big( \hbox{$\Frac{-1}{p}$}\, {\rm log}_p(\eta)\Big) := 
\hbox{det} \Big( \hbox{$\Frac{-1}{p}$}\, {\rm log}_p(\eta^{\tau\sigma^{-1} }) \Big)_{\sigma, \tau \in G}, $$

with integer coefficients of $K\Q_p$.
This Frobenius determinant is called, in all the paper, the {\it normalized $p$-adic regulator of $\eta$}.
We have ${\rm Reg}_p^G (\eta) \equiv \Delta_p^G (\eta) \pmod p$, where

\centerline{ $\Delta_p^G (\eta) := {\rm Frob}^G(\alpha_p(\eta) ) =
\hbox{det}\big( \alpha_p(\eta)^{\tau\sigma^{-1}}\big)_{\sigma, \tau \in G}$ }

\medskip  
is called the local regulator of $\eta$ (cf. \S\,\ref{sub6}).

\smallskip
(ii) For a real Galois field $K$, the usual $p$-adic regulator ${\mathcal R}_p(K)$ 
of the units is given by a minor of order $n-1$ of 
 ${\rm Frob}^G ( {\rm log}_p(\varepsilon)) = 
 \hbox{det}\big( {\rm log}_p(\varepsilon^{\tau\sigma^{-1} })\big)_{\sigma, \tau \in G}$, 
where $\varepsilon$ is a suitable Minkowski unit, and the $p$-adic integer

\smallskip
\centerline{$p^{-(n - 1)} \cdot {\mathcal R}_p(K) = \hbox{det}\Big(\hbox{$\Frac{-1}{p}$}\, 
{\rm log}_p(\varepsilon^{\tau\sigma^{-1} })\Big)_{\sigma \ne 1, \tau \ne 1}$ }

\smallskip  
is called the normalized $p$-adic regulator of $K$.
\end{definitions}

From Lemma \ref{lemm1} and after division by $p$ of the logarithms, we are reduced 
to linear algebra reasoning over $\Z/p^{N}\Z$, $N\geq 1$; in particular, ${\rm rg}_p(F)$ is 
the $\Z/p^{N}\Z$-rank of the matrix
$\big ( \frac{-1}{p}{\rm log}_p (\eta^{\tau\sigma^{-1}}) \!\pmod {p^{N}}\big )_{\sigma, \tau \in G}$, 
for $N$ large enough.

\smallskip
If a minor $M$ of order ${\rm rg}(F)$ is nonzero modulo $p^{N}$,
then it gives ${\rm rg}_p(F)$, and it is the chosen practical viewpoint that we shall limit to $N=1$,
hence to the $\alpha_p(\eta)$ modulo $p$; in this case, ${\rm rg}_p(F)$ is a priori greater or equal to
the $\Z/p\Z$-rank of the matrix $\big(\alpha_p(\eta)^{\tau\sigma^{-1}}\!\!\pmod p\big)_{\sigma, \tau \in G}$.
If ${\rm rg}(F)=n$, then the Leopoldt--Jaulent conjecture gives 
$\hbox{det}\big ( \frac{-1}{p}{\rm log}_p (\eta^{\tau\sigma^{-1}}) 
\big )_{\sigma, \tau \in G} \sim p^e$, $e \geq 0$.

\subsubsection{Strong form of the Leopoldt--Jaulent conjecture} \label{global} 
The previous local point of view (for all $p$ except a finite number) can be 
analyzed in the following two manners:

\medskip
(a) {\it Local analysis.} We make no assumption on ${\rm rg}(F)$.
If there exists in $F$ a relation $\prd_{\sigma \in G} (\eta^{\sigma^{-1}})^{\lambda(\sigma)} =\zeta$ 
(root of unity),
$\lambda(\sigma) \in \Z$ not all zero (i.e., ${\rm rg}(F)< n$), then for all $p$, prime to $\eta$, we have
$\sm_{\sigma \in G}\lambda(\sigma)\, {\rm log}_p (\eta^{\sigma^{-1}}) =0$ (i.e., ${\rm rg}_p(F)< n$).
These global relations are transmitted into the weaker local relations 
$\sm_{\sigma \in G}\lambda(\sigma)\, \alpha_p(\eta)^{\sigma^{-1}} \equiv 0 \pmod p$;
they are said to be trivial (they do not come from a numerical circumstance with coefficients 
depending on the considered prime $p$, but to the existence of a non trivial {\it global relation} 
in $F$ given by some constants $\lambda(\sigma) \in \Z$).

\smallskip
Conversely, if we have for fixed integers $\lambda(\sigma) \in \Z$, not all 
zero, the family of local conditions (for all $p$ except a finite number)

\smallskip
\centerline{\hspace{3.0cm} 
$\Big(\sm_{\sigma \in G} \lambda(\sigma)\, \alpha_p(\eta)^{\sigma^{-1}} \equiv 0 
 \pmod p\Big)_p$ , \hspace{3cm} (*)}

\smallskip
the question is to know if this is globalisable under the form 
$\prd_{\sigma \in G} (\eta^{\sigma^{-1}})^{\lambda(\sigma)} = \zeta$.

We assume only the congruences
$\sm_{\sigma \in G} \lambda(\sigma)\, \alpha_p(\eta)^{\sigma^{-1}} \equiv 0 \pmod p$
for all $p$ except a finite number, with some $\lambda(\sigma) \in \Z\,$, not all zero 
and independent of $p$. 

\smallskip
Let $\eta_0 := \prd_{\sigma \in G} (\eta^{\sigma^{-1}})^{\lambda(\sigma)} \in F$; 
then ${\rm log}_p (\eta_0) \equiv 0 \pmod {p^2}$ (i.e., ${\rm Log}_p (\eta_0) \equiv 0 \pmod {p^2}$) 
and $\eta_0$ is, in $\prd_{v \div p} K_v^\times$, of the form $\xi\,(1+ \beta\,p)^p$, 
$\beta$ $p$-integer of $\prd_{v \div p} K_v$ and $\xi$ of torsion (of prime to $p$ 
order, for $p$ large enough); so $\eta_0 \in \prd_{v \div p} K_v^{\times p}$ for almost all $p$.
Conjecturaly, $\eta_0$ is a root of unity of $K$ (from Conjecture \ref{conj31}). 

\medskip
(b) {\it Global analysis.} \label{clj} By comparison, suppose that, in a projective limit framework, 
we have coefficients $\widehat \lambda(\sigma) \in \widehat \Z = \prod_p \Z_p$, such that

\smallskip
\centerline{$\sm_{\sigma \in G} \widehat \lambda_p(\sigma) {\rm Log}_{p} (\eta^{\sigma^{-1}}) = 0$, 
for all $p$ prime to $\eta$, }

\smallskip  
where for all $\sigma \in G$, $\widehat \lambda_p(\sigma)$ 
is the $p$-component of $\widehat \lambda(\sigma)$.

Let $i := (i_v)_{v,\,v(\eta)=0}$ be the diagonal embedding $F \otimes \widehat \Z \to \widehat U$, 
where

\smallskip
\centerline{$\widehat U = \prd_{p,\, (p,\eta)=1} \Big(\prd_{ v\vert p} U_v^1\ \times 
\prd_{ v\notdiv p,\, v(\eta)=0} \mu_p(K_v) \Big)$, }

\smallskip  
$\mu_p(K_v)$ being the group of $p$th roots of unity in $K_v$ and, for $v \div p$, 
$$U_v^1 = \mu_p(K_v) \times U' , \ \ \hbox{where $U'$ is $\Z_p$-free.}$$

We put $\widehat \eta_0 := \prd_{\sigma \in G} \big (\eta^{\sigma^{-1}} \big )^{\widehat \lambda(\sigma)} 
\in F \otimes \widehat \Z$ 
and we denote by $\widehat \eta_{0,p} = \prd_{\sigma \in G} 
\big (\eta^{\sigma^{-1}} \big )^{\widehat \lambda_p(\sigma)}$ 
the $p$-component of $\widehat \eta_0$ ($p$ prime to $\eta$).
Since ${\rm Log}_{p} (\widehat \eta_{0,p})=0$ for all $p$ prime to $\eta$, we have for all place $v$ 
prime to $\eta$, $i_v(\widehat \eta_0) = \xi_v$, where (generally) $\xi_v$ is a root of unity of order a
divisor of $\ell^{n_\ell}-1$, where $\ell$ is the residue characteristic of $v$ 
(the places $v \div p$ of $K$ such that $\xi_v$ is of order divisible by $p$ are finite in number).
We can write
$$i(\widehat \eta_0) \in i (F \otimes \widehat \Z) \,\bigcap\, \prd_{p,\, (p,\eta)=1}
 \Big( \prd_{v ,\, v(\eta)=0} \mu_p(K_v) \Big). $$

By using the analogue for $F$ of the local--global characterization of the $p$-adic conjecture of 
Leopoldt--Jaulent (\cite[ \S\,2]{J}; see also \cite[III.3.6.6]{Grcft} in the case of units), we can state
(under this conjecture, same reasoning) that we have
$$i (F \otimes \widehat \Z)\ \,\bigcap\, \prd_{p,\, (p,\eta)=1} \!
\Big( \prd_{ v ,\, v(\eta)=0} \mu_p(K_v) \Big) = i (\mu (K)). $$

We deduce that $\widehat \eta_{0,p}$ is a root of unity $\zeta_p\in K$ for all $p$ prime to $\eta$. 

\smallskip
If moreover we suppose that $\widehat \lambda_p(\sigma) \equiv \lambda(\sigma) \pmod p$ 
for all $\sigma \in G$ and all $p$ prime to $\eta$, where the $\lambda(\sigma)$ are 
given rational integers, then $\eta_0 \in F^\times$, defined by 
$$\eta_0 := \prd_{\sigma \in G}\big (\eta^{\sigma^{-1}} \big)^{\lambda(\sigma)}, $$
 
is equal to $\widehat \eta_{0,p}$ up to a local $p$th power at $p$, thus
$\eta_0 \,\zeta_p^{-1} \in \prd_{v \div p} K_v^{\times p}$; we obtain the 
situation of \S\,(a) since, for all $p$ large enough, $\zeta_p = 1$
and  there is coincidence.

\smallskip
We can see our approach as a very important weakening to this classical $p$-adic context 
concerning the Leopoldt--Jaulent conjecture for all prime $p$; but as consideration, 
to have non empty information of a $p$-adic nature, we have been obliged to suppose 
the existence of the family of rational integers (not all zero)
$(\lambda(\sigma))_{\sigma \in G}$ satisfying the relation (*).

\subsubsection{General study project}\label{pg}
Our purpose, in connection with the previous $p$-adic comments,
is to see with what probability (a priori very small) the normalized regulators 
${\rm Reg}_p^G (\eta)$ of $\eta$ ($\eta$ fixed) are divisible by $p$ ($p \to \infty$).

A normalized regulator ${\rm Reg}_p^G (\eta)$ can be 
factorised by means of powers of $\chi$-regu\-lators 
${\rm Reg}_p^\chi (\eta)$ (for the irreducible rational characters $\chi$ of $G$).
This factorization does not depend on $p$. On the other hand, one can factorize 
${\rm Reg}_p^\chi (\eta)$ by means of $\theta$-components 
${\rm Reg}_p^\theta (\eta)$ (for the irreducible $p$-adic characters 
$\theta \div \chi$); this factorization depends on the residue degree 
of $p$ in the field of values of the absolutely irreducible characters 
$\varphi \div \chi$ of $G$. 
Then we shall get the congruence 
$${\rm Reg}_p^\theta (\eta) \equiv \Delta_p^\theta (\eta) \pmod p$$
 
where the local $\theta$-regulator $\Delta_p^\theta (\eta)$ is the 
$\theta$-component of $\Delta_p^G (\eta) = {\rm Frob}^G (\alpha_p(\eta) )$ (cf. \S\,\ref{sub6}).
We shall deduce a probabilistic study in order to apply the heuristic principle of Borell--Cantelli.

\begin{remark}\label{probvsdens} Lemma \ref{integer} shall allow us to reduce 
(modulo $\Q^\times$) to an $\eta \in Z_K$, 
what we suppose in numerical and Diophantine studies.
When the integer $\eta$ is fixed or varies in a small numerical neighborhood 
{\it (in an Archimedean meaning and not $p$-adic)} and when $p\to\infty$, we shall speak of {\it probability}, 
for instance ${\rm Prob} \big ({\rm Reg}_p^G (\eta)\equiv 0 \pmod p \big)$; 
on the other hand, when $p$ is fixed and $\eta$ is the variable (defined modulo~$p^2$ in our study), 
the probability coincides with the density of the numbers $\eta \in K^\times$ (prime to $p$) 
satisfying the property.

\smallskip
It is clear that densities are canonical and are computed by means of algebraic calculations.
As {\it probabilities} are linked to {\it densities}, one can confuse the two notions as soon as
they are at most $\frac{O(1)}{p^2}$, and then ``excluded'' from probabilistic considerations
of the Borell--Cantelli principle. 

\smallskip
On the other hand, in the case of densities $\frac{O(1)}{p}$, the distinction is necessary.
The idea (developed in \cite{Grqf} for ordinary Fermat quotients) is that, conjecturally, when 
$\eta$ is given, these {\it probabilities} are less than {\it densities} when $p \to \infty$ and that,
under the existence of a binomial law for ${\rm Prob}\big(\Delta_p^\theta (z) \equiv 0 \pmod p \big)$ 
($z$ running through a suitable set of residues modulo $p$), this probability is 
$\Frac{O(1)}{p^{ {\rm log}_2(p)/ {\rm log}(c_0(\eta)) -O(1) }}$ instead of $\Frac{O(1)}{p}$, 
where $c_0(\eta)={\rm max}_{\nu \in G}( \vert \eta^\nu \vert )$,
which suggests the finiteness of the number of cases (Theorem \ref{heur3}). 
\end{remark}

\subsection{Representations and group determinants (Frobenius determinants)}\label{sub2}
 We make no assumption on the Galois group $G$; for this, we
begin by a general recall in terms of representations (for a comprehensive course on representations 
and characters, see \cite{Se1}; for the Abelian case, see \cite{Wa} or \cite{C}).

\subsubsection{General notation}\label{subnot}
As $\C[G]$ is the regular representation, we have the isomorphism 
$\C[G] \simeq \bigoplus_\rho {\rm deg}(\rho) \,. \,V_\rho$, 
where $(\rho, V_\rho)$ runs through the set of absolutely irreducible 
representations of $G$ and where ${\rm deg} (\rho)$ is the degree ($\C$-dimension of~$V_\rho$). 

\smallskip
We denote by $\varphi$ the character of $\rho$; consequently, ${\rm deg}(\rho) 
= \varphi(1)$. We choose to index objects depending on $\rho$ by the letter $\varphi$ (e.g. $V_\varphi$)
and to keep $\rho = \rho_\varphi$ as a homomorphism of $G$ into ${\rm End}(V_\varphi)$.

\smallskip 
For the algebra $\C[G]$ of endomorphisms $E \in \C[G]$, acting on the basis $\{\nu,\ \nu \in G\}$ by multiplication
$\nu \mapsto E \,.\, \nu$, we have the isomorphism $\C[G] \simeq \bigoplus_\varphi {\rm End}(V_\varphi)$, with 
${\rm End}(V_\varphi) \simeq e_\varphi \C[G]$, where the $e_\varphi = \Frac{\varphi (1)}{\vert G \vert} 
\sm_{\nu \in G} \varphi (\nu^{-1} )\,\nu$ are the central orthogonal idempotents of $\C[G]$.

\smallskip
For the decomposition of $e_\varphi \C[G]$ into a direct sum of $\varphi (1)$ irreducible representations, 
isomorphic to $V_\varphi$, we use the projectors comming from a matrix representation
$M(\rho_\varphi(\nu)) = \big( a_{ij}^\varphi(\nu) \big)_{i,j}$ (\cite[ \S\,I.2.7]{Se1})

\smallskip
\centerline{$\pi_i^\varphi = \Frac{\varphi(1)}{ n } \sm_{\nu \in G} a_{ii}^\varphi(\nu^{-1}) \,\nu$, 
$\ \,i= 1, \ldots, \varphi(1)$, }

giving a system of  (non central) orthogonal idempotents such that 
$e_\varphi = \sm_i \pi_i^\varphi$.

\subsubsection{Recalls on group determinants {\rm (from \cite{C})}}\label{sub3} 
Let $G$ be a finite group and let ${\rm Frob}^G(X) = 
\hbox{det}\big( X_{\tau{\sigma^{-1}}} \big)_{\sigma, \tau \in G}$
be the determinant of the group $G$, or Frobenius determinant, 
with indeterminates $X := (X_\nu)_{\nu \in G}$.
We then have the formula
$${\rm Frob}^G(X) = \prd_\varphi \hbox{det} 
\Big( \sm_{\nu \in G} X_\nu \rho_\varphi (\nu^{-1}) \Big)^{\varphi(1)}. $$

Hence the existence of homogeneous polynomials $P^\varphi(X)$,
of degrees $\varphi(1)$, such that
$${\rm Frob}^G(X) = \prd_\varphi P^\varphi (X)^{\varphi(1)} . $$

The specialization $X_\nu \mapsto \Frac{-1}{p} {\rm log}_p(\eta^{\nu })$ 
leads to (Definitions \ref{defimodif})
$$ {\rm Reg}_p^G(\eta) := {\rm Frob}^G\Big( \hbox{ $\Frac{-1}{p} $ } {\rm log}_p(\eta ) \Big) = 
\prd_\varphi \hbox{det} \Big(\sm_{\nu \in G}\hbox{ $\Frac{-1}{p} $ } 
{\rm log}_p(\eta^{\nu })\, \rho_\varphi(\nu^{-1}) \Big)^{\varphi(1)},$$

and from ${\rm Reg}_p^\varphi(\eta) := 
P^\varphi \Big (\ldots, \Frac{-1}{p} {\rm log}_p(\eta^{\nu }) , \ldots \Big) = 
 \hbox{det} \Big(\sm_{\nu \in G}$ $\Frac{-1}{p} {\rm log}_p(\eta^{\nu })\, \rho_\varphi(\nu^{-1}) \Big )$,

\smallskip  
we group into partial products associated with the characters $\chi$ and $\theta$ irreducible over $\Q$
and $\Q_p$, respectively
$${\rm Reg}_p^\chi(\eta) = \prd_{\varphi \div \chi}{\rm Reg}_p^\varphi(\eta) \ \ \& \ \ 
{\rm Reg}_p^\theta(\eta) = \prd_{\varphi \div \theta}{\rm Reg}_p^\varphi(\eta). $$

\subsubsection{Practical calculation of the \texorpdfstring{$P^\varphi(X)$}{Lg}}\label{calp}
The polynomials $P^\varphi(X)$ are obtained in the following way:
from the vectorial space $V = \C[G]$ (provided with the basis $G$), we consider the
endomorphism of $V[X]$, $L(X) = \sm_{\nu \in G} X_\nu \,\nu^{-1}$, which is such that

\smallskip
\centerline{$\Big(\sm_{\nu \in G} X_\nu \,\nu^{-1} \Big)\,.\,\tau = \sm_{\nu \in G} X_\nu \,\nu^{-1} \tau = 
\sm_{\sigma \in G} X_{\tau \sigma^{-1}} \, \sigma, \ \ \,\forall \tau \in G$. }

\medskip
So, the determinant of this endomorphism in the basis $\{\tau, \ \tau \in G\}$ is the Frobenius determinant 
(defined up to the sign).

\smallskip
Let $(\rho_\varphi, V_\varphi)$ be the family of non isomorphic absolutely irreductible representations.
We shall take for $ {\rm End}(V_\varphi)$ the component $e_\varphi\C[G]$ associated with the
character $\varphi$.
We use the algebra isomorphism $\widetilde \rho : V \to \prd_\varphi {\rm End}(V_\varphi)$ defined by

\centerline{$\sm_{\nu \in G} a(\nu)\, \nu^{-1} \longmapsto 
\Big(\sm_{\nu \in G} a(\nu) \, \rho_\varphi (\nu^{-1})\Big)_\varphi$. }

\smallskip
where $\rho_\varphi(\nu^{-1}) = e_\varphi \,\nu^{-1}$ in the previous identification.
From the Maschke theorem, we get for the endomorphism $L(X)$

\medskip
\centerline{$\hbox{det}_V(L(X)) = \prd_\varphi \big( \hbox{det}_{V_\varphi}(L^\varphi(X)) \big)^{\varphi(1)}$, } 

\medskip  
where $L^\varphi(X) = \sm_{\nu \in G} X_\nu\,\rho_\varphi(\nu^{-1}) \in {\rm End} (V_\varphi [X])$. 
We put 
$$P^\varphi(X) := \hbox{det}_{V_\varphi}(L^\varphi (X)). $$

With a matrix realization $M(\rho_\varphi(\nu) )= \big( a^\varphi_{ij}(\nu) \big)_{i,j}$
of the $\rho_\varphi(\nu)$, the matrix associated with $L^\varphi (X)$ is
$M^\varphi(X) = \Big(\sm_{\nu \in G}a^\varphi_{ij}(\nu^{-1})\, X_\nu \Big)_{i,j}$, 
of determinant $P^\varphi(X)$.

Let $g$ be the least common multiple of the orders of the elements of $G$; it is known 
that representations are realizable over the field $C_g = \Q(\mu_g)$ of $g$th roots of unity 
(\cite[\S\,12.3]{Se1}).
So, we may suppose that the $a^\varphi_{ij}(\nu)$ are $p$-integer algebraic numbers,
for all $p$ large enough (i.e., $P^\varphi(X)\in Z_{C_g, (p)}[X]$ for all $\varphi$).

\smallskip
Let $\Gamma := {\rm Gal}(C_g/\Q)$ (commutative).
Given an absolutely irreducible representation $\rho_\varphi : G \mapsto {\rm End}_{C_g}(V_\varphi)$, 
we define its conjugates in the following Galois manner so that
for all $s \in \Gamma$, $\rho^s_\varphi$ is the representation $G \mapsto {\rm End}_{C_g}(V_{\varphi^s})
\simeq e_{\varphi^s} C_g[G]$ of character $\varphi^s$ defined by 
$\varphi^s(\nu) = (\varphi(\nu) )^s$, for all $s \in \Gamma$.
We have, for all $s \in \Gamma$, $\varphi^s(\nu) = 
\varphi(\nu^{\omega(s)} )$, where $\omega$ is the character $\Gamma \to(\Z/g\Z)^\times$ 
of the action of $\Gamma$ on $\mu_g$.
We also put $\varphi^t(\nu) := \varphi(\nu^t)$ for all integer $t$ prime to $g$ 
($\Gamma$-conjugation).

\subsubsection{Rational and \texorpdfstring{$p$}{Lg}-adic characters -- Idempotents}\label{defi01} 
We recall their practical determination.

\smallskip
 (i) {\it Rational characters}. We put, for $\varphi$ fixed

\medskip
\centerline{$\chi = \sm_{s\in {\rm Gal}(C/\Q)} \varphi^s =: \sm_{\varphi \div \chi} \varphi
 \ \ {\rm and} \ \ P^\chi (X) := \prd_{s\in {\rm Gal}(C/\Q)} P^{\varphi^s}(X) =: 
 \prd_{\varphi \div \chi}P^{\varphi^s}(X) , $}

\smallskip
 where $C \subseteq C_g$ is the field of values of any $\Q$-conjugate of $\varphi$.

\smallskip
(ii) {\it $p$-adic characters}. If $p\notdiv g$, denote, for $\chi$ fixed, by $L$
and $D$ the field and the decomposition group of $p$ in $C/\Q$. 
Let $f = \vert D \vert$ be the residue degree of $p$ in $C/\Q$ and 
$h = [L : \Q]$ the number of prime ideals ${\mathfrak p}$
above $p$ in $C$ (or $L$); thus $[C : \Q] = h\,f$.

\smallskip
Let $\varphi \div \chi$. We put
$$\theta(\nu) := \sm_{s\in D} \varphi^s(\nu) \in L, \ \, \hbox{for all $\nu \in G$}
\ \ \& \ \ P^\theta (X) := \prd_{s\in D} P^{\varphi^s}(X) =: \prd_{\varphi \div \theta} P^{\varphi}(X). $$

We fix one of the $h$ prime ideals ${\mathfrak p} \div p$ of $L$ (we shall say that $\theta$ 
and ${\mathfrak p}$ are associated).
As $L_{{\mathfrak p}^t}= \Q_p$ for all $t\in {\rm Gal}(C/\Q)/ D$, we have congruences of the form
$\theta(\nu) \equiv r_{{\mathfrak p}^t} (\nu)\pmod {{\mathfrak p}^t}$ in $L$, 
$r_{{\mathfrak p}^t} (\nu) \in \Z$; 
the rationals $r_{{\mathfrak p}^t} (\nu)$ depend numerically of the residue images 
at ${\mathfrak p}^t$ of the trace in $C/L$ of the $\varphi (\nu)$.

\smallskip
If $\theta = \sm_{s \in D} \varphi^{s}$ and ${\mathfrak p}$ are associated, the $h$ conjugates of 
$\theta$ are 
the $\theta^{t} = \sm_{s \in D} (\varphi^{t})^{s}$ and we have $\theta^t(\nu) \equiv 
r_{{\mathfrak p}^{t^{-1}}} (\nu)\pmod {\mathfrak p}$
(or $\theta^{t^{-1}} (\nu)\equiv r_{{\mathfrak p}^t} (\nu) \pmod {\mathfrak p}$).
As the $\theta^t$ are seen in $\Z_p \subset L_{\mathfrak p}$, 
we shall write by abuse $\theta^t(\nu) \equiv r_{{\mathfrak p}^{t^{-1}}} (\nu)\pmod p$.

\smallskip
For $p$ fixed, the integer $f$ depends only on $\chi$ and is called the residue degree 
of the characters $\varphi, \theta$ and $\chi$. 
We have, by $\Gamma$-conjugation, $\varphi^{p^i}(\nu) = \varphi(\nu^{p^i}) = \varphi(\nu)^{s_p^i}$, 
where $s_p$ is the Frobenius automorphism (of order $f$) in $C/\Q$.

\medskip
(iii) {\it Idempotents.} 
We put $e_\chi = \sm_{\varphi \div \chi}e_\varphi$ and $e_\theta = 
\sm_{\varphi \div \theta}e_\varphi$; thus
$e_\chi = \sm_{\theta \div \chi} e_\theta$. The $e_\theta$ (resp. $e_\chi$) 
give a fundamental system of orthogonal idempotents
of $\Q_p[G]$ (resp. $\Q[G]$). We can replace $\Q_p$ (resp. $\Q$) by $\Z_p$ 
(resp. $\Z_{(p)}$) because $p \notdiv g$.

\medskip
From $P^\varphi(X) = \hbox{det}_{V_\varphi}(L^\varphi(X))$ we deduce that 
$P^{\varphi^s}(X) = 
\hbox{det}_{V_{\varphi^s}}(L^{\varphi^s} (X))$ where 
$L^{\varphi^s}(X) = \sm_{\nu \in G} X_\nu\,\rho_\varphi^s(\nu^{-1})$ 
is given via the $(a^\varphi_{ij} (\nu^{-1}))^s$, which defines the conjugate 
by $s$ of the polynomial $P^\varphi(X)$ (i.e., of its coefficients).

\begin{theorem}\label{lemm2} (i) For all $p$ large enough, the polynomials 
$P^\chi (X)$ (resp. $P^\theta (X)$) have rational $p$-integer coefficients 
(resp. $p$-adic integer coefficients). 

\smallskip
(ii) For all irreducible character $\varphi$, we have $P^\varphi (\ldots,X_{\pi \nu}, \ldots) = 
\zeta_\pi\, P^\varphi (\ldots,X_{\nu}, \ldots)$ for all $\pi \in G$, where $\zeta_\pi^g = 1$.
\end{theorem}

\begin{proof} (i) As $P^\varphi(X)\in Z_{C,(p)}[X]$ for all $\varphi \div \chi$, 
$P^\chi (X)$ $ = \!\!\! \prd_{s \in {\rm Gal}(C/\Q)} P^{\varphi^s} (X)$ is invariant by Galois.
Likewise $P^\theta (X) = \prd_{s \in D} P^{\varphi^s} (X) \in L[X] 
\subset L_{\mathfrak p}[X] = \Q_p[X]$.

\smallskip
(ii) For $\pi \in G$ call $[\pi]$ the operator defined by $[\pi] X_\nu = X_{(\pi\nu)}$
for all $\nu \in G$. Then $[\pi]$ and $\widetilde \rho : V[X] \to \prod_\varphi {\rm End}(V_\varphi[X])$ 
commute; moreover, since $\rho_\varphi$ is a homomorphism, we have the following formula
$$[\pi]\Big (\sm_{\nu \in G} X_\nu \,\rho_\varphi(\nu^{-1}) \Big)= 
\sm_{\nu \in G} X_{\pi\nu} \,\rho_\varphi(\nu^{-1}) =
\Big (\sm_{\nu \in G} X_\nu \,\rho_\varphi(\nu^{-1})\Big )\,\rho_\varphi (\pi) . $$

Then, since the determinant of $\rho_\varphi(\pi) \in {\rm End} (V_\varphi)$
is that of a diagonal matrix whose diagonal is formed of roots of unity, we get
$$ \hbox{det} \Big([\pi]\Big (\sm_{\nu \in G} X_\nu \rho_\varphi(\nu^{-1})\Big ) \Big) = 
\zeta_\pi \, \hbox{det} \Big(\sm_{\nu \in G} X_\nu \rho_\varphi(\nu^{-1})\Big), $$

where $\zeta_\pi$ is of order a divisor of the order of $\rho_\varphi(\pi)$ which is a divisor of $g$.
\end{proof}

\begin{corollary} \label{coro3} 
For all $\pi \in G$ and all absolutely irreducible character $\varphi$, 
we have $P^\varphi (\ldots, \alpha^{\pi\nu }, \ldots) = 
\zeta_\pi\, P^\varphi (\ldots, \alpha^\nu, \ldots)$ by the specialization 
$X_\nu \mapsto \alpha^\nu$, $\alpha \in Z_K$.

\smallskip
Consequentely, $P^\chi (\ldots, \alpha^{\pi\nu}, \ldots) = \pm P^\chi (\ldots, \alpha^\nu, \ldots)$ for all 
$\pi \in G$.\,\footnote{\,Sign $+$ except if $\chi = \varphi $ is quadratic and $\varphi (\pi) = -1$.}

\smallskip
In the same way, $P^\theta (\ldots, \alpha^{\pi\nu}, \ldots) = 
\zeta'_\pi P^\theta (\ldots, \alpha^\nu, \ldots)$ for all $\pi \in G$,
where $\zeta'_\pi$ is of order a divisor of g.c.d.\,$(g, p-1)$. 
\end{corollary}

\subsubsection{Numerical determinants}\label{sub111}
In this section, there is no reference to a prime number $p$ and the characters that we consider 
are absolutely irreducible or rational. 
The above leads to define the numerical $\chi$-determinants of Frobenius of any $\alpha \in Z_K$ 
(i.e., independent of the given $\eta \in K^\times$).
 
\begin{definition} \label{defi4} Let $G$ be a finite group and let ${\rm Frob}^G(X)$ be the 
associated group determinant.
The $\chi$-determinants (with indeterminates and numerical) are by definition the expressions
$$\hbox{${\rm Frob}^\chi (X)= \prd_{\varphi \div \chi} P^\varphi (X) \ \ {\rm and } \ \
{\rm Frob}^\chi (\alpha) = \prd_{\varphi \div \chi} P^\varphi (\ldots, \alpha^\nu,\ldots)$,}$$
so that ${\rm Frob}^G(\alpha) = \prd_\chi \, ({\rm Frob}^\chi (\alpha))^{\varphi(1)}$ 
(where $\varphi \div \chi$ for each $\chi$).
\end{definition}
 
\begin{example}\label{ex5} {\rm In the case of the group 
$D_6 = \{1, \sigma, \sigma^2, \tau, \tau\sigma, \tau\sigma^2 \}$, 
we have the following numerical $\chi$-determinants

\footnotesize
\smallskip
${\rm Frob}^1(\alpha) = \alpha +\alpha^{\sigma} +\alpha^{\sigma^2} +
\alpha^{\tau} +\alpha^{\tau\sigma} +\alpha^{\tau\sigma^2}$,

\smallskip
${\rm Frob}^{\chi_1}(\alpha) = \alpha +\alpha^{\sigma} +\alpha^{\sigma^2} 
-\alpha^{\tau} -\alpha^{\tau\sigma} -\alpha^{\tau\sigma^2}$,

\smallskip
${\rm Frob}^{\chi_2}(\alpha) = \alpha^{2} +\alpha^{2}{}^{\sigma} 
+\alpha^{2}{}^{\sigma^2} -\alpha^{2}{}^{\tau} -\alpha^{2}{}^{\tau\sigma} 
-\alpha^{2}{}^{\tau\sigma^2} - \alpha \alpha^{\sigma} 
- \alpha^{\sigma} \alpha^{\sigma^2} -\alpha^{\sigma^2} \alpha$

\hfill $+ \alpha^{\tau} \alpha^{\tau\sigma}+ \alpha^{\tau\sigma} \alpha^{\tau\sigma^2} +
\alpha^{\tau\sigma^2} \alpha^{\tau}$.

\normalsize
\smallskip
The two last one are of the form ${\rm Frob}'\,.\,\sqrt m$, ${\rm Frob}' \in \Q$, 
where $k = \Q(\sqrt m)$ is the quadratic subfield of $K$ and we neglect the factor 
$\sqrt m$; but ${\rm Frob}^{\chi_2}(\alpha)$ appears to the square in the 
determinant ${\rm Frob}^G(\alpha)$ and the result is rational, which is 
not the case of ${\rm Frob}^{\chi_1}(\alpha)$. This is specific of quadratic characters. }
\end{example}

For computations, we can return to the matrix realizations ($C = \Q$, $\varphi = \chi_2$)
\footnotesize
\begin{align*}
\rho_\varphi(1) &= \left(\begin{matrix} 
1&0\\
0&1\end{matrix}\right), \ \ \ \ \ \ \rho_\varphi(\sigma) = \left(\begin{matrix} 
-1&-1\\
1&0\end{matrix}\right), \ \ \rho_\varphi(\sigma^2) = \left(\begin{matrix} 
0&1\\
-1&-1\end{matrix}\right), \\
\rho_\varphi(\tau) &= \left(\begin{matrix} 
1&0\\
-1&-1\end{matrix}\right), \ \ \rho_\varphi(\tau\sigma) = \left(\begin{matrix} 
-1&-1\\
0&1\end{matrix}\right),\ \ \rho_\varphi(\tau\sigma^2) = \left(\begin{matrix} 
0&1\\
1&0\end{matrix}\right), 
\end{align*}

\normalsize
which leads (by specialization and by taking the determinant) to
\footnotesize
\begin{align*}
\sm_{\nu \in G}X_\nu \rho_\varphi(\nu^{-1}) &= \left(\begin{matrix} 
X_1-X_{\sigma^2} +X_{\tau}-X_{\tau\sigma} & X_{\sigma} 
- X_{\sigma^2}-X_{\tau\sigma} +X_{\tau\sigma^2} \\
-X_{\sigma} + X_{\sigma^2}-X_{\tau} + X_{\tau\sigma^2} & 
X_1-X_{\sigma} -X_{\tau} + X_{\tau\sigma} \end{matrix}\right), \\
{\rm Frob}^{\chi_2}(\alpha) &= \left \vert\begin{matrix} 
\alpha-\alpha^{\sigma^2} +\alpha^{\tau}-\alpha^{\tau\sigma} & 
\alpha^{\sigma} - \alpha^{\sigma^2}-\alpha^{\tau\sigma} +\alpha^{\tau\sigma^2} \\
-\alpha^{\sigma} + \alpha^{\sigma^2}-\alpha^{\tau} + \alpha^{\tau\sigma^2} & 
\alpha-\alpha^{\sigma} -\alpha^{\tau} + \alpha^{\tau\sigma} \end{matrix}\right \vert. 
\end{align*}

\normalsize
Still for $\chi_2$ (of degree 2) and the representation 
$e_{\chi_2}\,\Q[G] \simeq2\,V_\varphi$, there exist two 
orthogonal projectors $\pi_1, \pi_2$, of sum $e_{\chi_2} = 
\frac{1}{3}(2 - \sigma - \sigma^2)$ (\S\,\ref{subnot}), which yields here 
$$\ \pi_1 = \Frac{1}{3}(1-\sigma^2 + \tau - \tau\sigma)\ \ \& 
\ \ \pi_2 = \Frac{1}{3}(1-\sigma - \tau + \tau\sigma).$$

\subsection{The local \texorpdfstring{$\theta$}{Lg}-regulators}\label{sub6}
Let $\eta \in K^\times$ be given and let $p$ be large enough so that $p$
is unramified in $K$, prime to $n = [K : \Q]$ and $\eta$.

\subsubsection{Generalities}\label{gene1}
We fix an algebraic integer $\alpha \in Z_K$ defined by $\alpha \equiv \alpha_p(\eta) \pmod p$.
We obtain the determinant, with coefficients in $Z_K$, defined modulo $p$
$$\Delta^G_p(\eta) := {\rm Frob}^G(\alpha) = 
\hbox{det} \big( \alpha^{\tau\sigma^{-1}}\big)_{\sigma, \tau \in G} =
\prd_\chi \ \prd_{\theta \div \chi} \ 
\prd_{\varphi \div \theta} P^{\varphi} (\ldots, \alpha^\nu,\ldots )^{\varphi(1)} . $$

If $\Delta^G_p(\eta) \notin \Q$, we find again the existence of a factor $\sqrt m$
which comes from the resolvant of a quadratic character of $G$ and that we neglect
in the definitions of regulators.

\begin{definition}\label{defi10}
For all $p$ large enough and for each $\Q_p$-irreducible character $\theta$ of $G$, we call
local $\theta$-regulator of $\eta$, the $p$-adic integer defined by

\medskip
\centerline{$\Delta_p^\theta(\eta) := 
\prd_{\varphi \div \theta} P^{\varphi} (\ldots, \alpha^\nu,\ldots )$, 
for $\alpha \equiv \alpha_p(\eta) := \frac{1}{p} (\eta^{p^{n_p}-1} -1)\pmod p$. }

\smallskip
For $\theta \div \chi$ ($\chi$ fixed), the corresponding local $\theta$-regulators depend on
the splitting of $p$ in $C/\Q$ and there are $h = \frac{[C : \Q]}{f}$ such regulators, where $f$ is 
their residue degree (\S\,\ref{defi01}\,(ii)). These regulators are only defined modulo $p$.
\end{definition}

\begin{remark}\label{rema110} In the same manner, we may write (for $p$ large enough) 
that the normalized regulator ${\rm Reg}_p^G(\eta)$ is equal to 
$$\prd_\chi {\rm Reg}_p^\chi(\eta)^{\varphi(1)} = \prd_\theta 
{\rm Reg}_p^\theta(\eta)^{\varphi(1)}, $$

where 
$${\rm Reg}_p^\theta(\eta) = 
\prd_{\varphi \div \theta} P^{\varphi}\big (\ldots, \hbox{$\frac{-1}{p}$} {\rm log}_p(\eta^\nu) ,\ldots\big ). $$

We then have the congruences 
$${\rm Reg}_p^\theta(\eta) \equiv \Delta_p^\theta(\eta)\pmod p; $$ 

so $p$ divides ${\rm Reg}_p^\theta(\eta)$ if and only if $\Delta_p^\theta(\eta) 
\equiv 0 \pmod p$; in this case, there exists $e \geq 1$ such that 
$p^{e \varphi(1)}$ divides ${\rm Reg}_p^\theta(\eta)$
 where at each time $\varphi\div \theta$ (\S\,\ref{sub3}). 

\smallskip
We shall speak of an extra $p$-divisibility if $e \geq 2$.
\end{remark}

\subsubsection{Particular remarks}\label{rema11} 
(i) We have 
$$\Delta_p^\chi (\eta) :=\No_{C/\Q}\big( P^{\varphi} (\ldots, \alpha^\nu ,\ldots ) \big ) \in \Z
 \ \,\hbox{($\varphi \div \chi$ fixed), }$$
 
with the convention on the notation $\No_{C/\Q}$, 
especially when $K$ and $C$ are not lineairely disjoint. We recall the quadratic exception for $\chi$.

\smallskip
In the same way
$$\Delta_p^\theta (\eta) :=\No_{\mathfrak p}\big ( P^{\varphi} (\ldots, \alpha^\nu ,\ldots ) \big ), $$

where for ${\mathfrak p} \div p$ in $C$, ${\mathfrak p}$ associated with $\theta$, $\No_{\mathfrak p}$ 
denotes the absolute local norm (issued from $\No_{C/L}$) in the completion of $C$ at ${\mathfrak p}$; 
we find again $\Delta_p^\chi (\eta)$ as a product of the correspondent local norms at $p$.

\smallskip
Same normic relations by replacing $\Delta_p$ by ${\rm Reg}_p$ and $\alpha$ by 
$\frac{-1}{p} {\rm log}_p(\eta)$.

\smallskip
(ii) If $H = \{\nu \in G,\, \varphi(\nu) = \varphi(1) \}$ is the kernel of $\varphi \div \theta \div \chi$ 
(which only depends on~$\chi$) 
and if $K'$ is the subfield of $K$ fixed by $H$, we have 
$$\Delta_p^\theta (\eta)= \Delta_p^{\theta'} (\No_{K/K'}(\eta) )$$ 

where $\theta'$ is the faithful character resulting from $\theta$. 
By replacing $\eta$ by $\eta' := \No_{K/K'}(\eta)$ one always can 
suppose that $\theta$ is a faithful character.

\subsubsection{Characters \texorpdfstring{$\chi$}{Lg} of degree 1, of order 1 or 2}\label{sub7}
Let $\eta \in K^\times$ and let $\alpha \equiv \alpha_p(\eta) \pmod p$, $\alpha \in Z_K$.

\smallskip
 (i) If $\chi =\theta =1$, the $\theta$-regulator corresponds to 
 $\No_{K/\Q}(\eta) = a \in \Q^\times$ and is given by 
${\rm Tr}_{K/\Q}(\alpha)$, in other words
$$\Delta^1_p (\eta) \equiv \Frac{-1}{p} {\rm log}_p(a) \equiv \Frac{1}{p}(a^{p-1} - 1)
 \equiv q_p(a) \pmod p$$ 
 
(Fermat quotient of $a$); for classical properties and use of Fermat quotients, 
see, e.g., \cite{EM}, \cite{GM}, \cite{Grqf}, \cite{Hat}, \cite{H-B}, \cite{KR}, \cite{OS}, \cite{Si}.

\smallskip
For $a = 659$ and $p \leq 10^9$, we only find the solutions $p = 23$, $131$, $2221$, $9161$, $65983$.
 See \cite[Pr. A-1]{Grpro}. For $a = 47$ and $a = 72$, we find no solution for $p \leq 10^{11}$.

\medskip
(ii) If $\chi=\theta$ is quadratic and if $k= \Q(\sqrt m\,)$ is the quadratic 
subfield of $K$ fixed by the kernel of $\chi$,
we obtain a $\theta$-regulator corresponding to the case 
$\No_{K/k}(\eta) \in k^\times\,\!\setminus\!\Q^\times$; 
if ${\rm Tr}_{K/k}(\alpha) =: u + v\,\sqrt m \in k$, it is given by
$$\Delta^\theta_p (\eta) \equiv (1-\tau) ( u + v\,\sqrt m)\equiv 2v\,\sqrt m \pmod p . $$

 If $K$ is a real quadratic field with the fundamental unit $\varepsilon$, 
 because of the multiplicative relation of 
dependence $\varepsilon^{1+\sigma} =\pm 1$, the $1$-regulators $\Delta_p^1(\varepsilon)$
are trivialy zero modulo $p$. The $\theta$-regulator of the quadratic character is
$\Delta^\theta_p(\varepsilon) \equiv 2\,v\,\sqrt m \pmod p$ 
(computed via $\varepsilon^{p^{n_p}-1} \equiv 1 +p\, v\,\sqrt m \pmod{p^2}$). 

\smallskip
We compute the $\theta$-regulator $\Delta^\theta_p(\varepsilon)$
of the fondamental unit $\ \varepsilon = 5+2\,\sqrt 6$, for all $p \leq 10^9$ ($p \ne 2, 3$) (see \cite[Pr. A-2]{Grpro}
valuable for any quadratic integer).
We find a $\theta$-regulator equal to zero modulo $p$ only for $p= 7, 523$, which gives a second observation 
on the rarity of the phenomenon.

\smallskip
Let $\eta = 1 + \sqrt 6$ of norm $-5$. We have ${\rm rg}(F)=2$
(no trivial nullities). We verifiy that Fermat quotients $\Delta^1_p(\eta)$ of $-5$ 
are all nonzero modulo $p$ in the tested interval.
The solutions for $\Delta^\theta_p(\eta) \equiv 0 \pmod p$, $\theta \ne 1$, are $p=11, 37, 163, 4219$.
For $\eta=1+5 \sqrt{-1}$ of norm $26$, we find for $\theta\ne 1$ the two solutions $p = 73, 12021953$.
For the golden number $\Frac{1+ \sqrt 5}{2}$ we find no solution in the tested interval.

\subsubsection{Criterion of trivial nullity for local \texorpdfstring{$\chi$}{Lg}-regulators}\label{sub5}
 Let $\eta \in K^\times$ and let $F$ be the $\Z[G]$-module generated by $\eta$.

\begin{remark}\label{rema51}
In the decomposition ${\rm Frob}^G(\alpha) = \prd_\chi {\rm Frob}^\chi(\alpha)^{\varphi(1)}$, when 
$\alpha \equiv \alpha_p(\eta) \pmod p$, some of the local $\chi$-regulators $\Delta_p^\chi(\eta)$ are zero
modulo $p$ as soon as there exists a non trivial global multiplicative relation of the form 

\medskip
\centerline {$\prd_{\nu \in G}(\eta^{\nu^{-1}})^{\lambda(\nu)} = 1$, $\ \lambda(\nu) \in \Z$, }

\smallskip  
which yields $\sm_{\nu \in G} \lambda (\nu) \,\alpha^{\nu^{-1}} \equiv 0 \pmod p$ for all $p$ prime to $\eta$.
\end{remark}

\begin{lemma} \label{nul} If we have ${\rm dim}_\Q ((F \otimes \Q)^{e_\chi} )< {\rm dim}_\Q 
(e_\chi \Q[G]) = [C : \Q]\, \varphi(1)^2$ 
(i.e., there exists $U\in \Q[G]$ such that $\eta^{U_\chi} = 1$, with $U_\chi := e_\chi U\ne 0$), then
the local $\chi$-regulators $\Delta_p^\chi(\eta) :={\rm Frob}^\chi (\alpha)$ are zero modulo $p$ for 
all $p$ large enough (they are said trivialy null modulo $p$).
\end{lemma}

This implies the trivial nullity modulo $p$ of certain $\Delta_p^\theta(\eta)$, $\theta \div \chi$, 
namely those for which $U_\theta := e_\theta U \not\equiv 0 \pmod p$; for the proof, see the 
Lemmas of \S\,\ref{sub9} (criterion of nullity modulo $p$ of $\Delta_p^\theta(\eta)$).

\begin{remarks}\label{rema511}
(i) If $\varphi(1) = 1$, $\Delta_p^\chi(\eta)$ trivialy null modulo $p$ is equivalent to $\eta^{e_\chi} = 1$ 
(i.e., $U_\chi = e_\chi$), in which case $\Delta_p^\theta(\eta) \equiv 0 \!\!\pmod p$ 
trivialy for all $\theta \div \chi$.

\smallskip
(ii) If $\varphi(1) > 1$, $\Delta_p^\chi(\eta)$ is trivialy null modulo $p$ if there exists $i$, 
$1 \leq i \leq \varphi(1)$, such that, in $F \otimes C$, we have $\eta^{\pi_i^\varphi} = 1$ 
for $\varphi \div \chi$ (\S\,\ref{subnot}).

\smallskip
For instance, for $G = D_6$ and $\varphi =\chi= \chi_2$, the elements

\medskip
\centerline{$\pi_1^\varphi = \frac{1}{3}(1-\sigma^2 + \tau -\tau\sigma)$ and 
$\pi_2^\varphi = \frac{1}{3}(1-\sigma - \tau +\tau\sigma)$ }

\medskip  
are such that $e_\chi \pi_i^\varphi = \pi_i^\varphi$, $i=1, 2$, 
$\pi_1^\varphi+\pi_2^\varphi = e_\chi$, and $\pi_1^\varphi\pi_2^\varphi = 0$ (cf. Example \ref{ex5}). 

\medskip
So we may have the non trivial $\varphi$-relation $\eta^{U_1} := \eta^{1-\sigma^2 + \tau -\tau\sigma} = 1$
while $\eta^{U_2} := \eta^{1-\sigma - \tau +\tau\sigma} \ne 1$ 
(i.e., ${\rm dim}_\Q (F \otimes \Q)^{e_\chi}=2$ for ${\rm dim}_\Q (e_\chi \Q[G])=4$)); we would have
$\eta^{e_\chi .\,(U_1+U_2)} = \eta^{3 e_\chi} =\eta^{3 U_2} \ne 1$, but we verify that the $\chi$-regulator
$\Delta_p^\chi(\eta)$ is equal to zero modulo $p$ because of the first relation.

\medskip
To suppose ${\rm rg}(F)=n$ avoids this disadvantage. We can always suppose it by multiplying $\eta$ by 
a suitable $\eta'$ in such a way that $(FF') \otimes \Q \simeq \Q[G]$ and $F\cap F' =1$ (obvious notation).

\smallskip
(iii) For $U \in \Z_{(p)}[G]$, we have $U_\chi = \sm_{\varphi \div \chi} U_\varphi$ and 
$U_\varphi = e_\varphi U_\chi$.
We have $U_\chi \equiv 0 \pmod p$ if and only if $U_\varphi \equiv 0 \pmod p$ for at least a (donc all)
$\varphi \div \chi$ (because the $\varphi \div \chi$ are conjugate by ${\rm Gal}(C/\Q)$).

\smallskip
These congruences (mod $p$) in the group algebras mean (depending on the case)
$$\hbox{$\pmod{p \,\Z_{(p)}[G]}\ \ $ or $\pmod{p \,Z_{C,(p)} [G]}$} $$ 

where $Z_{C,(p)}$ is the ring of $p$-integers of the field of values $C$ of the $\varphi \div \chi$.

\smallskip
This does not occur for $U_\chi = \sm_{\theta\div \chi} U_\theta$ and $U_\theta = e_\theta U_\chi$ 
because $U_\theta \equiv 0 \pmod p$ in $\Z_p[G]$ means $U_\theta \equiv 0 \pmod {\mathfrak p}$
in $L[G]$ (for $\theta$ and ${\mathfrak p}$ associated), which is only equivalent to
$U_\varphi \equiv 0 \pmod {\mathfrak p}$ for all $\varphi \div \theta$ (\S\,\ref{defi01}\,(ii)).
\end{remarks}

\begin{examples}\label{ex1} {\rm
a) $G=C_n$. Let $G$ cyclic of order $n$ and let $\chi$ of order $d \div n$; then
the elements $\eta \in K^\times$ such that $\eta^{e_\chi} = 1$ 
correspond to the trivial nullity $\!\!\!\pmod p$ of
$$\Delta_p^\chi(\eta) = \No_{C /\Q} \big( \sm_{\nu\in G} \varphi (\nu^{-1})\,\alpha^{\nu} \big).$$
For $n = 3$ (for which $C=\Q(j)$, where $j^3=1$, $j\ne 1$), we have the two rational idempotents
$$e_1= \hbox{$\frac{1}{3}$} (1+ \sigma + \sigma^2), \ \ e_\chi= \hbox{$\frac{1}{3}$} (2-\sigma- \sigma^2) . $$

(i) The $\eta \in K^\times$ such that $\eta^{e_1} = 1$
(i.e., of norm 1 in $F\otimes\Q$), correspond to the trivial nullity of
$\Delta_p^1 (\eta) = \alpha + \alpha^\sigma+\alpha^{\sigma^2}$.

\smallskip
(ii) The $\eta \in K^\times$ such that $\eta^{e_\chi} = 1$ or $\No_{K/\Q}(\eta) = \eta^3$, 
hence such that $\eta \in \Q^\times$, correspond to the trivial nullity of
$\Delta_p^\chi(\eta) = \No_{\Q(j)/\Q}(\alpha +j^2\,\alpha^{\sigma} + j\,\alpha^{\sigma^{2}})$.

\medskip
b) $G = D_6$. The three idempotents for the group $D_6$ are
\begin{align*}
 e_1 & = \hbox{$\frac{1}{6}$} (1+ \sigma + \sigma^2 + \tau + \tau \sigma +\tau \sigma^2), \\
 e_{\chi_1}& = \hbox{$\frac{1}{6}$} (1+ \sigma + \sigma^2 - (\tau + \tau \sigma +\tau \sigma^2)), \\
 e_{\chi_2}&= \hbox{$\frac{1}{6}$} (2 - \sigma - \sigma^2). 
\end{align*}

(i) The $\eta$ such that $\eta^{e_1} = 1$ correspond to the trivial nullity of
$\Delta_p^1 (\eta) = {\rm Tr}_{K/\Q} (\alpha)$.

\smallskip
(ii) The $\eta$ such that $\eta^{e_{\chi_1} }= 1$ are such that $\No_{K/k}(\eta) \in\Q^\times$, where $k$ is the quadratic subfield of $K$, and correspond to the trivial nullity of

\smallskip
$\Delta_p^{\chi_1} (\eta) = \alpha +\alpha^{\sigma} +\alpha^{\sigma^2} -\alpha^{\tau} 
-\alpha^{\tau\sigma} -\alpha^{\tau\sigma^2} = (1-\tau)\, {\rm Tr}_{K/k} (\alpha)$.

\smallskip
(iii) The $\eta$ such that $\eta^{U_{\chi_2}}=1$ for $U_{\chi_2} \in e_{\chi_2} \Q[G] \setminus\! \{0\}$
lead to the trivial nullity~of

\smallskip
$\Delta_p^{\chi_2} (\eta) = \alpha^{2} +\alpha^{2}{}^{\sigma} +\alpha^{2}{}^{\sigma^2} 
-\alpha^{2}{}^{\tau} -\alpha^{2}{}^{\tau\sigma} -\alpha^{2}{}^{\tau\sigma^2} 
- \alpha \alpha^{\sigma} - \alpha^{\sigma} \alpha^{\sigma^2} -\alpha^{\sigma^2}\alpha$

\hfill$ + \alpha^{\tau} \alpha^{\tau\sigma} + \alpha^{\tau\sigma} \alpha^{\tau\sigma^2} 
+\alpha^{\tau\sigma^2} \alpha^{\tau}$. }
\end{examples}

\section{\texorpdfstring{$\F_p$}{Lg}-linear relations between the conjugates 
of \texorpdfstring{$\alpha$}{Lg}}\label{section3}

Let $\eta \in K^\times$ be fixed and let $p$ be a large enough prime number. 

\smallskip
Let $\alpha_p(\eta) := \frac{1}{p}\big( \eta^{p^{n_p} - 1}-1 \big)\in Z_{K,(p)}$.
We intend to establish the relation between the nullity modulo $p$ of certain 
$\Delta_p^\theta(\eta)$ and the existence of certain $\F_p$-linear relations 
between the conjugates of $\alpha_p(\eta)$ modulo~$p$.
We implicitely suppose ${\rm rg}(F)=n$. First, let us establish elementary generalities:

\subsection{\texorpdfstring{$\F_p$}{Lg}-independence}\label{sub8}
Let $\alpha \in K$, arbitrary (so $\alpha \in Z_{K,(p)}$ for all $p$ large enough).
 We shall say that the $\alpha^\nu$, $\nu \in G$, are $\F_p$-independent if, 
 for all family of coefficients $u(\nu) \in \Z_{(p)}$, the congruence 
 $\sm_{\nu\in G} u(\nu) \, \alpha^{\nu^{-1}} \equiv 0 \pmod p$, in $Z_{K,(p)}$, 
 implies $u(\nu) \equiv 0 \pmod p$ for all $\nu \in G$. 
 
 \smallskip
 We then have the following result where we recall that
${\rm Frob}^G (\alpha)= \hbox{det}\big(\alpha^{\tau \sigma^{-1}} \big)_{\sigma, \tau \in G}$:

\begin{proposition}\label{prop13} Let $\alpha \in K$ be given. We assume $p$ large enough 
in such a way that $\alpha \in Z_{K,(p)}$ and $p$ does not divide the discriminant of $K$.

\smallskip
 (i) The $\alpha^\nu$ are $\F_p$-independent if and only if 
 $\alpha$ is a normal $\Z_{(p)}$-basis of $Z_{K,(p)}$.

(ii) The $\alpha^\nu$ are $\F_p$-independent if and only if 
${\rm Frob}^G(\alpha) $ is prime to $p$.
\end{proposition}

\begin{proof}
(i) If $\alpha$ is a normal $\Z_{(p)}$-basis of $Z_{K,(p)}$, any congruence
$\sm_{\nu\in G} u(\nu) \, \alpha^{\nu^{-1}} \equiv 0 \pmod p$, $u(\nu) \in \Z_{(p)}$,
leads to $u(\nu) \equiv 0 \pmod p$, for all $\nu \in G$.

\smallskip
Assume now that the $\alpha^\nu$ are $\F_p$-independent and that there exists a non-trivial relation 
of $\Q$-linear dependence between the conjugates of $\alpha$;

 it follows a relation of the form
$\sm_{\nu \in G}{r(\nu)}\alpha^{\nu^{-1}}\!\! = 0$ with integers $r(\nu)$, not all zero, such that 
${\rm p.g.c.d. }(r(\nu))_\nu = 1$; whence $r(\nu) \equiv 0 \pmod p$ for all $\nu \in G$ (absurd). 
Consequentely $\alpha$ is yet a normal $\Q$-basis of $K$. 
If $\beta \in Z_{K,(p)} \setminus \{0\}$, there exist some $r(\nu) \in \Z$, not all zero, 
and an integer $d$, prime to 
${\rm p.g.c.d. }(r(\nu))_\nu$, such that $d\, \beta = \sm_{\nu \in G}{r(\nu)}\alpha^{\nu^{-1}}$. 
We have $p\notdiv d$ otherwise the $r(\nu)$ should be divisible by $p$. 
Thus $\alpha$ is a normal $\Z_{(p)}$-basis of $Z_{K,(p)}$.

\medskip
(ii) Suppose that the $\alpha^\nu$ are $\F_p$-independent; as $\alpha = \frac{1}{d} \beta$,
$\beta \in Z_K\setminus p Z_K$, $d\in \Z \setminus p \Z$, one can return to the integer case for $\alpha$.
As $p$ is large enough, it does not divide the discriminant of $K/\Q$, and the discriminant of the normal 
$\Z_{(p)}$-basis $\alpha$, of $Z_{K,(p)}$, is prime to $p$ (indeed, the conductor ${\mathfrak f} \in \Z$ 
such that ${\mathfrak f} \, Z_K \subseteq \bigoplus_\nu \Z\,\alpha^\nu$ is not divisible by $p$ and the two 
discriminants coincide up to a $p$-adic unit).
But the discriminant of the normal basis $\alpha$ is the square of the Frobenius determinant 
${\rm Frob}^G(\alpha) = \hbox{det}\big(\alpha^{\tau \sigma^{-1}} \big)_{\sigma, \tau \in G}$.

Suppose ${\rm Frob}^G(\alpha)$ prime to $p$, and suppose there exist some $\lambda(\sigma) \in \Z_{(p)}$, 
not all divisible by $p$, such that $\sm_{\sigma \in G} \lambda(\sigma)\, \alpha^{\sigma^{-1}} \equiv 0 \pmod p$. 
By conjugation by $\tau \in G$, we obtain a $\Z_{(p)}$-linear relation over the lines of the form

\smallskip
\centerline{$\sm_{\sigma \in G} \lambda(\sigma) (\ldots, \alpha^{\tau \sigma^{-1}} , \ldots)_\tau \equiv
 (\ldots, 0 , \ldots)_\tau \pmod p$, }

\smallskip  
whence ${\rm Frob}^G(\alpha) \equiv 0 \pmod p$ (absurd).
\end{proof}

\begin{corollary}\label{coro14} If for $p$ large enough at least one of the local 
$\theta$-regulators $\Delta_p^\theta(\eta)$ is zero modulo $p$, then the 
$\alpha_p(\eta)^\nu$ are not $\F_p$-independent and there exists a $\F_p$-linear 
relation $\sm_{\nu\in G} u(\nu) \, \alpha_p(\eta)^{\nu^{-1}} 
\equiv 0 \pmod p$, with $u(\nu) \in \Z_{(p)}$ not all divisible by $p$.
\end{corollary}

\subsection{Criterion of nullity modulo \texorpdfstring{$p$}{Lg} of the 
\texorpdfstring{$\Delta_p^\theta(\eta)$}{Lg}}\label{sub9}
We refer to \S\,\ref{defi01} using the decomposition field $L$ of $p$ in $C/\Q$ and $D = {\rm Gal}(C/L)$. 
To simplify, we suppose $K \cap C = \Q$. We recall that $Z_{C,(p)}$ is the ring of $p$-integers of $C$. 

\subsubsection{Main lemmas}\label{sub155}
Let $\eta \in K^\times$ be such that the multiplicative $\Z[G]$-module generated by $\eta$ is of $\Z$-rank $n$.
We fix $\alpha \equiv \alpha_p(\eta) \pmod p$ in $Z_K$. As usual, $\varphi$ denotes an absolutely
irreducible character and $\theta$ an irreducible $p$-adic character.

\begin{definition}\label{defi11} (i) If $\sm_{\nu \in G} u(\nu)\,\alpha^{\nu^{-1}} \equiv 0 \pmod p$, 
$u(\nu) \in \Z_{(p)}$ for all $\nu \in G$, we call {\it associated relation with} $\alpha$ the element 
$$U = \sm_{\nu \in G} u(\nu)\,{\nu^{-1}} \in \Z_{(p)}[G], $$ 

and we define, for characters $\varphi$ and $\theta$,
the $\varphi$-relations $U_\varphi := e_\varphi .\,U\in Z_{C,(p)}[G]$, and the $\theta$-relations 
$U_\theta := e_\theta .\,U\in Z_{L,(p)}[G]$.

\medskip
(ii) We denote by ${\mathcal L}$ the $G$-module of the relations $U \in \Z_{(p)} [G]$ 
(defined modulo $p\,\Z_{(p)}[G]$), associated with $\alpha$.
 Seen in $\F_p [G]$, we have ${\mathcal L} = \{0\}$ if and only if
the $\alpha^\nu$ are $\F_p$-independent (\S\,\ref{sub8}) and we have 
${\mathcal L} = \F_p [G]$ if and only if $\alpha \equiv 0 \pmod p$.

\medskip
(iii) For $\theta$, we denote by ${\mathcal L}^\theta \simeq \delta \, V_\theta$ the 
$\theta$-component $e_\theta\,{\mathcal L}$, where $V_\theta$ (of $\F_p$-dimension $f \varphi(1)$) 
is the irreducible representation of character $\theta$; then $0\leq \delta \leq \varphi(1)$. 

\smallskip
(iv) Let ${\mathfrak p}\div p$ the prime ideal of $L$ associated with $\theta$. Thus
$\theta(\nu)=\sm_{s\in D} \varphi^s (\nu) \in Z_{L,(p)}$ is defined via 
$\theta(\nu) \equiv r_{\mathfrak p} (\nu) \pmod {\mathfrak p}$, 
$r_{\mathfrak p} (\nu) \in \Z$; if $U\in \Z_{(p)}[G]$, $U_\theta\in Z_{L,(p)}[G]$ 
is congruent modulo ${\mathfrak p}$ to an element of $\Z_{(p)}[G]$.
We shall view $U_\theta$ in $\Z_{p}[G] \!\pmod p$ or in $Z_{L,(p)}[G] \!\pmod {\mathfrak p}$ 
depending on the context (see Remark \ref{rema511}).
\end{definition}

Let $U = \sm_{\nu\in G} u(\nu) \,{\nu^{-1}} \in \Z_{(p)}[G]$; then $U_\varphi = 
\sm_{\nu \in G} u_\varphi(\nu)\,\nu^{-1} 
\in Z_{C,(p)}[G]$, with $u_\varphi(\nu) = 
\frac{\varphi(1)}{ n } \sm_{\tau\in G} \varphi(\tau^{-1}) u(\nu \tau)$. 
We then have $U_\theta = \sm_{\varphi \div \theta}U_\varphi$.

\begin{lemma}\label{lemm15} If $U = \sm_{\nu\in G} u(\nu)\,\nu^{-1} \in {\mathcal L}$, then 
$U_\varphi.\alpha := \sm_{\nu\in G} u_\varphi(\nu)\,\alpha^{\nu^{-1}} \equiv 0 \! \pmod p$ for all 
irreducible character $\varphi$. 
\end{lemma}

\begin{proof} We have
$U_\varphi\,.\, \alpha = \frac{\varphi(1)}{ n } 
\sm_{\tau\in G} \,\varphi(\tau^{-1}) \Big( \sm_{\sigma\in G} u(\sigma) 
\alpha^{\tau\sigma^{-1}} \Big) \equiv 0 \pmod p$, by conjugation by 
$\tau$ of $\sm_{\sigma\in G} u(\sigma) \alpha^{\sigma^{-1}} \equiv 0 \pmod p$.
\end{proof}

\begin{lemma}\label{lemm16} Let $U \in {\mathcal L}$, let ${\mathfrak p}$ be associated with $\theta$,
and let $\varphi \div \theta$ be such that $U_\varphi \not\equiv 0 \!\pmod {\mathfrak p}$ 
(condition independent of the choice of $\varphi \div \theta$).
Then the endomorphism $E_\varphi := e_\varphi \sm_{\nu\in G} \alpha^\nu \nu^{-1}$ of
${\rm End}_{KC} \!(V_\varphi)$ is not invertible modulo ${\mathfrak p}$.
\end{lemma}

\begin{proof} Let us work by transposition of endomorphisms 
(which does not change determinants). We have
\begin{align*}
 U_{\varphi} \,.\, E_\varphi &= e_\varphi \sm_{\nu\in G} U_{\varphi}\, \alpha^\nu\, \nu^{-1}
 = e_\varphi \sm_{\nu\in G} \alpha^\nu \, \sm_{\sigma\in G} u_{\varphi} (\sigma)\sigma^{-1} \nu^{-1} \\
& =e_\varphi \sm_{\tau\in G} \Big( \sm_{\nu\in G} u_{\varphi}(\nu^{-1}\tau) \,\alpha^\nu\Big)\tau^{-1} 
= e_\varphi \sm_{\tau\in G} \big( U_{\varphi} \,.\,\alpha \big)^{\tau} \tau^{-1}\equiv 0 \pmod p,
\end{align*}
from Lemma \ref{lemm15} above.
\end{proof}

As $E_\varphi$ is an endomorphism of $V_\varphi$ over $KC$, for the prime ideal 
${\mathfrak p} \div p$ of $C$ such that $U_\varphi \not\equiv 0 \pmod {\mathfrak p}$, 
there exists a prime ideal ${\mathfrak P}\div {\mathfrak p}$ 
of $KC $ for which $\hbox{det}(E_\varphi)\equiv 0\! \pmod {\mathfrak P}$. 
But any conjugation by $\tau \in G$ gives 
$$E_\varphi^\tau = e_\varphi \sm_{\nu\in G} \alpha^{\tau\nu} \, \nu^{-1} = 
e_\varphi \sm_{\nu\in G} \alpha^{\nu}\, \nu^{-1} \, .\, (e_\varphi \tau)= 
E_\varphi\, \circ\, e_\varphi \tau, $$

and we obtain
$\hbox{det}(E_\varphi^\tau) = \hbox{det}(E_\varphi)\, \hbox{det}(e_\varphi \tau) 
\equiv 0 \pmod {\mathfrak P^\tau}$,
whence $\hbox{det}(E_\varphi) \equiv 0 \pmod {\prod_{\tau \in G} \mathfrak P^\tau}$ 
since the $ \hbox{det}(e_\varphi \tau)$ are invertible.

\smallskip
Since $\hbox{det}(E_\varphi) \equiv 0 \pmod {\mathfrak p}$ (extended to $KC$), this yields
$P^{\varphi} (\ldots, \alpha^\nu, \ldots) \equiv 0 \pmod {\mathfrak p}$ which may be written
$\Delta_p^\varphi(\eta) \equiv 0 \pmod {\mathfrak p}$.
Since $\Delta_p^\theta(\eta)$ is the local norm at ${\mathfrak p}$ of $\Delta_p^\varphi(\eta)$, we get:

\begin{corollary}\label{coro161} If $U_\varphi \not\equiv 0 \pmod {\mathfrak p}$, we have
$\Delta_p^\theta(\eta) \equiv 0 \pmod {{\mathfrak p}^f}$ (or modulo $p^f$ in $L_{\mathfrak p} = \Q_p$) 
for the $p$-adic character $\theta$ above $\varphi$ associated with ${\mathfrak p}$.
\end{corollary}

\begin{lemma}\label{lemm161} Reciprocally, if 
$E_\varphi := e_\varphi \sm_{\nu\in G} \alpha^{\nu} \, \nu^{-1} \in {\rm End}_{KC}(V_\varphi)$ 
is not invertible modulo ${\mathfrak p}$, there exists a nonzero $\varphi$-relation modulo 
${\mathfrak p}$ of the form $W = \sm_{\sigma \in G} w(\sigma) \sigma^{-1}$ in 
$e_\varphi \, Z_{C,(p)}[G]$, such that $W.\, \alpha \equiv 0 \pmod {\mathfrak p}$.
\end{lemma}

\begin{proof} Lemma \ref{lemm1} allowing $Z_{C,(p)}$-linear reasoning, there exists 
$W \in e_\varphi Z_{C,(p)} [G]$ such that $W \not\equiv 0 \pmod {\mathfrak p}$ is in the 
kernel of the transposed of $E_\varphi$, which may be written 
$W \,.\, E_\varphi \,.\,\equiv 0 \pmod {\mathfrak P}$ 
for ${\mathfrak P} \div {\mathfrak p}$ in $KC$.

\smallskip
The relation $E_\varphi^\tau = E_\varphi\,\circ\, e_\varphi \tau$ and the fact that 
$W$ is with coefficients in $Z_{C,(p)}$
 shows, by conjugations, that the congruence occurs modulo ${\mathfrak p}$ (extended).

\smallskip
Put $W = \sm_{\sigma \in G} w (\sigma) \sigma^{-1}$, $w (\sigma) \in Z_{C,(p)}$ for all $\sigma \in G$;
the congruence $W \, .\, E_\varphi \equiv 0 \pmod {\mathfrak p}$ may be written successively 
(since $ e_\varphi W = W$)
\begin{align*}
&\sm_{\nu \in G} \sm_{\sigma \in G} w (\sigma) \alpha^\nu \sigma^{-1} \,\nu^{-1} \equiv 
\sm_{\sigma \in G} w (\sigma) \sm_{t \in G} \alpha^{t^{-1} \sigma^{-1}} \,t \equiv 0 \pmod {\mathfrak p}, \\
&\sm_{t \in G}\big(\sm_{\sigma \in G} w(\sigma) \alpha^{t^{-1}\sigma^{-1}} \big)\,t \equiv 0 \pmod {\mathfrak p}; 
\end{align*}

so $\sm_{\sigma \in G} w(\sigma) \alpha^{t^{-1}\sigma^{-1}} \!\!\equiv 0\pmod {\mathfrak p}$, for all $t \in G$, 
whence $\sm_{\sigma \in G} w(\sigma) \alpha^{\sigma^{-1}} \!\!\equiv 0 \! \pmod {\mathfrak p}$, 
giving the non-trivial associated $\varphi$-relation modulo ${\mathfrak p}$
$$W = \sm_{\sigma \in G} w(\sigma)\,\sigma^{-1} \in e_\varphi Z_{C,(p)}[G], $$

such that $W.\, \alpha \equiv 0 \pmod {\mathfrak p}$  (but $W$ is not necessarily in
$e_\varphi \Z_{(p)}[G]$).
\end{proof}

\begin{lemma}\label{integer} In the study of the $\Delta_p^\theta(\eta)$, $\theta \ne 1$, 
one may suppose $\eta \in Z_K$.
\end{lemma}

\begin{proof} Put $\eta=\mu . d^{-1}$, $\mu \in Z_K$, $d \in \Z$. We have 
$\alpha_p(\eta) \equiv \alpha_p(\mu)-\alpha_p(d)$ $\pmod p$, and 
$\sm_{\nu \in G} u(\nu) \alpha_p(\eta)^{\nu^{-1}} \equiv 
\sm_{\nu \in G} u(\nu) \alpha_p(\mu)^{\nu^{-1}} \pmod p$,
for all $\theta$-relations relative to $\eta$,
because $\alpha_p(d)$ is invariant by Galois and $\theta \ne 1$; whence 
${\mathcal L}^\theta(\eta)= {\mathcal L}^\theta(\mu)$
and $\Delta_p^{\theta}(\eta)$ \& $\Delta_p^{\theta}(\mu)$ null (or not) at the 
same time (Theorem \ref{theo24} below).
\end{proof}

Then we shall suppose $\eta \in Z_K$ for certain Diophantine reasoning 
(essentially in Sections \ref{section6}, \ref{section7}), 
but we can keep $\eta \in K^\times$ in general statements.

\subsubsection{Main statement}\label{sub15}
The technical results of \S\,\ref{sub155} lead to the following:

\begin{theorem}\label{theo24} Let $K/\Q$ be a Galois extension of degree $n$ of Galois group $G$.
Let $\eta \in K^\times$ be such that the multiplicative $\Z[G]$-module generated by $\eta$ is of $\Z$-rank $n$. 

\smallskip  
For any unramified prime $p>2$, prime to $n$ and $\eta$, put $\eta_1 := \eta^{p^{n_p} - 1} = 
1 + p\,\alpha_p(\eta)$, $\alpha_p(\eta) \in Z_{K,(p)}$, where $n_p$ is the residue degree of $p$ in $K/\Q$. 

\smallskip
Let ${\mathcal L}$ be the $G$-module of relations 
$U = \sm_{\nu \in G} u(\nu)\, \nu^{-1} \in \Z_{(p)} [G]$ regarding $\alpha_p(\eta)$, i.e., such that,
by definition,

\medskip
\centerline{$\sm_{\nu \in G} u(\nu)\, \alpha_p(\eta)^{\nu^{-1}} \equiv 0 \pmod p, \ \,u(\nu) \in \Z_{(p)} \ 
\hbox{(Definitions \ref{defi11})}. $}

\smallskip  
Let $\theta$ be an irreducible $p$-adic character of $G$ and let $f$ be the residue degree of 
$p$ in the field of values of the absolutely irreducible characters $\varphi \div \theta$. 

\smallskip
Then, seen in $\F_p[G]$, the $G$-module ${\mathcal L}^\theta := e_\theta {\mathcal L}$ is of nonzero $\F_p$-dimension if and only if the local $\theta$-regulator $\Delta_p^\theta(\eta)$ (\S\,\ref{sub6}) is zero modulo $p$. 

\smallskip
When it is the case, the $\F_p$-dimension of ${\mathcal L}^\theta$ is $\delta f \varphi(1)$, 
with $1 \leq \delta \leq \varphi(1)$.
\end{theorem}

\begin{proof}
 (a) If ${\mathcal L}^\theta \ne \{0\}$, there exists $U=\sm_{\nu\in G} u(\nu) \,\nu^{-1} \in {\mathcal L}$
 such that $U_\theta \not\equiv 0 \pmod p$; then $U_\varphi \not\equiv 0 \pmod {\mathfrak p}$ for all 
 $\varphi \div \theta$. 
From Lemma \ref{lemm16} and Corollary \ref{coro161}, we have $\Delta_p^\theta(\eta) \equiv 0 \pmod p$.

\smallskip
(b) Suppose $\Delta_p^\theta(\eta) \equiv 0 \pmod p$ and let $\alpha \equiv \alpha_p(\eta) \pmod p$,
$\alpha\in Z_{K,(p)}$; by the resulting nullity modulo $p$ of ${\rm Frob}^G(\alpha)$,
 there exists a relation of $\F_p$-dependence of the form 
 $\sm_{\nu \in G} u(\nu)\, \alpha^{\nu^{-1}} \equiv 0 \pmod p$, 
$u(\nu) \in \Z_{(p)}$ not all divisible by $p$, and we have 
$U = \sm_{\nu \in G} u(\nu)\, \nu^{-1} \in {\mathcal L}$
(Corollary \ref{coro14}), but we need to deduce that ${\mathcal L}^\theta \ne \{0\}$.
From Lemma \ref{lemm161}, there exists, for $\varphi \div \theta$, a non trivial
$\varphi$-relation modulo ${\mathfrak p}$ of the form $W := \sm_{\nu \in G} w(\nu)\nu^{-1}$, 
$w(\nu) \in Z_{C,(p)}$, such that $W.\,\alpha \equiv 0\!\! \pmod {\mathfrak p}$. 

If $\{z, \ldots, z^{f}\}$ is a $Z_{L,(p)}$-basis of $Z_{C,(p)}$, then 
$w(\nu) = \sm_{i =1,\ldots,f} a_i (\nu) z^i$, with $a_i(\nu) \in Z_{L,(p)}$ for all $i$ and all $\nu$, whence 
$\sm_{\nu \in G} \sm_{i =1,\ldots,f} a_i(\nu) z^i\, \alpha^{\nu^{-1}} \equiv 0 \! \pmod {\mathfrak p}$;
identifying on the basis of the $z^i$ one obtains the system of relations in $Z_{KL,(p)}$

\medskip
\centerline {$\sm_{\nu \in G} a_i(\nu) \alpha^{\nu^{-1}} \equiv 0 \pmod {\mathfrak p}, \ i = 1, \ldots, f. $}

\smallskip  
For all $i$, and all $\nu$, there exist some $r_{\mathfrak p}^i (\nu) \in \Z$ such that $ a_i(\nu) \equiv 
r_{\mathfrak p}^i (\nu) \pmod {\mathfrak p}$, whence
$\sm_{\nu \in G} a_i(\nu) \alpha^{\nu^{-1}}\equiv \sm_{\nu \in G}r_{\mathfrak p}^i (\nu) \alpha^{\nu^{-1}}
\equiv 0 \pmod {\mathfrak p}$; since $\sm_{\nu \in G}r_{\mathfrak p}^i (\nu) \alpha^{\nu^{-1}}$ is in 
$K$, this yields $\sm_{\nu \in G}r_{\mathfrak p}^i (\nu) \alpha^{\nu^{-1}}\equiv 0 \pmod p$.
Since $W$ is a non trivial $\varphi$-relation modulo ${\mathfrak p}$, the $r_{\mathfrak p}^i (\nu)$ 
are not all zero modulo $p$ and there exists a non trivial relation 
$\sm_{\nu \in G}r_{\mathfrak p}^i (\nu)\,\alpha^{\nu^{-1}}$ for at least an index $i \in \{1,\ldots,f\}$.
As $W$ is a $\varphi$-relation, this is transmitted to $\sm_{\nu \in G} a_i(\nu)\,\alpha^{\nu^{-1}}$ 
and consequently, $\sm_{\nu \in G} r_{\mathfrak p}^i (\nu) \, {\nu^{-1}}$ 
(a $\varphi$-relation invariant by $D$), is a non trivial $\theta$-relation of ${\mathcal L}$.

\smallskip
In fact one can prove that the matrix $\big ( r_{\mathfrak p}^i (\nu) \big)_{i,\nu}$ is of rank $f$.
\end{proof}

\begin{corollary}\label{coro25} When ${\mathcal L}^\theta \ne \{0\}$, we get local lifts of the form 
$\eta^{U_\theta} \in \prd_{v \div p} K_v^{\times p}$ for all $\theta$-relation 
$U_\theta \in {\mathcal L}^\theta$. If we represent, modulo $p$,  
$U_\theta \in \Z_p[G]$ by $U'_\theta \in \Z[G]$, then $\eta_1^{U_\theta}$ 
is a global element of $K^\times$ being a local $p$th pover at $p$.
\end{corollary}

\begin{proof}
We have $\eta_1^{U_\theta} \!= (1+p\,\alpha_p(\eta))^{U_\theta} \equiv 
1 + p\, U_\theta \cdot \alpha_p(\eta)\! \pmod {p^2}$ and, since by definition 
$U_\theta\, \cdot\,\alpha_p(\eta) \equiv 0 \pmod p$, this yields
$\eta_1^{U_\theta} = 1 +p^2 \beta$, $\beta \in Z_{K,(p)}$.
Thus $\eta_1^{U_\theta} = (1+p\,\gamma)^p$, $\gamma \in \prd_{v \div p} K_v$, and
$\eta = \eta^{p^{n_p}} \eta_1^{-1}$ implies $\eta^{U_\theta} \in \prd_{v \div p} K_v^{\times p}$.
\end{proof}

\section{Heuristic considerations and experiments}\label{section4}

\subsection{Probabilistic methods}\label{sub11} 
If some events $E_p$, indexed by the prime numbers, are independent and of probabilities ${\rm Pr} (E_p)$, 
we may apply the heuristic principle of Borel--Cantelli that is to say: if the series $\sum_p {\rm Pr} (E_p)$ 
is convergent, then the natural conjecture is that the events $E_p$ are realized finitely many times, and that 
if it is divergent they are realized infinitely many times with a suitable density (see \cite[Chap.\,III.1]{T}). 
In our case, $E_p$ is, for $\eta \in K^\times$ fixed, the events
$$\hbox{ ``${\rm Reg}_p^G(\eta) \equiv 0 \!\pmod p$'' \ or \  ``$\Delta_p^\theta (\eta) \equiv 0 \pmod p$'' } $$

for a choice of $\theta$ for each $p$ (\S\,\ref{sub6}).
In the general case, since ${\rm Reg}_p^\theta(\eta)$ is a local norm in the extension $C/\Q$, 
such a local regulator is either prime to $p$, either divisible by $p^f$, where $f$ is the residue degree of $p$ 
in this extension; similarly, if the irreducible character $\varphi \div \theta$ is of degree $\varphi (1) \geq 2$, 
${\rm Reg}_p^G(\eta)$ is divisible by $p^{f\varphi (1)}$.

\smallskip
We shall see that the degree $\varphi (1)$ does not occur for probabilities but that, on the contrary, 
the number $\delta$ 
such that ${\mathcal L}^\theta \simeq \delta V_\theta$ occurs, as well as $f$, 
under the formula $\Frac{O(1)}{p^{f \delta^2}}$ 
which is the probability to have ``$\, \Delta_p^\theta (\eta) \equiv 0 \pmod p\ $ \& 
$\ {\mathcal L}^\theta \simeq \delta V_\theta\, $'' (\S\,\ref{HP}). 

\smallskip
 We shall neglect primes $p$ for which at least two $\theta$-regulators $\Delta_p^\theta(\eta)$
are divisible by $p$, such a probability being at most $\Frac{O(1)}{p^2}$, given the independence of the
local $\theta$-regulators (\S\,\ref{sub12}).

\smallskip
It will remain the case $\Delta^\theta_p(\eta) \equiv 0 \pmod p$ for a unique $p$-adic 
character $\theta$ of $G$ under the conditions $f=1$ and the representation 
${\mathcal L}^\theta$ being minimal (i.e., $\delta =1$); then we will have 
${\rm Reg}_p^G(\eta) \sim p^{e \varphi(1)}$ with $e = 1$, the case $e \geq 2$ 
being also of probability at most $\Frac{O(1)}{p^2}$ (\S\,\ref{subex}).

\smallskip
 The obstruction for the utilisation of the heuristic principle of Borel--Can\-telli
would come from primes $p$ satisfying the following definition:

\begin{definition}\label{defidec} A prime number $p$ constitutes a case of {\it minimal $p$-divisibility}
for the normalized regulator ${\rm Reg}_p^G (\eta)$ if ${\mathcal L}^\theta \ne 0$ 
(i.e., $\Delta^\theta_p(\eta) \equiv 0 \pmod p$) for a unique irreducible $p$-adic character 
$\theta$ of $G$ satisfying furthermore the following conditions 

\smallskip
(i) $p$ is totaly split in $C$ (i.e., $f=1$),

\smallskip
(ii) ${\mathcal L}^\theta \simeq V_\theta$ (i.e., $\delta=1$),

\smallskip
(iii) ${\rm Reg}_p^\theta (\eta) \sim p$ 
(i.e., ${\rm Reg}_p^G (\eta) \sim p^{\varphi(1)}$ has no extra $p$-divisibilities).
\end{definition}

If $G$ is Abelian, this concerns certain $p \equiv 1 \pmod d$, where $d$ is the order of $\varphi \div \theta$.

\smallskip
If $G=1$ (situation of the Fermat quotient of a rational), this occurs for all $p$.

\subsection{Principles of analysis -- linearization of the problem}\label{sub13}
Let $\eta \in Z_K$ be given such that the multiplicative $G$-module generated by 
$\eta$ is of $\Z$-rank $n$, even if any case of sub-representation may be studied 
in an analogous way.

\subsubsection{Densities vs probabilities}\label{sub14}
We can verify by experiments the following heuristic principles using the function {\it random} of PARI 
to define an arbitrary integer $\gamma$ of $K$, prime to $p$ 
(in fact we are only interested by the class modulo $p^2$ of $\gamma$):

\medskip
(i) If under a $p$-adic point of view, $\alpha_p(\gamma) \pmod p$ runs through 
the quotient ring $Z_{K,(p)}/(p) \simeq \F_p^{\,n}$,
experiments show that the statistical result remains excellent if one limits 
$\gamma$ into a small {\it Archimedean} domain (defined for instance by 
$\vert c_i\vert \ll p$ for the components $c_i$ of $\gamma$ on a basis, or by 
${\rm max}_{\nu\in G}(\vert \gamma^\nu \vert) \ll p$), 
which preserves the Diophantine aspect and proves 
an uniform distribution (required limitation when $p^n$ is very large). 
In \cite{H-B} it is proved the uniform distribution of Fermat quotients 
and it is easy to conjecturer that this is general.

\smallskip
As explained in Remark \ref{probvsdens}, we must distinguish the notion of probability ($\gamma$ fixed 
and $p \to \infty$) from that of density, purely algebaic, when they are equal to $\Frac{O(1)}{p}$; 
we establish Sections \ref{section6} and \ref{section7} the analogue of the study 
conducted in \cite{Grqf} for Fermat quotients (with numerical verifications for the groups $C_3, D_6$), 
which constitutes a serious justification of the conjectures of Section \ref{section8}.

(ii) Let $(e_i)_{i=1,\ldots,n}$ be a $\Z_{(p)}$-basis of $Z_{K,(p)}$ 
and put $\alpha_p(\gamma) = \sm_{i=1}^n A_i e_i$, $A_i \in \Z_{(p)}$; 
then, modulo $p$, the variables $A_i$ are independent and equiprobable in $\F_p$, 
and this does not depend on $K$ nor of the choice of the basis. 

Any non trivial relation of the form $\sm_{\nu \in G} u(\nu)\,\alpha_p(\gamma)^{\nu^{-1}}
\equiv 0 \pmod p$ is translated into an analogous non trivial relation on the $A_i$
(because the conjugates of the $e_j$ are linear forms on the $e_i$, independent of $p$).

\subsubsection{Main Heuristic.} \label{HP} The probability (comming from the 
corresponding density) of $\Delta_p^\theta (\eta)\equiv 0 \pmod p$ is that of 
${\mathcal L}^\theta \ne \{0\}$ (Definitions \ref{defi11}, Theorem \ref{theo24}).
If ${\mathcal L}^\theta \simeq \delta V_\theta$, $\delta \ne 0$, 
we shall justify that we must assign to this case the probability
$${\rm Prob}\Big( \hbox{${\mathcal L}^\theta \simeq \delta V_\theta$, 
$\ 1 \leq \delta \leq \varphi(1)$} \Big) \leq \dsfrac{O(1)}{p^{f \delta^2}}, $$

where $f$ is the residue degree of $\theta$, where we consider $V_\theta$ 
as a $\F_p$-representation and then, by extension of scalars, 
$V_\theta \otimes \F_{p^f} $ and $V_\varphi $ as $\F_{p^f}$-representations.

\smallskip
Indeed, we have ${\mathcal L}^\theta \otimes \F_{p^f} = 
\bigoplus_{\varphi \div \theta} {\mathcal L}^\varphi$, where 
${\mathcal L}^\varphi \simeq \delta V_\varphi$, and 
the idea comes from the fact that when ${\mathcal L}^\varphi \simeq \varphi(1) V_\varphi \simeq 
e_\varphi \F_{p^f}[G]$ (i.e., $e_\varphi \alpha_p(\eta) \equiv 0 \pmod p$), the correspondent probability is 
$\Frac{O(1)}{p^{f\varphi(1)^2}}$ (minimal) since $e_\varphi \alpha_p(\eta)$ is defined by $f \varphi(1)^2$ 
$\F_p$-independent components ($\F_p$-dimension of $\varphi(1) V_\varphi$).
But $e_\varphi \F_{p^f}[G] \simeq {\rm End}(V_\varphi)$ as an algebra of endomorphisms 
of a $\F_{p^f}$-space of dimension $\varphi(1)$.

\smallskip
Therefore, ${\mathcal L}^\varphi \simeq \delta V_\varphi$ is then seen as a sub-algebra of 
endomorphisms of a $\F_{p^f}$-space of dimension $\delta$, whence a probability-density 
$\Frac{O(1)}{p^{f\delta^2}}$ to get ${\mathcal L}^\varphi \simeq \delta V_\varphi$ 
(i.e., ${\mathcal L}^\theta \simeq \delta V_\theta$). 
The case $f = \delta =1$ establishes the case where the notion of 
probability must be substituated for that of density.

\smallskip
We shall note that the probability to have all
 the $\Delta_p^\theta (\eta)\equiv 0 \!\!\pmod p$ with each time $\delta=\varphi(1)$ 
 (i.e., $\alpha_p(\eta)\equiv 0 \pmod p$, equivalent for the $n$ components of 
 $\alpha_p(\eta)$ to be zero modulo $p$) is then $\Frac{O(1)}{p^n}$ since 
 $\sm_\theta f \varphi(1)^2 = \vert G \vert=n$. 
 This shows the consistency of the proposed heuristic.
 
\smallskip
The most frequent non trivial case is $\delta =1$ (the residue degree $f$ 
depends canonically of $p$ contrary to $\delta $ which is ``numerical''). 
For instance, from $\delta =1$ to $\delta =2$ (for $f=1$), this increases the 
probabilities from $\Frac{O(1)}{p}$ to $\Frac{O(1)}{p^{4}}$, near from zero for $p\to\infty$ 
(very well confirmed by numerical statistics, cf. \S\,\ref{subD6}).

\begin{example}\label{ex11}{\rm 
Case of $G=D_6$ ($f=1$, $1 \leq \delta \leq 2$). Let $\theta$ be the 
irreducible character of degree 2; the representation $e_\theta \F_p[G]$ is isomorphic to $2\,V_\theta$ 
where $V_\theta$ is of $\F_p$-dimension 2. 
On may generate $e_\theta \F_p[G]$ as follows (see Remark \ref{rema511}\,(ii))
\footnotesize
\begin{align*}
 U_1 &= 1-\sigma^2 + \tau -\tau\sigma, \ \ \ \sigma U_1 = \sigma-1 + \tau\sigma^2 -\tau , \ \ \ \ \,
\sigma^2 U_1 = -U_1- \sigma U_1, \\
 U_2 &= 1-\sigma - \tau +\tau\sigma, \ \ \ \sigma U_2 = -\sigma^2+\sigma + \tau -\tau\sigma^2, \ \ \ 
 \sigma^2 U_2 = -U_2- \sigma U_2,
\end{align*}
\centerline{ $\tau U_1 = -\sigma U_1$, $\ \tau\sigma U_1 = - U_1$, $\ \tau\sigma^2 U_1 = -\sigma^2 U_1$, }

\medskip
\centerline{ $\tau U_2 = - U_2$, $\ \tau\sigma U_2 = - \sigma^2 U_2$, $\ \tau\sigma^2 U_2 = -\sigma U_2$. }

\normalsize
\medskip  
The elements $U_1, \sigma U_1, U_2, \sigma U_2$ constitute a $\F_p$-basis of the 
space of the $\theta$-relations, which justifies the probability $\Frac{O(1)}{p}$ 
only for the case $\delta = 1$, but $\Frac{O(1)}{p^4}$ for $\delta = 2$. }
\end{example}

\subsection{Probabilistic independence (over \texorpdfstring{$\theta$}{Lg}) 
of the variables \texorpdfstring{$\Delta^\theta_p(\gamma)$}{Lg}}\label{sub12}
We process the case of the group $D_6$, by use of the {\it random} function, to verify two aspects: 

\smallskip
(i) The independence of the $\theta$-regulators (probability at most $\Frac{O(1)}{p^2}$ 
to get two $\theta$-regulators $\Delta^\theta_p(\gamma)$ and $\Delta^{\theta'}_p(\gamma)$ 
null modulo $p$, for $\theta \ne \theta'$).

\smallskip
(ii) the probability $\Frac{O(1)}{p}$ to have the nullity modulo $p$ of 
$\Delta_p^{\theta}(\gamma)$ for the character 
$\theta = \chi_2$ of degree 2, the case of characters of 
degree 1 being analogous.

\medskip
We consider the field $K$ (compositum of $\Q(\sqrt[3] 2)$ and of $\Q(j)$, 
where $j$ denotes a cubic root of unity) defined by the polynomial 
$$Q=x^6+9x^4-4x^3+27x^2+36x+31. $$

We take at random $\gamma$ modulo $p^2$, prime to $p$, which gives some $\alpha = \alpha_p(\gamma)$ uniformly distributed modulo $p$. 
The Pr. A-3 of \cite{Grpro} compute the conjugates of $\alpha$ on the basis 
$\{x^5, x^4, x^3, x^2, x, 1\}$.
The variable $N_0$ is the number of $\gamma$ prime to $p$.
The variables $N_1, N_2, N_3, N_{12}, N_{13} , N_{23} , N_{123}$ give the number of cases of 
simultaneous nullities of 1, 2 or 3 regulators (characters $\chi_0$, $\chi_1$, $\chi_2$ of degree 2, respectively). 

\medskip
For $p = 13$ we obtain the following values

\smallskip
$N_0= 999115$ ; $N_1= 76820$ ; $N_2= 77009$ ; $N_3= 82239$ ; 

\smallskip
$N_{12 }= 5898$ ; $N_{13 }= 6301$ ; $N_{23 }= 6453$ ; $N_{123 }= 442$, and the respective densities

\smallskip
$\frac{N_1}{N_0} = 0.076888$ ; $\frac{N_2}{N_0} = 0.07707 $ ; $\frac{N_3}{N_0} = 0.0823$ ; 

\smallskip
$\frac{N_{12}}{N_0} = 0.00590$ ; $\frac{N_{13}}{N_0} = 0.006306$ ; $\frac{N_{23}}{N_0} = 0.006458$ ; 
$\frac{N_{123}}{N_0} = 0.0004424$ ; 

\smallskip  
with $\frac{1}{p} = 0.07692$, $\frac{1}{p^2} = 0.005917$, $\frac{1}{p^3} = 0.000455$, 
whence the expected probabilities.

\medskip
For $p = 37$ we obtain the following values 

\smallskip
$N_0= 999952$ ; $N_1= 27153$ ; $N_2= 27054$ ; $N_3= 27747$ ; 

\smallskip
$N_{12 }= 718$ ; $N_{13 }= 761$ ; $N_{23 }= 755$ ; $N_{123 }= 16$, and the respective densities

\smallskip
$\frac{N_1}{N_0} = 0.0271543$ ; $\frac{N_2}{N_0} = 0.027055$ ; $\frac{N_3}{N_0} = 0.0277483$ ; 

\smallskip
$\frac{N_{12}}{N_0} = 0.000718$ ; $\frac{N_{13}}{N_0} = 0.000761$ ; $\frac{N_{23}}{N_0} = 0.000755$ ; 
$\frac{N_{123}}{N_0} = 1.600 \times 10^{-5}$,

\smallskip  
with $\frac{1}{p} = 0.027027$, $\frac{1}{p^2} = 0.00073046$, $\frac{1}{p^3} = 1.97 \times 10^{-5}$.

\subsection{Statistics on the matrix rank of the components}\label{sub112}

A first statistic experiment consists in determining the probability to have 
at least a non trivial relation between the conjugates
of $\alpha$; if $\alpha^\nu = \sm_{i=1}^n A_i(\nu) \,e_i$, 
then the matrix $\big( A_i(\nu) \big)_{i, \nu}$ must be of 
$\F_p$-rank strictly less than $n$. 
For $\theta \div \chi$, the probability of nullity modulo $p$ of a single $\Delta_p^\theta(\gamma)$ 
is $\Frac{1}{p^{f\delta^2}}$; the probability to have at least a $\Delta_p^\theta(\gamma)$ zero modulo $p$ 
for $\theta \div \chi$ is $\Frac{h}{p^{f\delta^2}}$. 
So if we denote by $h_i$, $f_i$, $\delta_i$, the above parameters for the totality of the
 $p$-adic characters of $G$ (grouped by rational characters $\chi_i$), the theoretical probability to obtain 
 a matrix of $\F_p$-rank $< n$ is given by
$$\sm_{i} \Frac{h_i}{p^{f_i\delta_i^2}} - \sm_{i< j} \Frac{h_i}{p^{f_i\delta_i^2}}\Frac{h_j}{p^{f_j\delta_j^2}} + 
\sm_{i< j< k} \Frac{h_i}{p^{f_i\delta_i^2}}\Frac{h_j}{p^{f_j\delta_j^2}} \Frac{h_k}{p^{f_k\delta_k^2}} - \cdots \,, $$

which can be verified by means of programs calculating, for some random $\gamma$,
the number of cases of $\F_p$-rank $< n$ ($G \simeq C_3, C_5, D_6$, respectively, in the variables 
$N_3$, $N_5$, $N_6$). 
Each group $G$ is given via a polynomial defining $K$, but numerical experiments show that the nature of the probabilistic results only depends on $G$ but not of the choice of $K$ nor of the polynomial defining it.

\subsubsection {Case \texorpdfstring{$G$}{Lg} cyclic of order 3\ \rm (two rational characters)}\label{subC3}
We use the Shanks polynomial $P=x^3 - 11 x^2-14 x -1$.
In the case $p\equiv 1 \pmod 3$ we have three $p$-adic characters
 of residue degree $f=1$, in the case $p\equiv 2\pmod 3$ we have a $p$-adic character
of residue degree $f=2$ and the unit character. 
We obtain the following examples (see \cite[Pr. A-4]{Grpro}), where $N_0$ is the number of tested cases 

\smallskip
$p = 41$, $N_0= 4999931 $, $N_3 = 124889$, $\frac{N_3}{N_0} = 0.024978$,
probability $0.024970$.

\smallskip
$p = 43$, $N_0= 4999952$, $N_3 = 341000 $, $\frac{N_3}{N_0} = 0.068200$,
probability $0.068685$.

\subsubsection {Case \texorpdfstring{$G$}{Lg} cyclic of order 5\ \rm (two rational characters)}\label{subC5}
It is the unique studied case for which there are 
(for $p \equiv -1 \pmod 5$) two $p$-adic characters of residue degree $f=2$.
Numerical values obtained (see \cite[Pr. A-5.1]{Grpro}):

\smallskip
$p = 7\ \,$, $N_0= 499977$, $N_5 = 71650 $, $\frac{N_5}{N_0} = 0.14330$,
probability $0.143214$.

\smallskip
$p = 19$, $N_0= 500000$, $N_5 = 29033$, $\frac{N_5}{N_0} = 0.05806$,
probability $0.057880$.

\smallskip
$p = 31$, $N_0= 500000$, $N_5 = 75737$, $\frac{N_5}{N_0} = 0.15147$,
probability $0.151214$.

\smallskip
By modification of the end of the program (\cite[Pr. A-5.2]{Grpro}),
we test the frequency of nullity modulo $p$ of the $\theta$-regulators related to 
 two $p$-adic characters ($p=31$ totally split), and only two among the four 
non trivial characters, namely for instance for $\theta_1$ and $\theta_2$ defined by 
$\theta_1(\sigma^{-1}) \equiv 2$, $\theta_2(\sigma^{-1}) \equiv 4 \pmod p$
\begin{align*}
\Delta_p^{\theta_1}(\gamma) &= \alpha + 2\alpha^\sigma + 4\alpha^{\sigma^2} + 8\alpha^{\sigma^3} 
+16 \alpha^{\sigma^4}, \\
\Delta_p^{\theta_2} (\gamma) &= \alpha + 4\alpha^{\sigma} + 16\alpha^{\sigma^2} +2 \alpha^{\sigma^3} 
+ 8\alpha^{\sigma^4}.
\end{align*}

\smallskip
For $N_0 = 1000000$, $N_1 = 943$ (number of simultaneous nullities of the two regulators), we have 
$\frac{N_1}{N_0} = 0.000943$ and the probability $0.001040$, which shows the independence of regulators 
regarding the $p$-adic characters of a same rational character.

\subsubsection {Case \texorpdfstring{$G$}{Lg} diedral of order 6\ \rm 
(three rational and $p$-adic characters)}\label{subD6}
In this case we have $h = f =1$ for all the characters. 
The results do not depend on congruence classes of the primes $p$ because 
$C = \Q$ (see \cite[Pr. A-6.1]{Grpro}):

\medskip
$p = 13$, $N_0= 49954$, $N_6 = 10794$, $\frac{N_6}{N_0} = 0.21607$,
probability $0.21347$.

$p = 17$, $N_0= 49516$, $\ \,N_6 = 8337$, $\frac{N_6}{N_0} = 0.16836$,
probability $0.16629$.

$p = 29$, $N_0= 49815$, $\ \,N_6 = 5056 $, $\frac{N_6}{N_0} = 0.10149$,
probability $0.09992$.

$p = 31$, $N_0= 40982$, $\ \,N_6 = 3854$, $\frac{N_6}{N_0} = 0.09404$,
probability $0.09368$.

$p = 37$, $N_0= 49998$, $\ \,N_6 = 3959$, $\frac{N_6}{N_0} = 0.07918$,
probability $0.07890$. 

\medskip
Then we take again the same program to make the statistics of the case 
$\delta = 2$ for the character $\chi_2$ of degree 2, 
which may be tested by computing the number $N_2$ of cases 
where the regulators $\Delta_p^1(\gamma)$ and 
$\Delta_p^{\chi_1}(\gamma)$ are nonzero modulo $p$, 
and the matrix of the components of rank 2. This is equivalent to

\medskip
\centerline{ $\Delta_p^\theta(\gamma) \equiv 0 \pmod p$ 
for $\theta = \chi_2$ and ${\mathcal L}^\theta$ of dimension 4}

\smallskip  
(\cite[Pr. A-6.2]{Grpro}). We get the following result for $p=13$

\smallskip
$N_0 = 499541$; $N_2 = 18$; $\frac{N_2}{N_0}= 3.60 \times 10^{-5}$;
 $\frac{1}{p^4}= 3.50 \times 10^{-5}$;\par
$N_1 = 34925$ (number of $\Delta_p^{\chi_2}(\gamma) \equiv 0 \pmod p$); $\frac{N_1}{N_0}=0.06991$; 
$\frac{1}{p}= 0.07692$.

\subsection{Local independence of the components on a basis}\label{sub16}
It remains to verify the nature of ``independent random variables'' of $A_1, \ldots, A_n$; 
we only give two numerical examples ($G = C_3$ and $G=D_6$).

\subsubsection{Cubique cyclic case}\label{sub17}
Let $K$ be the cubic cyclic field defined by the polynomial $x^3-11 x^2-14 x-1$, of conductor $163$. 
This is to check that the variables $A, B, C$, defining $\alpha \equiv A x^2+ B x + C \pmod p$ 
are independent.

\smallskip
The Pr. A-7 of \cite{Grpro} considers random prime to $p$ integers 
$\gamma$ modulo $p^2$, in a small sub-domain 
of $(\Z/p^2\Z)^3$. Then it computes for instance
the number of pairs $(A, B)$ (resp. $(B, C)$, $(C, A)$) 
having an arbitrary fixed value in $\F_p^2$, then the number of 
cases where $\Delta_p^\chi(\gamma) \equiv 0 \pmod p$.

\smallskip
We denote by $N_0$ the number of  prime to $p$ integers $\gamma$ modulo $p^2$ considered, by
$N_1$ the number of cases where $\Delta_p^\chi(\gamma) \equiv 0 \pmod p$ ($\chi$ rational $\ne 1$), 
by $N_2$ the number of pairs $(A, B)$ having the imposed value modulo $p$, and the program computes 
the proportions $\frac{N_1}{N_0}$, $\frac{N_2}{N_0}$, together with $\frac{2}{p}$ or $\frac{1}{p^2}$.

\smallskip
In the array below, we give two cases of residue degree $2$ in $\Q(j)/\Q$ ($j^3=1$, $j\ne 1$) 
and we continue with totally split cases
$$\begin{array}{lllllllll}
\ p &\ N_0 &\ N_1 &\ N_2 &\ \ \frac{N_1}{N_0} & \ \ \frac{N_2}{N_0} &
\ \ \frac{1}{p^2} & \\ \vspace{-0.4cm} \\
5 & 255562 & 10023 & 10155 & 0.039219& 0.039736& 0.04 & \\
11 & 499624 & 4127 & 4191 & 0.00826& 0.008388 & 0.00826& \\
\ p &\ N_0 &\ N_1 &\ N_2 &\ \ \frac{N_1}{N_0} & \ \ \frac{N_2}{N_0} &
\ \ \frac{1}{p^2} &\ \ \frac{2}{p}\\ \vspace{-0.4cm} \\
7 & 498553 & 132167 & 10275 &0.2651& 0.0206 & 0.0204 &0.286 \\
13 & 392751 & 57826 & 2401 & 0.1472 & 0.006113& 0.005917 & 0.154\\
19 & 499907 & 51293 & 1421 & 0.1025 & 0.00284 & 0.00277& 0.105
\end{array} $$

The proportions $\frac{N_2}{N_0}$ are near from $\frac{1}{p^2}$. 
In all the cases $p\equiv 1 \pmod 3$ the proportions $\frac{N_1}{N_0}$ are near from $\frac{2}{p}$
(existence of two $p$-adic characters),
and near from $\frac{1}{p^2}$ in the case $p\equiv 2 \pmod 3$.
If we only impose a numerical value, one gets 
a proportion near from $\frac{1}{p}$, and near from $\frac{1}{p^3}$ if we impose the three values.

\subsubsection{Diedral case \texorpdfstring{$D_6$}{Lg}}\label{sub18}
An analogous study uses Pr. A-8 of \cite{Grpro} and gives the expected results.
For $p=17$, we obtain for three conditions among the six components of $\alpha$,
$N_0 = 494865$, $N_3 = 111$ and $\frac{N_3}{N_0} = 0.0002243$, for $\frac{1}{p^3} =0.0002035$.

\subsection{Extra \texorpdfstring{$p$}{Lg}-divisibilities of regulators} \label{subex}
Recall the decomposition of the normalized regulator of $\eta$ (Remark \ref{rema110} 
and \S\,\ref{rema11})

\medskip
\centerline{${\rm Reg}_p^G (\eta) = \prd_\theta {\rm Reg}_p^\theta (\eta)^{\varphi(1)} \ 
{\rm and}\ \ {\rm Reg}_p^\theta (\eta) = 
\No_{\mathfrak p}\big (P^\varphi \big(\ldots, 
\hbox{$\frac{-1}{p}{\rm log}_p(\eta^\nu)$}, \ldots \big)\big)$.}

\smallskip  
In the case of minimal $p$-divisibility (Definition \ref{defidec}), we have
${\rm Reg}_p^\theta (\eta) \sim p$ for a unique $\theta$, and $ {\rm Reg}_p^G (\eta) \sim p^{\varphi(1)}$.

\smallskip
If we only suppose that $p$ is totally split in $C/\Q$ ($f=1$) and that there
exists $\theta$ such that ${\rm Reg}_p^\theta (\eta) \equiv \Delta_p^\theta (\eta) \equiv 0 \pmod p$ 
(with $\delta=1$), we may have possible extra $p$-divisibilities ${\rm Reg}_p^\theta (\eta) \sim p^e$, 
$e\geq 2$ (then ${\rm Reg}_p^G (\eta) \sim p^{e\,\varphi(1)}$ if $\theta$ is unique), 
for which we want to verify that they are of probability $\Frac{O(1)}{p^2}$.

In \cite{Grpro}, for $K=\Q(j, \sqrt[3]{2})$, $G=D_6$ (in which case any large enough $p$ is convenient 
for the test), the Pr. A-9 cheks this fact for the regulator

\medskip  
$ {\rm Reg}_p^{\chi_2} (\eta) = \hbox{$\frac{1}{\sqrt {-3} }$} 
(E_1^2+E_2^2+E_3^2 - E_4^2-E_5^2-E_6^2 -E_1. E_2-E_2.E_3 -E_3.E_1$

\hfill $+E_4.E_5+E_5.E_6+E_6.E_4) \in \Z$,

\medskip  
where the $E_i$, $1\leq i \leq 6$, are the conjugates of an integer of $K$ (indeed, one may suppose that 
$\Frac{-1}{p}{\rm log}_p(\eta)$ is represented modulo $p^2$ by an arbitrary integer $E\in K$). 

\smallskip
For $p=101$ and $10^6$ tests via {\it random}, we obtain a density of cases $e\geq 2$ equal to
$1.01 \times 10^{-4}$ for a theoretical probability $0.98 \times 10^{-4}$.

\smallskip
For $p=149$, we obtain $4.60\times 10^{-5}$ for a probability $4.50 \times 10^{-5}$.

\medskip
The case of characters of degree 1 offers no difficulty (under the condition $f=1$) and 
we shall make the heuristic assumption that it is the same for all group and all character 
in the $p$-splitted case, and in particular
that $P^\varphi \big(\ldots,\frac{-1}{p}{\rm log}_p(\eta^\nu), \ldots \big)$ 
may have any $p$-adic valuation with the corresponding probability. 
It would be interesting to prove that this property of the polynomials
$P^\varphi (X)$ is universal.

\section{Numerical study of two particular cases}\label{section5}

\subsection{Abelian case }\label{sub27}
We can always reduce to the case where $G$ is cyclic of order $n > 2$,
generated by $\sigma$ (see \S\,\ref{sub7} for the case $n \leq 2$).

\subsubsection{Example of the maximal real subfield 
of \texorpdfstring{$\Q(\mu_{11})$}{Lg}}\label{sub28}
${}$

\smallskip
a) {Search of solutions $p$ such that $\Delta_p^\theta (\eta)\equiv 0 \pmod p$}.\label{sub281} Put 
$$\hbox{$\eta = a \, x^4+b \, x^3+c \, x^2+d \, x+e\ $, with $x = \zeta_{11}+ \zeta_{11}^{-1}$} $$

(see \cite[Pr. A-11 and A-10 for the cubic case]{Grpro},).

\smallskip
(i) For $\eta = -2 \, x^4+ x^3 -3$, the solutions $p \leq 10^7$ are $31, 101, 39451$
 splitted in~$\Q(\zeta_5)$. 

\smallskip
Consider the numerical data for $p=31$
$$\begin{array}{llllllllllllll}
\alpha &\equiv& 25 x^4 & \!\!\!+\!\!\! & 10 x^3 & \!\!\!+\!\!\! & 7 x^2 & \!\!\!+\!\!\! & 21 x & \!\!\!+\!\!\! & 29 &\pmod p \\
\alpha^{\sigma} &\equiv& 4 x^4 & \!\!\!+\!\!\! & 15 x^3 & \!\!\!+\!\!\! & 25 x^2 & \!\!\!+\!\!\! & 7 x & \!\!\!+\!\!\! & 16 &
\pmod p \\
\alpha^{\sigma^2} &\equiv& 26 x^4 & \!\!\!+\!\!\! & 20 x^3 & \!\!\!+\!\!\! & 26 x^2 & \!\!\!+\!\!\! & 18 x & \!\!\!+\!\!\! & 22 & \pmod p\\
\alpha^{\sigma^3} &\equiv& 17 x^4 & \!\!\!+\!\!\! & 6 x^3 & \!\!\!+\!\!\! & 21 x^2 & \!\!\!+\!\!\! & 24 x & \!\!\!+\!\!\! & 4 &\pmod p\\
\alpha^{\sigma^4} &\equiv& 21 x^4 & \!\!\!+\!\!\! & 11 x^3 & \!\!\!+\!\!\! & 14 x^2 & \!\!\!+\!\!\! & 23 x & \!\!\!+\!\!\! & 19 &\pmod p
\end{array}$$

For $r = 4$, which is such that $\theta(\sigma) \equiv r \pmod {\mathfrak p}$ for a pair $(\theta, {\mathfrak p})$, 
we immediately have, as expected
$$\Delta^\theta_p(\eta) =\alpha + r^{-1}\alpha^{\sigma} + r^{-2}\alpha^{\sigma^2} 
+r^{-3}\alpha^{\sigma^3} + r^{-4}\alpha^{\sigma^4} \equiv 0 \pmod p$$

identically on the basis $\{x^4, x^3, x^2, x, 1\}$.

\medskip
(ii) For $\eta = 10 \, x^4 -7 \, x^3 + x -2$, we find the unique solution $p=7$,
first totaly inert case in $\Q(\zeta_5)$. The program gives that all the conjugates of $\alpha$ are zero modulo $p$
(whence moreover $\Delta^1_p(\eta) \equiv 0 \pmod p $).

\smallskip
It is clear that the inert case in $\Q(\zeta_5)/\Q$ is very rare. Furthermore, $p$ is small to compensate 
a probability $\Frac{O(1)}{p^4}$.

\smallskip
(iii) For $\eta = 10 \, x^4-7 \, x^3-3 \, x^2+ x -2$, we find $p=79$
(two $p$-adic characters $\theta$ of residue degree $f=2$; $p$ splitted in $L = \Q(\sqrt 5)$).

\smallskip
The resolvant $\alpha + \zeta_5 \alpha^\sigma + \zeta_5^2\alpha^{\sigma^2} + \zeta_5^3\alpha^{\sigma^3}
 + \zeta_5^4\alpha^{\sigma^4}$ (which corresponds to $\Delta^\varphi_p(\eta)$ for 
 $\varphi(\sigma) = \zeta_5^{-1}$) 
 is decomposed in the following way on the relativ basis $\{1, \zeta_5\}$.

\smallskip
We have the relation $\zeta_5^2 -\zeta_5 \frac{\sqrt 5 - 1}{2} + 1 =0$ defining the irreducible polynomial of 
$\zeta_5$ over $\Q(\sqrt 5)$. We then get
$\zeta_5^3 = -\zeta_5 \frac{\sqrt 5 - 1}{2} + \frac{1-\sqrt 5}{2}$, $\zeta_5^4 = -\zeta_5 + \frac{\sqrt 5 - 1}{2}$, 
and the system of relations in $K(\zeta_5)$ expressing $\Delta^\varphi_p(\eta) \equiv 0 \pmod {\mathfrak p}$
\begin{align*}
\alpha - \alpha^{\sigma^2} + \Frac{\sqrt 5 - 1}{2} (\alpha^{\sigma^4}-\alpha^{\sigma^3}) &
\equiv 0 \pmod {\mathfrak p} \\
\alpha^{\sigma} - \alpha^{\sigma^4} + \Frac{\sqrt 5 - 1}{2} (\alpha^{\sigma^2}-\alpha^{\sigma^3}) &
\equiv 0 \pmod {\mathfrak p}.
\end{align*}

Then, the ideal ${\mathfrak p}$ is for instance defined by the congruence $\sqrt 5 \equiv 20 \pmod {\mathfrak p}$, 
whence $\frac{\sqrt 5 - 1}{2} \equiv 49 \pmod {\mathfrak p}$ which defines the coefficients $r_ i (\nu)$, 
$i = 1, 2$, and $(\theta, {\mathfrak p})$. 

\smallskip
We have obtained two linear relations with independent rational coefficients
\begin{align*}
\alpha - \alpha^{\sigma^2} + 49\, (\alpha^{\sigma^4}-\alpha^{\sigma^3}) 
&\equiv 0 \pmod p \\
\alpha^{\sigma^4} - \alpha^{\sigma} + 49\, (\alpha^{\sigma^3} - \alpha^{\sigma^2}) 
&\equiv 0 \pmod p.
\end{align*}

The numerical data for $\alpha$ and its conjugates are
$$\begin{array}{llllllllllllll}
\alpha &\!\equiv& 37 x^4 & \!\!\!+\!\!\! & 13 x^3 & \!\!\!+\!\!\! & 19 x^2 & \!\!\!+\!\!\! & 3 x & \!\!\!+\!\!\! & 10 &\pmod p \\
\alpha^{\sigma} &\!\equiv& 75 x^4 & \!\!\!+\!\!\! & 24 x^3 & \!\!\!+\!\!\! & 45 x^2 & \!\!\!+\!\!\! & 73 x & \!\!\!+\!\!\! & 33 &\pmod p\\
\alpha^{\sigma^2} &\!\equiv& 5 x^4 & \!\!\!+\!\!\! & 51 x^3 & \!\!\!+\!\!\! & 22 x^2 & \!\!\!+\!\!\! & 60 x & \!\!\!+\!\!\! & 1 &\pmod p\\
\alpha^{\sigma^3} &\!\equiv& 70 x^4 & \!\!\!+\!\!\! & 33 x^3 & \!\!\!+\!\!\! & 40 x^2 & \!\!\!+\!\!\! & 8 x & \!\!\!+\!\!\! & 77 & \pmod p\\
\alpha^{\sigma^4} &\!\equiv& 50 x^4 & \!\!\!+\!\!\! & 37 x^3 & \!\!\!+\!\!\! & 32 x^2 & \!\!\!+\!\!\! & 14 x & \!\!\!+\!\!\! & 22 & \pmod p
\end{array}$$
  
which satisfy the system of the two above congruences.

\smallskip
We have the two independent relations, defining ${\mathcal L}^\theta \simeq V_\theta$ of $\F_p$-dimension 2
$$\ 1 - {\sigma^2} + 49\, ({\sigma^4}-{\sigma^3})\ \ \& \ \ {\sigma^4} - {\sigma} + 49\, ({\sigma^3}-{\sigma^2}), $$

the second one being the conjugate by $\sigma^4$ of the first one.
Whence the probability $\Frac{2}{p^2}$ (two choices $\sqrt 5 \equiv \pm 20 \pmod {\mathfrak p}$).

\medskip
b) {Computation of the density of $\Delta_p^\theta (\eta) \equiv 0 \pmod p$ as a function of $f$}.\label{sub282}
In \cite[Pr. A-12]{Grpro}, the program takes again the previous case and is concerned with the various possible residue degrees of $p$ in $\Q(\zeta_{11}+\zeta_{11}^{-1})/\Q$ to verify that the probability for 
$\Delta_p^\theta (\eta) \equiv 0 \pmod p$ is indeed $\Frac{O(1)}{p^f}$.

\smallskip
We display the theoretical probabilities, depending on the case ($f =1,2,4$), and the number 
$N_1$ of solutions compared with the number $N_0$ of tested $\eta$.

\smallskip
For $p=31$, the residue degree is 1 and we obtain the values $N_1 = 61505$, 

$\frac{N_1}{N_0} = 0.1230$ for $\frac{4}{p} - \frac{6}{p^2} + \frac{4}{p^3} - \frac{1}{p^4}= 0.12292$. 

\smallskip
For $p=19$, the residue degree is 2 and we obtain the values $N_1 = 2756$, 

$\frac{N_1}{N_0} = 0.005512$ for $\frac{2}{p^2}- \frac{1}{p^4}= 0.00553$.

\smallskip
For $p=13$, the residue degree is 4 and we obtain the values $N_1 = 17$, 

$\frac{N_1}{N_0} = 3.40\times 10^{-5}$ for $\frac{1}{p^4} = 3.50 \times 10^{-5}$.

\subsection{Case of the group \texorpdfstring{$D_6$}{Lg}} \label{sub29}
Let $k = \Q(\sqrt m)$ be the quadratic subfield of $K$ and let $\chi_1$, $\chi_2$ be
the two non trivial irreducible rational characters (and $p$-adic) of~$D_6$. We still use
$K=\Q(\sqrt[3] 2, j)$ where $j$ denotes a cubic root of unity ($m=-3$). 

\subsubsection{Recalls}
We study the three local $\chi$-regulators $\Delta^\chi_p(\eta)$, each time supposed 
non trivialy null modulo $p$ (no $\chi$-relations in $F$). We have

\smallskip
\centerline{$\alpha =\alpha_p(\eta)$, $\alpha'= \alpha^\sigma$,
$\alpha''= \alpha^{\sigma^2}$, $\beta = \alpha^\tau$, $\beta' = \alpha^{\tau\sigma} = \alpha'{}^\tau$, 
$\beta'' = \alpha^{\tau\sigma^2} = \alpha''{}^\tau$.}

\medskip
(i) Case of $\Delta^1_p(\eta)$; thus $\No_{K/\Q} (\eta)= a \ne \pm1$, in which case, $\Delta_p^1 (\eta)$ 
is the Fermat quotient of $a$.

\smallskip
(ii) Cas $\Delta^{\chi_1}_p(\eta)$; thus $\No_{K/k} (\eta) \in k^\times \setminus\! \Q^\times$ 
and we suppose that
$$\Delta_p^{\chi_1} (\eta)= \alpha +\alpha' +\alpha'' -\beta -\beta' -\beta'' \equiv 0 \pmod p. $$

If $A = \alpha +\alpha' +\alpha'' =: u+v \sqrt m$, then $\Delta_p^{\chi_1} (\eta)=
A - A^\tau = 2v \sqrt m\equiv 0 \pmod p$; 
we then have the unique condition $v \equiv 0 \pmod p$, which yields the probability $\Frac{O(1)}{p}$.

\smallskip
(iii) Case $\Delta^{\chi_2}_p(\eta)$ (considered up to the factor $\sqrt m$); we have 
${\rm dim}((F\otimes \Q)^{e_\chi}) = 4$
(case of a character of degree 2), which yields, for $\varphi = \theta=\chi_2$, the condition
\begin{align*}
\Delta_p^{\theta} (\eta)&= \alpha^{2} +\alpha'{}^2 +\alpha''{}^{2} -\beta^{2} -\beta'{}^2 -\beta''{}^{2}\\
& \hspace{2cm} - \alpha \alpha' - \alpha' \alpha'' -\alpha'' \alpha + \beta \beta'+ \beta' \beta''+\beta'' \beta 
\equiv 0 \pmod p
\end{align*}
(cf. Example \ref{ex5}).
The calculation of the three representations 
$${\mathcal L}^{\theta'}\simeq \delta' V_{\theta'}, \ \ 0\leq \delta' \leq \varphi'(1), $$

allows us to know what are the $\Delta_p^{\theta'}(\eta)$ equal to zero modulo $p$, 
even if we can exclude the case 
where $\Delta_p^1(\eta)$ or $\Delta_p^{\chi_1}(\eta)$ is zero modulo $p$. 

\smallskip
We begin with examples concerning the $p$-adic character $\theta = \chi_2$.
The Pr. A-13 of \cite{Grpro} computes the conjugates of $\alpha$ on the basis of powers of $x = \sqrt[3] 2 + j$.
This allows us to find the relations of $\F_p$-dependence of these conjugates, under the form
$$c_1 \alpha + c_2 \alpha^{\sigma} + c_3\alpha^{\sigma^2}
+ c_4\alpha^\tau + c_5\alpha^{\tau\sigma}+ c_6\alpha^{\tau\sigma^2} \equiv 0 \pmod p.$$

\subsubsection{Case \texorpdfstring{$\eta = x^5-3x^4-7x^2+x -1$}{Lg}}
We find the solutions $p = 7, 13, 69677, 387161$, up to $10^7$.

\smallskip
a) For $p=7$, we have the following numerical data
$$\begin{array}{lllllllllllllll}
 \alpha &\equiv& &0x^5& \!\!\!+\!\! \! &2x^4 & \!\!\!+\!\! \! &1 x^3 & \!\!\!+\!\! \! &1 x^2 &
  \!\!\!+\!\! \! & 5x & \!\!\!+\!\! \! & 0 & \pmod p \\
 \alpha^\sigma &\equiv& & 1 x^5 & \!\!\!+\!\! \! &1 x^4 & \!\!\!+\!\! \! & 6x^3 &
  \!\!\!+\!\! \! & 3x^2 & \!\!\!+\!\! \! & 5x & \!\!\!+\!\! \! & 2 &\pmod p \\
 \alpha^{\sigma^2} &\equiv& & 0x^5 & \!\!\!+\!\! \! & 2x^4 & \!\!\!+\!\! \! &
  3x^3 & \!\!\!+\!\! \! &0x^2& \!\!\!+\!\! \! & 4x & \!\!\!+\!\! \! &0 &\pmod p \\
 \alpha^\tau &\equiv& & 0x^5 & \!\!\!+\!\! \! & 5x^4 & \!\!\!+\!\! \! & 6x^3 &
  \!\!\!+\!\! \! & 6x^2 & \!\!\!+\!\! \! & 2x & \!\!\!+\!\! \! & 6 &\pmod p \\
 \alpha^{\tau\sigma} &\equiv& & 0x^5 & \!\!\!+\!\! \! & 5x^4 & \!\!\!+\!\! \! &
  4x^3 & \!\!\!+\!\! \! & 0x^2& \!\!\!+\!\! \! & 3x & \!\!\!+\!\! \! & 6 &\pmod p\\
 \alpha^{\tau\sigma^2} &\equiv& & 6x^5 & \!\!\!+\!\! \! & 6x^4 & \!\!\!+\!\! \! &
 1 x^3 & \!\!\!+\!\! \! & 4x^2 & \!\!\!+\!\! \! & 2x & \!\!\!+\!\! \! & 4 &\pmod p,
\end{array}$$

which yields the two linearly independent $\F_p$-relations 
$$\hbox{$\alpha -\alpha^{\sigma} + \alpha^\tau- \alpha^{\tau\sigma^2}\equiv 0 \pmod p$ \ \ \& \ \  
$\alpha - \alpha^{\sigma^2} + \alpha^\tau - \alpha^{\tau\sigma} \equiv 0 \pmod p$,} $$

and their lifts 
$$\hbox{$\eta_1^{1 - \sigma+ \tau -\tau \sigma^2 } \equiv 1 \pmod {p^2}$  \ \ \& \ \  
$\eta_1^{1 - \sigma^2 + \tau - \tau\sigma } \equiv 1 \pmod {p^2}$. }$$

For the $\theta$-relation $U= 1-\sigma+\tau-\tau \sigma^2$ we obtain $\sigma^2U = -U-\sigma U$,
$\tau U = -\sigma^2 U$, $\tau \sigma U = - \sigma U$, $\tau \sigma^2 U = - U$, and $U$ generates
a space of dimension 2 (${\mathcal L}^\theta \simeq V_\theta$).

\smallskip
b) For $p=13$, we obtain the relations
$$\hbox{$\alpha - \alpha^{\sigma^2}+ \alpha^{\tau} - \alpha^{\tau\sigma^2} \equiv 0\ \pmod p$ \ \ \& \ \  
$\alpha^{\sigma^2} - \alpha^{\sigma} + \alpha^{\tau\sigma} - \alpha^{\tau} \equiv 0 \ \pmod p$}$$

and their lifts 
$$\hbox{$\eta_1^{1- \sigma^2 + \tau - \tau\sigma^2 } \equiv 1 \pmod {p^2}$  \ \ \& \ \   
$\eta_1^{ \sigma^2 -\sigma + \tau\sigma - \tau } \equiv 1 \pmod {p^2}$. }$$

 For the $\theta$-relation
$U=1- \sigma^2 + \tau - \tau\sigma^2$, we obtain $\sigma^2 U = -U - \sigma U$, $\tau U = U$, 
$\tau \sigma U = \sigma^2 U$, $\tau \sigma^2 = \sigma U$ (${\mathcal L}^\theta \simeq V_\theta$). 

\smallskip
c) In the case of $p=69677$, we find the coefficients $(c_1, c_2, c_3, c_4, c_5, c_6)=$

\smallskip
\centerline{$(53404, 39540, 46410, 69676, 1, 0)\, \ \& \, \ (23267, 16273, 30137, 69676, 0, 1)$}

\smallskip
and a similar conclusion.

\subsubsection{Case \texorpdfstring{$\eta= x^5-x^4-7x^2+x-1$, $p=7$}{Lg}}
We obtain four independent $\F_p$-linear relations as 
$$\hbox{$\alpha^\tau - \alpha\equiv 0 \pmod p$ \ \ \& \ \ 
$\alpha+\alpha^\sigma+\alpha^{\sigma^2}\equiv 0 \pmod p$,}$$

and their conjugates.

\smallskip
Thus the three regulators are zero modulo $p$. But for $\theta$, ${\mathcal L}^\theta$ is generated by
$U= e_\theta (1-\tau)$ and by $\sigma U$; we have $\sigma^2 U = -U - \sigma U$, 
$\tau U = -U$, $\tau\sigma U = -\sigma^2 U$,
$\tau\sigma^2 U = -\sigma U$ (${\mathcal L}^\theta \simeq V_\theta$).

\subsubsection{Case \texorpdfstring{$\eta= x^5-2x^4+4x^3-3x^2+x-1$, $p=61$}{Lg}}\label{p=61} 
The $G$-module ${\mathcal L}$ is generated by the three independent 
$\F_p$-linear relations (see \cite[Pr. A-13]{Grpro})
$$\begin{array}{lllllllllllllll}
 19\alpha+56\alpha^\sigma+46\alpha^{\sigma^2}+\alpha^\tau & \equiv 0 \pmod p\\
46\alpha+19\alpha^\sigma+56\alpha^{\sigma^2}+\alpha^{\tau \sigma^2} &\equiv 0 \pmod p\\
56\alpha+46\alpha^\sigma+19\alpha^{\sigma^2}+\alpha^{\tau \sigma} &\equiv 0 \pmod p\,.
\end{array} $$

The idempotent $e_{1}$ gives the trivial relation (because $19+56+46+1 \equiv 0 \pmod {61}$), thus
 the Fermat quotient of $\Delta_p^{1} (\eta)$ is nonzero modulo $p$.

\smallskip
We obtain the ${\chi_1}$-relation corresponding to the idempotent $e_{\chi_1}$ 
by summation of the three relations, which yields 
$$\alpha+\alpha^\sigma+\alpha^{\sigma^2}-\alpha^\tau-\alpha^{\tau\sigma}-\alpha^{\tau\sigma^2}
\equiv 0 \pmod p$$

(whence for $\theta=\chi_1$ the nullity modulo $p$ of the $\theta$-regulator $\Delta_p^{\theta} (\eta)$). 

\smallskip
In fact it is a trivial nullity, the program finding that all the primes $p$ are solution 
for $\Delta_p^{\theta}(\eta)\equiv 0 \pmod p$; the conjugates of $\eta$ satisfy to 
$$\eta^{1+\sigma+\sigma^2-\tau-\tau\sigma-\tau\sigma^2}=1. $$

The choice of $\eta$ being random, this fact was a pure coincidence !

\smallskip
For $\theta=\chi_2$ (of degree 2), the $\theta$-regulator 
$\Delta_p^{\theta} (\eta)$ is zero modulo $p$ (non trivially) and this corresponds to
the following $\theta$-relation by use of $e_{\theta} = \frac{1}{3}(2-\sigma-\sigma^2)$
$$-\alpha + 36\alpha^\sigma+ 26 \alpha^{\sigma^2}+21\alpha^\tau 
+ 20 \alpha^{\tau\sigma} + 20\alpha^{\tau\sigma^2} \equiv 0 \pmod p. $$

By conjugation, this last relation generates a $\F_p$-space of dimension 2 
(in other words, ${\mathcal L}^\theta \simeq V_\theta$).
Indeed, for the corresponding $\theta$-relation
$$U=-1 + 36\sigma+ 26{\sigma^2}+21\tau + 20{\tau\sigma} + 20{\tau\sigma^2}, $$

we have by definition $\sigma^2 U= -U - \sigma U$ and we find the relations 
$\tau U = 24U+51\sigma U$, $\tau\sigma U = 27U+37\sigma U$, $\tau\sigma ^2 U = 10U+34\sigma U$.

\subsubsection{Case \texorpdfstring{$\eta=3 x^5-20x^4+15 x^3+16x^2+9x+21$, $p=7$}{Lg}}
We find
\begin{align*}
&\alpha \equiv \alpha^{\sigma} \equiv \alpha^{\sigma^2} \equiv 
6 x^5 + 2 x^4 + 4 x^3 + 3 x^2 + 6 \pmod p, \\
&\alpha^\tau\equiv \alpha^{\tau\sigma} \equiv
\alpha^{\tau\sigma^2}\equiv x^5 + 5 x^4 + 3 x^3 + 4 x^2 + 6 \pmod p. 
\end{align*}

This case for which the $G$-module ${\mathcal L}^\theta$ is of $\F_p$-dimension 4 
(${\mathcal L}^\theta \simeq 2 \,V_\theta \simeq e_{\theta}\,\F_p[G]$) is very rare, 
as we have seen \S\,\ref{subD6} (probability $\Frac{O(1)}{p^4}$), because we must 
take $\eta$ in such a way that ${\rm rg}(F)= 6$ and that none of the
$\Delta^\chi_p(\eta)$, $\chi = 1, \chi_1$, be zero modulo $p$, which is here the case.

\section{Sets of residues modulo \texorpdfstring{$p$}{Lg} in 
\texorpdfstring{$Z_K$}{Lg}}\label{section6}
The application of the Borel--Cantelli principle only depends on the 
obstruction of minimal $p$-divisibility (Definition \ref{defidec}).
So we propose in this section and the next one to remove this 
obstruction by means of the same heuristic used in \cite{Grqf} 
for Fermat quotients of rational integers. 

\smallskip
The fundamental point being the use of the Archimedean metric together with the $p$-adic one.

\subsection{Definition of sets of residues.}

\subsubsection{Recalls on Fermat quotients.}\label{princip}
In the case $K=\Q$, we work in the set of residues ${\mathcal I}_p := [1, p[$ 
to find the $z \in {\mathcal I}_p$ such that $\Delta_p^1(z)= q_p(z) \equiv 0 \pmod p$ 
(cf. \S\,\ref{sub7}\,(i)) or more generally 
$\Delta_p^1(z) \equiv u \pmod p$ for a given $u \in [0, p[$. 
We then study the invariants $m_p(u)$ (number of $z \in {\mathcal I}_p$ such that 
$\Delta_p^1(z) \equiv u \pmod p$) and 
$$M_p = {\rm max}_{u\in [0, p[} \, \big(m_p(u) \big)$$

(maximal number of repetitions of the Fermat quotient). 
Then, we observe the stability of $M_p = O({\rm log}(p))$, or 
$M_p =\Frac{ {\rm log}(p)}{ {\rm log}_2(p)} \cdot (1+ \epsilon(p))$ (see \cite{Grcompl} for 
improvements and discussion about this question), then the fact that a $u_0\in [0, p[$ such 
that $m_p(u_0)=M_p$ is random, and that the proportion of Fermat quotients obtained in $[0, p[$, 
by at least a $z \in {\mathcal I}_p$, tends to $1-e^{-1} \approx 0.63212$ when $p \to \infty$.

\smallskip
It is the analysis of these numerical results which suggests the existence of a
binomial law of probability on the $m_p(u)$, $u\in[0, p[$, with parameters 
$\big( p-1, \frac{1}{p} \big)$, giving :

\smallskip
\centerline{${\rm Prob}\big( m_p(u) \geq m\big) = 
\Frac{1}{p^{p-1}} \sm_{j=m}^{p-1} \hbox{$\binom{p-1} {j}$} (p-1)^{p-1-j} \ \ 
\hbox{(\cite[Section 4]{Grqf}). }$}

\medskip
In particular, the probability to have $m_p(u) \geq 1$, (i.e., $u \in [0, p[$ is reached),
is precisely rapidely equal to $1-e^{-1} \approx 0.63212$ when $p \to \infty$.

\medskip
If we apply this heuristic to $a \geq 2$ fixed and $\Delta_p^1(a) \equiv 0 \pmod p$, $p\to\infty$,
the solutions $z\in {\mathcal I}_p$ to $\Delta_p^1(z) \equiv 0 \pmod p$
are at least $h :=\big \lfloor \frac{{\rm log}(p)}{{\rm log}(a)} \big\rfloor$ in number (for the $z = a^j$, 
$1\leq j \leq h$, and are said {\it exceptional solutions)},
in which case, an elementary analytical calculation gives a probability of the form
$${\rm Prob}\big( q_p(a) \equiv 0 \! \!\!\pmod p\big) \leq \dsfrac{O(1)}{p^{ {\rm log}_2(p)/ {\rm log}(a)-O(1) }}, 
\hbox{ for $p \to \infty$.}$$ 

As $M_p = O({\rm log}(p))$ and since $M_p \geq m_p(0) \geq h$ in the case of such exceptional solutions,
one can say that $M_p \approx m_p(0) \approx h =O({\rm log}(p))$, even if $M_p>m_p(0)$ for some
reasons explained \S\,\ref{remfond} (i).

\smallskip
When $m_p(0)= O({\rm log}(p))$ (or $m_p(0)=M_p$) without the existence of $a \ll p$ such that
 $\Delta_p^1(a) \equiv 0 \pmod p$, we shall speak of {\it abundant solutions} for the 
 $z\in {\mathcal I}_p$ such that $\Delta_p^1(z) \equiv 0 \pmod p$. 
This means that a number almost maximal of repetitions to $\Delta_p^1(z) \equiv u \pmod p$ 
takes place for $u=0$.
The case of exceptional solutions is a (rarest) particular case of abundant solutions.

\subsubsection{Generalization for dimension \texorpdfstring{$n>1$}{Lg}.}\label{dimn} 
In the case of a field $K \ne \Q$, the ring of integers $Z_K$ is of $\Z$-dimension $n>1$, 
and similarly for $Z_K/p \,Z_K$ as $\F_p$-vectoriel space. Consequently, a natural 
set ${\mathcal I}_p$ in this case is for instance

\smallskip
\centerline{${\mathcal I}_p = \big\{ \sm_{i=1}^n z_i\,e_i, \ \,z_i \in 
\hbox{$]-\frac{p}{2}, \frac{p}{2}[$}\ \ \forall \, i \big\}$, }

\smallskip  
where $\big( e_i \big)_{i=1,\ldots,n}$ is a $\Z$-basis of $Z_K$. 
The other choice $z_i \in [1, p[$ is not possible because we need a 
complete set of residues modulo $p\,Z_K$ being also ``Archimedean'', that is to say
of the form $\{z \in Z_K,\ \vert z_i \vert_\infty < R\}$, where $R$ simply 
depends on $p$, as for $R=\frac{p}{2}$, because, contrary to the case 
$n=1$, $e_1=1$, the signs are not controled (especially if $K$ is not real). 

\smallskip
As for dimension 1, the fondamental principle still consists in the consideration of a fixed $\eta \in Z_K$ 
with primes $p\to\infty$, such that $\Delta_p^\theta(\eta) \equiv 0 \pmod p$, to remark that the first powers 
$\eta^j$ of $\eta$ are still in ${\mathcal I}_p$ (Lemma \ref{lem3}) and verify 
$\Delta_p^\theta(\eta^j) \equiv 0 \pmod p$ (Theorem \ref{thm1}), giving $O({\rm log}(p))$ 
exceptional solutions leading to the same conclusion as for the case $K=\Q$.

\smallskip
On the other hand, for $n>1$, the probabilistic study of the $\Delta_p^\theta(z)$, $z\in {\mathcal I}_p$ (in particular 
the computation of the $m_p(u)$ and of $M_p$), is numericaly out of range for very large prime numbers $p$ 
(program with loops needing $p^n$ calculations) and we must define another process allowing the use
of large $p$, while preserving the statistical relevance.

\smallskip
Before we can give an overview of these computations in ${\mathcal I}_p$ for dimension $n>1$ by means of
the cyclic cubic field $K= \Q(x)$, $x=\zeta_7+\zeta_7^{-1}$, where $\zeta_7$ is a $7$th root of unity 
(Pr. B of \cite{Grpro}). Put 
$$\hbox{$z = a\,x^2+b\,x+c \in {\mathcal I}_p$, $\ \ a, b, c \in ]-\frac{p}{2}, \frac{p}{2}[$.} $$

To limit ourselves to the conditions of Definition \ref{defidec}, we suppose 
$p\equiv 1 \pmod 3$ and we fixe $\theta \ne 1$ (defined by means 
of $r \in [1, p[$ of order 3 modulo $p$).

\smallskip
This raises the problem of weighting the values $m_p(u)$ and $M_p$ (very large); 
we have computed the quantities $m'_p(0) = \Frac{m_p(0)}{N_p}$, $M'_p = \Frac{n\,(p-1)\, M_p\,}{N_p}$, 
where $N_p = p^3-1$ or $(p-1)^3$ (depending on whether $n_p=3$ or $1$) 
is the number of triples $(a,b,c)$ such that $z$ is prime to $p$; 
these quantities coincide with the expressions of the case $n=1$. 
We denote by $u$ an element of $[0, p[$ which realises $M_p$.

\smallskip
The case of $M'_p$ is more difficult concerning a possible multiplicative 
constant (the factor $n$ seems coherent since it 
takes into account the action of $G$ on ${\mathcal I}_p$ 
when this set is a $G$-module, case where the chosen basis is a normal basis). 
We still obtain $M'_p = O({\rm log}(p))$.

For primes $p< 67$, the value of $\Frac{M'_p}{ {\rm log}(p)} = 
\Frac{3\, (p-1)\, M_p\,}{N_p\times {\rm log}(p)}$ is near from $1$, 
but it seems that this quantity is decreasing and rapidely bounded by $1$; 
the case $p=61$ is particular because $u=0$ (exceptional solutions: 
$m'_p = 1.92651$, $\frac{M'_p}{ {\rm log}(p)} = 1.41159$)
\footnotesize
$$\begin{array}{llllll}
p = 67 & n_p = 3 & u = 9 & N_p = 300762 \\
m'_p(0) = 0.98046 & M_p = 4732 & M'_p = 3.11520 & 
\frac{3\, (p-1)\, M_p\,}{N_p\times {\rm log}(p)} = 0.74354\vspace{0.08cm} \\
p = 73 & n_p = 3& u = 5 & N_p = 389016 \\
m'_p(0) = 0.99537 & M_p = 5568 & M'_p = 3.09161 & 
\frac{3\, (p-1)\, M_p\,}{N_p\times {\rm log}(p)} = 0.72290 \vspace{0.08cm} \\
& \ldots & \ldots & \vspace{0.08cm} \\
p = 139 & n_p = 1& u = 72 & N_p = 2628072 \\
m'_p(0) = 0.99107 & M_p = 19322 & M'_p = 3.04379 & 
\frac{3\, (p-1)\, M_p\,}{N_p\times {\rm log}(p)} = 0.61684 \vspace{0.08cm} \\
p = 151 & n_p = 3 & u = 75 & N_p = 3442950 \\
m'_p(0) = 0.97416 & M_p = 23458 & M'_p = 3.06600 & 
\frac{3\, (p-1)\, M_p\,}{N_p\times {\rm log}(p)} = 0.61108
\end{array} $$
\normalsize

\subsubsection{Another approach for dimension \texorpdfstring{$n>1$}{Lg}.}\label{Asub2}
The problem is multiplicative since the $O({\rm log}(p))$ first powers of $\eta$ must
belong to the Archimedean set $I_p \subseteq {\mathcal I}_p$ ($I_p$ to be defined) 
which must contain the exceptionnal 
solutions when $\Delta_p^\theta(\eta) \equiv 0 \pmod p$. 
Moreover, the numerical aspect needs to work in a ``structure of dimension 1'' 
by analogy with the case $K=\Q$. 
Give for this the following definitions:

\begin{definition}\label{defresidus}
We make choice of an integer basis $\big( e_i \big)_{i=1,\ldots,n}$ of $K$, 
and for all $\gamma \in Z_K$ we put
$\gamma = \sm_{i=1}^n c_i\, e_i$, $c_i \in\Z$ for all $i$.

(i) We call residue modulo $p$ of $\gamma$ the integer $[\gamma]_p := \sm_{i=1}^n [c_i]_p \, e_i$ 
of $Z_K$, with 
$$\hbox{$[c_i]_p\in\hbox{$]-\frac{p}{2} , \frac{p}{2} ]$}\, \ \&\, \ c_i \equiv [c_i]_p \pmod p$.}$$

(ii) We define the set of residues $I_p(\gamma) := \big\{ \big[\gamma^k \big]_p, \ \ k \in [1, p[ \big\}$.

(iii) We denote by $z = \sm_{i=1}^n z_i\, e_i$, $z_i \in \hbox{$]-\frac{p}{2} , \frac{p}{2} ]$}$ for all $i$, 
any element of $I_p(\gamma)$.
\end{definition}

In the case $n=1$ of the Fermat quotient of fixed $\eta=a$,
if $\gamma=g$ is a primitive root modulo $p>2$, $I_p(g) :=\{ [g^k ]_p, \ k \in [1, p[ \} = 
\{-\frac{p-1}{2}, \ldots, -1, 1, \ldots, \frac{p-1}{2}\}$
(up to the order). The set $\{ [a^k ]_p, \ k \in [1, p[ \} \subseteq I_p(g)$ has a periodicity 
if $a$ is not a primitive root modulo $p$ and we must not base the statistical study
on this set, but on $I_p(g)$. Moreover, the group of roots of unity (here $\pm1$) 
must be taken into account.

\smallskip
In the general case, let $\mu_K$ be the group of roots of unity of the field $K$.
We denote by $D$ the order of $\eta$ in $(Z_K/p\,Z_K)^\times \simeq \prd_{v \div p} F_v^\times$, 
$F_v \simeq \F_{p^{n_p}}$, and by $d \div D$ the order of 
$\eta$ in $\Big(\prd_{v\div p}F_v^\times \Big) \Big / i_p(\mu_K)$, 
where we suppose that $\eta$ generates a multiplicative $\Z[G]$-module of rank $n$. 
We have the following result for all prime $p$ large enough:

\begin{lemma}\label{lem2}
(i) There exists $\gamma \in Z_K$ such that $[\eta]_p=\big [\gamma^{(p^{n_p}-1)/D} \big]_p$.

\smallskip
(ii) If $n_p>1$, then $d$ (hence $D$) does not divide $p-1$.

\smallskip
(iii) We have $D\geq d\geq \dsfrac{{\rm log} (p-1)}{{\rm log} (c_0(\eta))}$, where 
$c_0(\eta) = {\rm max}_{\sigma \in G}(\vert \eta^\sigma \vert )$.
\end{lemma}

\begin{proof} (i) For all $v\div p$, let $g_v \in Z_K$ whose image in $F_v^\times$ is a generator, 
and let $\eta_v$ be the image of $\eta$. 
We put $\eta_v = g_v^{\lambda_v}$ then $\eta_v = g_v^{\lambda\,\mu_v}$ where 
$\lambda = {\rm p.g.c.d.\,}((\lambda_v)_v)$. 
Then $(g_v^{\mu_v})_v$ is of order $p^{n_p}-1$. It is sufficient to take 
$\gamma \in Z_K$ whose diagonal image in 
$\prd_{v \div p}F_v^\times$ is equal to $(g_v^{\mu_v})_v$ to 
obtain $\eta \equiv \gamma^\lambda \pmod p$. 

\smallskip
Replacing $\gamma$ by $\gamma^\mu$, $\mu$ prime to $p^{n_p}-1$, 
one may suppose that $\lambda = \frac{p^{n_p}-1}{D}$.

\smallskip
(ii) Let $v \div p$ and let $\tau_v \in G$ be the corresponding Frobenius automorphism; it is such that 
$\eta^{\tau_v} \equiv \eta^{p} \pmod {{\mathfrak p}_v}$ in $K$, whence 
$\eta^{\tau_v-1} \equiv \eta^{p-1} \pmod {{\mathfrak p}_v}$.
If we suppose $\eta^{p-1} \equiv \zeta \pmod p$, $\zeta \in \mu_K$ of order $r \geq 1$,
then $\eta^{r(\tau_v-1)} \equiv 1 \pmod {{\mathfrak p}_v}$, which leads to
 $\eta^{r\tau_v} \equiv \eta^r \pmod {{\mathfrak p}_v}$.
But the integer $\eta^{r\tau_v} - \eta^r$ is nonzero since $\eta^r$ is not in a strict subfield of $K$ and if
${\mathcal D}(\eta^r)$ is its discriminant, it is a nonzero rational integer hence not divisible by $p$ 
for all $p$ large enough (absurd).

\smallskip
(iii) We have $\eta^d = \zeta + \Lambda\,p$, where $\zeta \in \mu_K$, $\Lambda \in Z_K \setminus \{0\}$. 
Up to conjugation, one may suppose that $\vert \Lambda \vert \geq 1$. 
This yields $\vert \eta \vert^d \geq \vert \Lambda \vert\,p -\vert \zeta \vert\geq p-1$, and finally
$d \geq \Frac{{\rm log} (p-1)}{{\rm log} (\vert \eta \vert)} \geq \Frac{{\rm log} (p-1)}{{\rm log} (c_0(\eta))}$.
\end{proof}

The set $\{ [\gamma^t]_p , \ \, t\in [0, p^{n_p}[ \}$ is the union of $p^{n_p-1}$ sets
$I_p^{(\lambda)} = \{ [\gamma^{\lambda p + k}]_p$, $k \in [0, p [ \}$, $\lambda \in [0, p^{n_p-1}[$. 
A first heuristic is to say that these sets $I_p^{(\lambda)}$
have the same statistical behavior concerning the numbers $m_p(u)$ and $M_p$.
We may in general consider the set $I_p (\gamma) := \{ [\gamma^k]_p , \ k \in [1, p [ \}$. 
We distinguish two cases about numerical experiment:

\medskip
a) Case $n_p>1$. In general $\big\vert I_p (\eta)\big \vert = p-1$ except if $D<p$ (e.g. $p=5$, $n_p=2$, 
$\eta^3 \equiv 1 \pmod p$). 
But when $p\to\infty$ one can use the following heuristic/conjecture:

\begin{heuristic}\label{miracle} {\it We suppose that $K$ is distinct from $\Q$ and from a quadratic field. 
The primes $p$ for which $n_p>1$ and $\eta$ is of order $D$ modulo $p$, with $D<p$, are finite in number. }
\end{heuristic}

Put $n=n_p g_p$ and let $(\eta_v)_{v\div p}$ be the image of $\eta$ in $\prd_{v\div p} F_v^\times$. 
To say that $\eta$ is of order a divisor of $D$ modulo $p$ is equivalent to the $g_p$ independent 
conditions $\eta_v^D=1$ for all $v \div p$ whose probability is 
$\Big(\dsfrac{D}{p^{n_p}-1} \Big)^{g_p} \sim \dsfrac{D^{g_p}}{p^{n_pg_p}}$. 
If we sum over the $D < p$, we obtain the upper bound
$O(1) \dsfrac{p^{g_p+1}}{p^{n_p g_p}} = \dsfrac{O(1)}{p^{n_pg_p-g_p-1}}$, 
which is clear for $(n_p- 1)g_p>2$ and needs a particular study in the case 
$n_p=1$ and in the case of a quadratic field $K$ with $p$ inert.

\smallskip
In this last case, we can define a ``structure of dimension 1'' in the following way. We replace $\eta$
by $\eta' = \eta^{\tau - 1}$ where $\tau$, is a generator of $G$ and the Frobenius automorphism at $p$; 
we then have $\eta' \equiv \eta^{p-1} \!\!\pmod p$ and $\eta'$ is of order $D' \div p+1$ modulo~$p$. 
There exists $\gamma$ of order $p+1$ modulo $p$ such that $I_p(\gamma)$ contains 
$\eta' \equiv \gamma^{(p+1)/D'} \!\!\pmod p$. 
As $\eta^{2\tau}=\eta^{\tau+1}\eta' =: a \, \eta'$ the theory of the
$\Delta_p^\theta(\eta)$ is identical to that of $\Delta_p^\theta(\eta')$ for $\theta \ne 1$ and moreover, $\eta'$
is independent of $p$ and remains ``small''. We shall have $[\eta'{}^j]_p = \eta'{}^j$ for $1\leq j \leq h' = O(h)$ 
because $D'>h'$ as in Lemma \ref{lem2}.

\begin{remark} 
By using an analytical argument of \cite{T}, one may replace $\dsfrac{O(1)}{p^{n_pg_p-g_p-1}}$ 
by the upper bound
$\dsfrac{C_\epsilon}{p^{n_pg_p-g_p-\epsilon}} \ \hbox{(for all $\epsilon>0$ and $p$ large enough)}$,
which allows us to eliminate cubic and quartic cases, but not the quadratic case for which we have conjecturally 
infinitely many solutions $p$ (cf. \cite{Grcub}).
\end{remark}

The point (ii) of Lemma \ref{lem2} enforces this heuristic.
Thus we shall base the statistical study on $I_p = I_p (\eta)$. We admit, as for the case of
Fermat quotients (cf. \cite{H-B}), that the $\Delta_p^\theta(z)$ are uniformly distributed from
any set with $p-1$ elements of residues $z$ generated by the powers of a fixed integer.

\medskip
b) Case $n_p=1$ ($p$ totally split in $K$).
The order $D$ of $\eta$ modulo $p$ is a divisor of $p-1$ and the probability for this order to be a strict divisor 
of $p-1$ is $1-\Frac{\phi(p-1)}{p-1}$, where $\phi$ is the Euler function,
which roughly is between $\frac{1}{2}$ and $1-\frac{1.781}{{\rm log}_2(p)}$. 
So we cannot consider $ I_p (\eta)$.
We use Lemma \ref{lem2}\,(i) to create a set of residues of the form 
$I_p(\gamma) = \{ [\gamma^k]_p , \ k \in [1, p [ \}$ which contains $[\eta]_p$ and which has $p-1$ elements. 
We may always choose $\gamma$ such that $[\eta]_p = [\gamma^{(p-1)/D}]_p$. 

\smallskip
From Lemma \ref{lem2}\,(iii), $I_p(\gamma)$ contains
$d=O({\rm log}(p))$ distinct residues of the form $[\eta^j]_p$ for $1 \leq j \leq d$ which are not in $\mu_K$. 

\smallskip
For the numerical experiments, we shall use $I_p=I_p(\gamma)$
generated by a $\gamma$ having the good generating properties because the goal is to verify the validity
of the existence of a binomial probability law for the values of the $m_p(u)$, which is a property of $I_p$
and not a property of its elements; in other words, $I_p$ must be the analogue of $I_p(g)$ for the dimension 1 
and if we study a fixed $\eta$ (analogue of $a \geq 2$  for the dimension 1) when $p \to \infty$, one may say that
$\eta$ belongs to a suitable $I_p(\gamma)$ in which the heuristic applies (as for $a \in I_p(g)$).

\smallskip
The programs do not make the distinction between $\eta$ and $\gamma$ insofar as the case
$\vert I_p \vert < p-1$ is very rare. We suppress from $I_p$ the roots of unity $\zeta \in \mu_K$ because 
$\alpha_p(\zeta)=0$ would modify the statistics (we still meet this case in dimension 1 where 
$\{-1, 1\} \subset I_p = ]-\frac{p-1}{2}, \frac{p-1}{2}[$).

\smallskip
After, for the computation of the $m_p(u)$ and of $M_p$, relative to $I_p$, 
we shall prove that if $\Delta_p^\theta(\eta)\equiv 0 \pmod p$ 
(analogue of $q_p(a)\equiv 0 \pmod p$), we have
$\Delta_p^\theta(\eta^j)\equiv 0 \pmod p$ (analogue of 
$q_p(a^j)\equiv 0 \pmod p$), for all $j \leq h=O({\rm log}(p))$
(Lemma \ref{lem3} and Theorem \ref{thm1} below).

\subsubsection{Fundamental Archimedean principle.}\label{archi}
 If for instance, $\eta=\gamma = 2+i \in \Z[i]$, with $i^2=-1$, this yields
$\eta^2=3+4\,i$, $\ \eta^3=2+11\, i$, $\ \eta^4=-7+24\,i$, $\ \eta^5=-38+41\, i$, 
$\ \eta^6=-117+44\, i$, $\ \eta^7=-278-29\, i$,\ldots

\smallskip
We see that if $p\to\infty$, the residues $[\eta^{j } ]_p$ will coincide with the exact values
(not reduced), $\eta^j$, for a finite number of indices $j$, and after we shall have the 
corresponding residues; for $p=47$ we get $]-\frac{p}{2} , \frac{p}{2}[ = [-23, 23]$ and

\medskip
$I_p = \big \{$$[\eta]_p =\eta$, $\ [\eta^2]_p= \eta^2$, $\ [\eta^3]_p=\eta^3$, $\ 
[\eta^4]_p=-7-23\,i$, $\ [\eta^5]_p=9-6\, i$, $\ [\eta^6]_p=-23-3\, i$, $\ [\eta^7]_p=4+18\, i$, $\ldots \}$. 

\medskip
More precisely, we have the following result:

\begin{lemma}\label{lem3} Let $\eta \in Z_K \setminus \{0\}$, be an integer of $K$, 
distinct from a root of unity, and let 
$c_0(\eta) = {\rm max}_{\sigma \in G}(\vert \eta^\sigma \vert )$.
Then there exists an explicit constant $\Gamma(K) \geq 1$, independent of $\eta$ and $p$, 
such that $\big[\eta^j \big]_p=\eta^j$ for all $j$ such that
$$\hbox{$1\  \leq\  j\  \leq\  \dsfrac{ {\rm log} (p-1) - {\rm log} (2\,\Gamma(K))}{{\rm log} (c_0(\eta))}$} $$

(since $\vert \No_{K/\Q}(\eta) \vert \geq 1$ and $\eta$ 
is not a root of unity, we have $c_0(\eta) >1$).
\end{lemma}

\begin{proof} Put $\eta^j = \sm_{i=1}^n A_{j, i} \, e_i$, $A_{j, i} \in \Z$ for all $i=1,\ldots, n$.
We have $\eta^{j\,\sigma} = \sm_{i=1}^n A_{j, i} \,e_i^\sigma$ for all $\sigma \in G$. 
The matrix $\big(e_i^\sigma \big)_{i, \sigma}$ is invertible (the square of its determinant 
is the discriminant of the field $K$); the coefficients $\Gamma_i^\sigma$ of the inverse 
matrix are elements of $K$ independent of $\eta, p, j$, and

\medskip
\centerline{$A_{j, i} = \sm_{\sigma \in G} \Gamma_i^\sigma\, \eta^{j\,\sigma} , \ \ i= 1, \ldots, n$. }

\medskip
A sufficient condition to have $\vert A_{j, i} \vert < \frac{1}{2}\, p$ for all $i$, is that a common upper bound 
of these numbers be less than $\frac{1}{2}\, (p-1)$. But we have

\medskip
\centerline{$\big \vert \sm_{\sigma \in G} \Gamma_i^\sigma\, \eta^{j\,\sigma} \big \vert \leq 
\sm_{\sigma \in G}\vert \Gamma_i^\sigma\vert \, \vert \eta^{j\,\sigma}\vert 
\leq c_0(\eta)^j\,\sm_{\sigma \in G}\vert \Gamma_i^\sigma\vert $. }

Put 
\begin{equation}\label{Gamma}
\Gamma(K) := {\rm max}^{}_{\,i=1,\ldots,n}\big(\sm_{\sigma \in G}\vert \Gamma_i^\sigma\vert \big)
\end{equation}

(maximum of the sums of the lines); then it is sufficient to have 
$$c_0(\eta)^j \cdot \Gamma(K) \leq \Frac{1}{2}(p-1), $$
 
whence the result.
If $1=\sm_{k=1}^n \lambda_k e_k$, $\lambda_k \in \Z$, we have 
$$\hbox{$\sm_{\sigma \in G} \Gamma_i^\sigma \times 1^\sigma
= \sm_{\sigma \in G} \sm_{k=1}^n \Gamma_i^\sigma \lambda_k e_k^\sigma 
= \sm_{k=1}^n\delta_{i,k}\lambda_k = \lambda_i$,\, for all~$i$; } $$ 

there exists at least an index $i$ such that $\sm_{\sigma \in G} \vert\Gamma_i^\sigma \vert \geq 1$.
\end{proof}

The general case is thus analogous to that of Fermat quotients and leads to the following 
result with the notation of Lemma \ref{lem3}:

\begin{theorem}\label{thm1} Let $\eta\in Z_K$ generating a multiplicative 
$\Z[G]$-module of rank $n$. 
Let $p$ be large enough and let $I_p=I_p(\gamma)$ 
(Definition \ref{defresidus}) be such that $\vert I_p \vert = p-1$ 
and such that $\eta \in I_p$; let $\theta$ be an irreducible $p$-adic 
character of $G$. 

\smallskip  
If $\Delta_p^\theta(\eta) \equiv 0 \pmod p$ we have $z_j := \eta^j \in I_p$ 
and $\Delta_p^\theta(z_j) \equiv 0 \pmod p$ for all $j$ 
such that $1 \leq j \leq h$, where 
$h=h_p(\eta) := \dsfrac{ {\rm log} (p-1) - {\rm log} (2\,\Gamma(K))}{{\rm log} (c_0(\eta))}$ 
(Lemma \ref{lem3} and Relation \eqref{Gamma}). Moreover, $z_j \notin \mu_K$.
\end{theorem}

\begin{proof} Put $\eta=[\gamma^e]_p$. The case $n_p >1$ where $\gamma = \eta$ is obvious since $e=1$.
If $n_p=1$ and $\eta = \gamma^{\frac{p-1}{D}}$, $e=(p-1)/D$ and Lemma \ref{lem2}\,(iii) shows that 
$$D\ \geq\ d\ \geq\ \dsfrac{ {\rm log} (p-1)}{ {\rm log} (c_0(\eta))} \ >\  h; $$ 

consequently, the conditions $e\,h \leq p-1$ and $z_j \notin \mu_K$ are always fullfiled.
We know that $\alpha_p(\eta^j) \equiv j\,\alpha_p(\eta) \pmod p$ for all $j$ and that we have
$\eta \equiv \gamma^e \pmod p$, $e \in [1, p[$.
If we restrict ourselves to the $j\leq h$, we get $\eta^j \equiv \gamma^{e\,j} \pmod p$ and 
$\eta^j =\big[\eta^j \big]_p = \big[\gamma^{e\,j} \big]_p =: z_j \in I_p$ since $e\,j \leq p-1$.

\smallskip
By definition of the $G$-modules ${\mathcal L}^\theta$ (whose non triviality is equivalent to the nullity of
the corresponding $\Delta_p^\theta$), we have ${\mathcal L}^\theta(\eta^j) = {\mathcal L}^\theta(\eta)$ 
in $\F_p[G]$ because any $\theta$-relation $\sm_{\nu \in G} u(\nu) \alpha_p(\eta)^{\nu^{-1}} \equiv 0 \pmod p$, comming from ${\mathcal L}^\theta(\eta)$, is equivalent to

\smallskip
\centerline {$\sm_{\nu \in G} u(\nu) \alpha_p(\eta^j)^{\nu^{-1}} \equiv
 j\,.\,\sm_{\nu \in G} u(\nu) \alpha_p(\eta)^{\nu^{-1}} 
\equiv 0 \pmod p$, }

\smallskip  
$j \leq h$ never being divisible by $p$. 
So the $\Delta_p^\theta(\eta^j)$, characterized via the ${\mathcal L}^\theta(\eta^j)$, 
are all zero modulo $p$ as soon as $\Delta_p^\theta(\eta)$ is zero, and as was said, 
$\eta^j \in I_p$ for $1 \leq j \leq h$. 

\smallskip
This implies the existence of at least $h = O({\rm log}(p))$ exceptional solutions, 
relatively to the integer $\eta$.
\end{proof}

\section{Removal of the obstruction of minimal \texorpdfstring{$p$}{Lg}-divisibility}\label{section7}

\subsection{The invariants \texorpdfstring{$m_p(u)$}{Lg} and 
\texorpdfstring{$M_p$}{Lg}.}\label{repetitions}
For given $p$, $\theta$ and $u \in [0, p[$, let $m_p(u)$ be the number of $z\in I_p$ 
having a $\theta$-regulator $\Delta_p^\theta(z)$ congruent to $u$ modulo $p$. We denote by 
$$M_p= {\rm max}_{u\in [0, p[} \, \big(m_p(u) \big)$$ 

the {\it maximal number of repetitions}. 
We suppose a part of the conditions of Definition \ref{defidec} for $p$ and $\theta$, namely $f=\delta=1$.

\smallskip
We then obtain a remarkable stability for $M_p$, very regular function of $p$ which can be the subject of the
following heuristic as for Fermat quotients (\S\,\ref{princip}):

\begin{heuristic} \label{heur1} {\it For all $p\geq 2$ and all irreducible $p$-adic character $\theta$ of $G$,
given such that $f=\delta=1$ (cf. Definition \ref{defidec}), the number 
$M_p= {\rm max}_{u\in [0, p[} \, \big(m_p(u) \big)$ of residues $z\in I_p$ having, modulo $p$, the same 
local $\theta$-regulator, is $O({\rm log}(p))$ (see \cite{Grcompl} for more discussion about this). }
\end{heuristic}

As the mean value of $m_p(0)$ is near from 1, the abundant case (i.e., when $m_p(0)=O({\rm log}(p))$)
is as rare as the exceptional case (i.e., when $\Delta_p^\theta(\eta) \equiv 0 \pmod p$,
generating $h=O({\rm log}(p))$ solutions in $I_p$ of the form $\eta^j$, $j=1, \ldots, h$).

\subsection{Numerical experiments.}\label{exp}\ 
Let us give numerical justifications for properties of $m_p(u)$ and $M_p$. 
In the programs and results hereinafter, we start from a very simple numerical value
of $\gamma$ (many experiments have shown a high stability of results regarding this choice) 
and we compute  the set $I_p$ of residues $z$ of the form $ [\gamma^k]_p$, $k=1,\ldots, p-1$, 
then the values $\Delta_p^\theta(z) \pmod p$ which are managed in a list $L$ to determine 
$m_p(0)$ and $M_p=m_p(u_0)$ for a suitable $u_0$. 

\smallskip
We shall take $\gamma = \eta$ if $I_p$ satisfies the conditions mentioned \S\,\ref{Asub2}.

\subsubsection{Cyclic cubic case, \texorpdfstring{$p$}{Lg} inert 
in \texorpdfstring{$\Q(j)$}{Lg} (\texorpdfstring{$j^3=1$, $j\ne 1$}{Lg}), 
\texorpdfstring{$\theta\ne 1$}{Lg}.}\label{cubicin}
In this case, the statistical study of the $\Delta_p^\theta(z)$, for $z \in I_p$, 
is not necessary as we have explained since from the main heuristic \ref{HP}, we would have
${\rm Prob}\big (\Delta_p^\theta(z) \equiv 0 \!\! \pmod p\big ) = \Frac{O(1)}{p^2}$; 
however, one can calculate the values $m_p(0)$ and $M_p$ to see that 
$m_p(0)>0$ is very rare and in order to see what happens for $M_p$. 

\smallskip
Here, for $x = \zeta_7+\zeta_7^{-1}$, let $K=\Q(x)$ be the cubic cyclic field of conductor $7$, 
let $G = \{1, \sigma, \sigma^2\}$, and take $p\equiv -1 \pmod 6$ (in other words $f=2$).

\smallskip
For $\theta\ne 1$, we have 
$$\Delta_p^\theta(z) = \alpha^2+\alpha^{2\sigma}+\alpha^{2\sigma^2} 
-\alpha \alpha^\sigma -\alpha^\sigma \alpha^{\sigma^2} -\alpha^{\sigma^2} \alpha$$ 

(where $\alpha = \alpha_p(z)$), which is rational modulo $p$ (Pr. B-1 of \cite{Grpro}). 

\smallskip
Let $I_p$ be generated by $\gamma = x^2+2$; concerning the $328$ prime numbers 
$p\equiv -1 \pmod 6$, $5999 < p < 11999$, we have $m_p(0) > 0$ only for $5$ values of $p$ 
(i.e., $p=6761, 7937, 8861, 9941, 10739$) and then $323$ cases where $m_p(0) = 0$. 
But all the cases $m_p(0) > 0$ are due to the fact that there exists $d \div p-1$, $d\ne p-1$, such that 
$\gamma^d \equiv \rho \pmod p$, where $\rho$ is a rational;
so, for $z=[\rho^d]_p$, $\Delta_p^\theta(z)$ is trivially zero modulo $p$, and these cases are to be excluded
as explained \S\,\ref{Asub2}.

\smallskip
For $p\equiv 1 \pmod 6$, $6001 < p < 12001$, we shall find $134$ values of $p$, among $327$, for which 
$m_p(0) =0$, and the numbers $m_p(0)\ne 0$ will have higher mean values.
On the other hand, $M_p$ does not seem to depend on the decomposition of $p$ in~$\Q(j)$.

\smallskip
We have extracted the following examples for $p\equiv -1 \pmod 6$; the parameter 
$u_0 \in [0, p[$ furnishes an integer (among several a priori) such that $M_p=m_p(u_0)$
\footnotesize
$$\begin{array}{lllllll}
p = 59999 & n_p = 3 & u_0 = 25910 \\
m_p(0) = 0& M_p = 7& M_p/{\rm log(p)} = 0.63624 \vspace{0.08cm} \\
p = 60017& n_p= 1 & u_0 =51505 \\
m_p(0) = 0 & M_p=7 & M_p/{\rm log(p)} =0.63622\vspace{0.08cm} \\
p = 60029& n_p= 3& u_0 =19677 & \\
m_p(0) = 0& M_p = 7& M_p/{\rm log(p)}=0.63621\vspace{0.08cm} \\
p = 60041& n_p= 3 & u_0 =59841 & \\
m_p(0) = 0& M_p = 8 & M_p/{\rm log(p)}=0.72708 \vspace{0.08cm}\\ 
\end{array} $$
\normalsize 

We have for instance $m_p(0)=0$ and $M_p=8$ for $p = 60041$ ($u_0 = 59841$), 
and we obtain the following residues 
$z=[\gamma^j]_p$ solutions to $\Delta_p^\theta(z) \equiv 59841 \pmod p$)
\footnotesize
$$\begin{array}{rrrr}
\hbox{exponent\ \,} j & \hspace{3cm} \hbox{residues\ \,} \big[\gamma^k \big]_p \vspace{0.15cm} \\
 12869 & -17167 \, x^2 + 1730 \, x + 28097 \\
 31327 & 17781 \, x^2 + 4775 \, x + 25387 \\
 32191 & 3615 \, x^2 - 27037 \, x - 25973 \\
 39129 & 6079 \, x^2 + 24215 \, x + 18753 \\
 44870 & -11178 \, x^2 + 24638 \, x + 12843 \\
 54374 & 3053 \, x^2 - 24995 \, x - 12010 \\
 56394 & -3461 \, x^2 + 16186 \, x + 7608 \\
 56651 & -19244 \, x^2 - 9845 \, x + 3277
\end{array} $$
\normalsize

\subsubsection{Cyclic cubic case, \texorpdfstring{$p$}{Lg} splitted in 
\texorpdfstring{$\Q(j)$}{Lg} (\texorpdfstring{$j^3=1$, $j\ne 1$}{Lg}).}\label{cubicdec}

We then have $p\equiv 1 \pmod 6$, i.e., $f=1$. There are two $p$-adic $\theta$-regulators 
$$\Delta_p^\theta(z) = \alpha+r^2\,\alpha^{\sigma}+r\alpha^{\sigma^2}, $$ 

where $\alpha = \alpha_p(z)$ and where $r$ is one of the two elements of order 3 modulo $p$. 

\smallskip
In that case, $\Delta_p^\theta(z)$ is a ``resolvant of Hilbert modulo $p$''
which is congruent to a rational modulo the prime ideal $\mathfrak p$ associated with $\theta$. 
But Pr. B-2 of \cite{Grpro} gives $\Delta_p^\theta(z) = u_2 x^2+u_1 x+u_0$ which supposes that we use
a congruence of the form $x \equiv R \pmod {\mathfrak p}$ in order to obtain 
$\Delta_p^\theta(z) \equiv u \pmod p$.
We proceed in a different way: to get a rational, we multiply $\Delta_p^\theta(z)$ by
$$H:=x+r x^\sigma+r^2 x^{\sigma^2} \pmod p$$ 

which serves as a ``conjugate resolvant'' once for all; it is not divisible by $p$.

\smallskip
We generate $I_p$ by means of $\gamma=x^2+2$. We have extract the following examples
\footnotesize
$$\begin{array}{llllll}
p = 60037 & n_p = 3& u = 26443 \\
m_p(0) = 0 & M_p = 8 & M_p/{\rm log(p)}=0.72709 \vspace{0.05cm} \\
p = 60091 & n_p = 3 & u = 32679 \\
m_p(0) = 1 & M_p = 7&M_p/{\rm log(p)}=0.63615 \vspace{0.05cm} \\
p = 60103 & n_p = 1& u = 22560 \\
m_p(0) = 0 & M_p = 9 & M_p/{\rm log(p)}=0.81789 \vspace{0.05cm}\\
p = 60127 & n_p = 3& u = 55712 \\
m_p(0) = 1 & M_p = 7 & M_p/{\rm log(p)}=0.63612 \vspace{0.05cm} %\\
\end{array} $$
\normalsize

\subsubsection{Example of \texorpdfstring{$M_p=m_p(u)$}{Lg} for 
\texorpdfstring{$u>0$}{Lg}.}\label{5011}

Still in $K = \Q(x)$, with $x= \zeta_7+\zeta_7^{-1}$, we consider $\gamma=-5\,x^2 + 2\,x + 3$ 
(Pr. B-3 of \cite{Grpro}).
For $p = 5011$, we find a maximal number $M_p=7$ of $z\in I_p$ such that
$\Delta_p^\theta(z) \equiv u_0 \pmod p$ for $u_0 = 418$
(we then have $m_p(0) = 1$ and $M_p/{\rm log(p)}=0.82165$)
\footnotesize
$$\begin{array}{rrrr}
\hbox{exponent\ \,} j & \hspace{2.5cm} \hbox{residues\ \,} \big[\gamma^j \big]_p \vspace{0.15cm} \\
1233 & 2043 \, x^2 - 540 \, x - 359 \\
1297 & 810 \, x^2 + 74 \, x + 1078 \\
1932 & -1415 \, x^2 + 962 \, x - 1352 \\
2465 & 577 \, x^2 + 1380 \, x + 1727 \\
2941 & -1735 \, x^2 - 172 \, x + 1553 \\
3848 & 1168 \, x^2 - 816 \, x + 70 \\
4339 & -320 \, x^2 - 426 \, x + 468
\end{array} $$
\normalsize

\subsubsection{Diedral case of order 6.}\label{d6}
We consider the field $K=\Q(j, \sqrt[3]{2})$ (where $j$ denotes a cubic root of unity), with Galois group
$G=D_6 = \{1, \sigma, \sigma^2, \tau, \tau\sigma, \tau\sigma^2 \}$, 
and the unique irreducible $p$-adic character 
$\theta$ of degree 2 for which (with $\alpha := \alpha_p(z)$)

\medskip  
$\Delta_p^\theta (z) = \hbox{$\frac{1}{\sqrt {-3} }$}
\big(\alpha^{2} +\alpha^{2}{}^{\sigma} +\alpha^{2}{}^{\sigma^2} 
-\alpha^{2}{}^{\tau} -\alpha^{2}{}^{\tau\sigma} -\alpha^{2}{}^{\tau\sigma^2} 
- \alpha \alpha^{\sigma} - \alpha^{\sigma} \alpha^{\sigma^2} -\alpha^{\sigma^2} \alpha$

\hfill $+ \alpha^{\tau} \alpha^{\tau\sigma}+ \alpha^{\tau\sigma} \alpha^{\tau\sigma^2} +
\alpha^{\tau\sigma^2} \alpha^{\tau} \big)$.

\smallskip
We use here $I_p$ generated by $\gamma= 2\,x^5+2\,x^3+x -1$ (Pr. B-4 of \cite{Grpro})
\footnotesize
$$\begin{array}{llllll}
p = 3559 & u_0 = 2946 & \\
m_p(0) = 1 & M_p = 6 & M_p/{\rm log(p)}=0.73374 & \vspace{0.15cm} \\
p = 3571 & u_0 = 2286 & \\
m_p(0) = 1 & M_p = 5& M_p/{\rm log(p)}=0.61120 & \vspace{0.15cm} \\
p = 3581 & u_0 = 1 & \\
m_p(0) = 0 & M_p = 7 & M_p/{\rm log(p)}=0.85539 &\vspace{0.15cm} \\
p = 3583 & u_0 = 1852 & \\
m_p(0) = 0 & M_p = 6 & M_p/{\rm log(p)}=0.73314 &\vspace{0.15cm} 
\end{array} $$
\normalsize

We consider $\gamma= x^5+2\,x^4 -2\,x^3 -x +1$.
For $p = 1709$, we obtain $m_p(0) = 1$ and $M_p=6$ for $u_0 = 487$ 
(we have $M_p/{\rm log(p)}=0.80605$); 
whence the array of the $z=\big[\gamma^j \big]_p$ such that 
$\Delta_p^\theta (z) \equiv 487 \pmod p$
\footnotesize
$$\begin{array}{rrrr}
\hbox{exponent\ \,} j & \hbox{residues\ \,} \big[\gamma^j \big]_p \vspace{0.15cm} \\
 51 & -179 \, x^5 + 718 \, x^4 + 739 \, x^3 + 688 \, x^2 + 553 \, x - 159 \\
 81 & -212 \, x^5 - 730 \, x^4 - 634 \, x^3 + 849 \, x^2 - 161 \, x - 556 \\
 759 & -649 \, x^5 + 324 \, x^4 - 729 \, x^3 + 675 \, x^2 - 423 \, x + 149 \\
 1079 & 552 \, x^5 - 364 \, x^4 + 136 \, x^3 + 52 \, x^2 + 799 \, x + 335 \\
 1291 & 651 \, x^5 + 584 \, x^4 + 334 \, x^3 + 263 \, x^2 + 437 \, x + 624 \\
 1567 & 99 \, x^5 + 566 \, x^4 - 292 \, x^3 + 152 \, x^2 + 529 \, x - 645
\end{array} $$
\normalsize

 We shall examine in which way it is possible to have $m_p(0) = O({\rm log}(p))$ (abundant solutions) 
apart from the case of exceptional solutions, important point to justify the existence of a binomial probability law.

\subsubsection{Cas where \texorpdfstring{$m_p(0) = O({\rm log}(p))$}{Lg} apart from 
the exceptional case.}\label{nonex}

We intend to give numerical examples of prime numbers $p$ for which $I_p$ 
(generated by $\gamma = \eta \ll p$)
has $m_p(0) = O({\rm log}(p))$ solutions $z \in I_p$ to $\Delta_p^\theta(z) \equiv 0 \pmod p$,
in the case where these solutions {\it are not of the form}

\centerline{$\mu^j = [\mu^j ]_p$, $1 \leq j \leq h'= 
\Big\lfloor \dsfrac{ {\rm log} (p-1) - {\rm log} (2\,\Gamma(K))}{{\rm log} (c_0(\mu))} \Big\rfloor$ }

\smallskip
of the exceptional case when $\Delta_p^\theta(\mu) \equiv 0 \pmod p$, $\mu \ll p$. 
The numerical experiments show that this is as rare as in the exceptional 
case and we shall conclude on these various cases in \S\,\ref{remfond}.

\smallskip
(i) Case $G=1$.
Although this has been studied in \cite{Grqf} (array of \S\,4.3.1 giving pairs $(p, m_p(0))$ 
with $m_p(0)\geq 6$), by comparison one may see again the case of Fermat quotient for which
we always have $I_p = [1, p[$. 

\smallskip
We obtain the following array in which we have fixed $a=14$, to find cases $q_p(a) \equiv 0 \pmod p$ 
(exceptional solutions $a^j$) and the cases $m_p(0)=O({\rm log}(p))$ (abundant solutions $z$), 
for primes $p$ such that $3 \leq p \leq 10007$ (Pr. B-0 of \cite{Grpro})
\footnotesize
$$\begin{array}{llllll}
\hbox{$p$} \hspace{1.8cm} & \hbox{$u \in [0, p[$ tels que $q_p(z)=u$ for 
$M_p=m_p(u)$} & M_p & m_p(0) \vspace{0.20cm} \\ 
 p = 11 & abundant \ (z=3, 9) & M_p = 2 & m_p(0) = 2 \\
 & u = 0, 5 & & \vspace{0.10cm} \\
 p = 29 & exceptional \ (z=14) & M_p = 3 & m_p(0) = 1 \\ 
 & u = 24, 16, 1 & & \vspace{0.10cm} \\
 p = 353 & exceptional \ (z=14, 196) & M_p = 6 & m_p(0) = 2 \\
 & u = 297, 275 & & \vspace{0.10cm} \\
 p = 653 & abundant \ (z=84, 120, 197, 287, 410) & M_p = 5 & m_p(0) = 5 \\
 & u = 0, 99, 360, 241, 353, 617, 119, 399 & & \vspace{0.10cm}  \\
 p = 4909 & abundant \ (z=2189, 2234, 2406, 3266, 4649) & M_p = 5 & m_p(0) = 5 \\
 &u = 0, 4651, 2785, 3967, 648, 3544, 3322, 2381, \\
& 1843, 3465, 1089, 1483, 4171 && \vspace{0.10cm} \\
 p = 5107 & abundant \ (560, 1209, 1779, 2621, 4295, 4361) & M_p = 6 & m_p(0) = 6 \\
 & u =0, 2705, 4159 && 
\end{array}$$
\normalsize

(ii) Case $G=C_3$ (Pr. B-5 of \cite{Grpro}).
We use the cubic cyclic field of conductor $7$ and $I_p$ generated by $\gamma=x^2+x+2$.

\smallskip
For $p =2053$ (the least example with $M_p=m_p(0)=7$) we obtain the following residues 
$z=[\gamma^j]_p$ such that $\Delta_p^\theta(z) \equiv 0 \pmod p$ (for the unique $u_0=0$)
\footnotesize
$$\begin{array}{rrr}
\hbox{exponent\ \,} j & \hbox{residues\ \,} \big[\gamma^j \big]_p\vspace{0.10cm} \\ 
186 & 871 \, x^2 - 930 \, x + 496 \\ 
500 & 57 \, x^2 + 272 \, x + 478 \\
559 & -691 \, x^2 - 1003 \, x - 881 \\
1399 & 258 \, x^2 + 1002 \, x - 349 \\
1870 & -375 \, x^2 - 212 \, x + 240 \\ 
1981 & -464 \, x^2 + 818 \, x - 783 \\ 
2034 & 121 \, x^2 + 610 \, x + 524
\end{array}$$
\normalsize

 The example is clear since the exponents $j$ are not the first powers of a $\mu \in I_p$, $\mu \ll p$, 
and since there are no other solutions.

\smallskip
For $I_p$ generated by $\gamma=2\,x^2+x+3$ and $p=1987$, we have $M_p=m_p(0)=5$ 
for $u= 1026, 454, 282, 180, 0, 1734, 117, 325, 1225$
and an analogous array of residues for $u_0=0$.

\smallskip
For $\gamma=2\,x^2+x+2$, $p=37, 307, 2347$ give non exceptional abundant solutions. 
Only $p=79$ leads to a mixed case ($M_p=m_p(0)=4$), with $u=0, 71$ and the array 
of residues for $u_0=0$
\footnotesize
$$\begin{array}{rrrr}
\hbox{exponent\ \,} j & \hbox{residues\ \,} \big[\gamma^j \big]_p\vspace{0.10cm} \\ 
1 & 2 \, x^2 + x + 2 \\
2 & 17 \, x^2 + 8 \, x + 4 \\
20 & 19 \, x^2 - 11 \, x + 15 \\
35 & -35 \, x^2 - 33 \, x + 19
\end{array}$$
\normalsize

(iii) Diedral case of degree 6 (Pr. B-6 of \cite{Grpro}). The character $\theta$ of degree 2 allows us to confirm
 the previous computations.
For abundant solutions, the $m_p(0)\approx M_p$ maximum, equal to 6, 
is given by the following example, where $I_p$, for $p=331$, is generated by the integer
$\gamma = -x^5+x^4-x^3-x^2 +1$
\footnotesize
$$\begin{array}{rrr}
\hbox{exponent\ \,} j & \hbox{residues\ \,} \big[\gamma^j \big]_p\vspace{0.10cm} \\ 
48 & 59 \, x^5 - 46 \, x^4 - 87 \, x^3 + 141 \, x^2 + 158 \, x + 40 \\
102 & -61 \, x^5 - 114 \, x^4 + 119 \, x^3 + 11 \, x^2 - 125 \, x - 120 \\
138 & -123 \, x^5 - 122 \, x^4 - 79 \, x^3 - 61 \, x^2 + 22 \, x - 71 \\
155 & 91 \, x^5 + 100 \, x^4 + 136 \, x^3 + 138 \, x^2 + 152 \, x + 147 \\
180 & 152 \, x^5 - 8 \, x^4 - 59 \, x^3 - 165 \, x^2 + 92 \, x - 131 \\
322 & 49 \, x^5 - 158 \, x^4 - 13 \, x^3 - 14 \, x^2 - 33 \, x - 23
\end{array}$$
\normalsize

For $\gamma = -x^5-x^4+x^3-x^2-x +1$, $p=379$, we have a case of abundant solutions 
with $M_p=m_p(0)=5$ and the following residues
\footnotesize 
$$\begin{array}{rrr}
\hbox{exponent\ \,} j & \hbox{residues\ \,} \big[\gamma^j \big]_p\vspace{0.10cm} \\ 
49 & -147 \, x^5 - 39 \, x^4 - 73 \, x^3 + 138 \, x^2 + 40 \, x + 129 \\
104 & -169 \, x^5 - 105 \, x^4 - 45 \, x^3 - 180 \, x^2 - 174 \, x + 7 \\
149 & -91 \, x^5 + 48 \, x^4 - 155 \, x^3 + 62 \, x^2 + 183 \, x + 35 \\
223 & -178 \, x^5 - 14 \, x^4 - 101 \, x^3 + 150 \, x^2 - 189 \, x + 107 \\
304 & -103 \, x^5 + 131 \, x^4 + 3 \, x^3 + 165 \, x^2 + 140 \, x + 189
\end{array}$$
\normalsize

We have the following examples in the intervals of variations of the $81$ values of 
$\gamma = ax^5+bx^4+ cx^3+dx^2+ex +1$ (coefficients in $\{-1, 0, 1\}$), of the program 
for $2000 \leq p \leq 2500$):

\smallskip
a) For $\gamma = x^5 -x^4-x +1$ and $p=2441$, we have $m_p(0)=2$, $M_p=6$ (with $u_0=1426$) 
for an exceptional solution (but only for $h=1$) and the array
\footnotesize
$$\begin{array}{rrr}
\hbox{exponent\ \,} j & \hbox{residues\ \,} \big[\gamma^j \big]_p\vspace{0.15cm} \\ 
 1 & x^5 - x^4 - x + 1\\
915 & -442 \, x^5 - 129 \, x^4 - 125 \, x^3 - 651 \, x^2 - 645 \, x + 376
\end{array}$$
\normalsize

For $u_0=1426$ we obtain the following array of the residues $z$ giving $M_p=m_p(u_0)$
\footnotesize
$$\begin{array}{rrr}
\hbox{exponent\ \,} j & \hbox{residues\ \,} \big[\gamma^j \big]_p\vspace{0.15cm} \\ 
1839 & -169 \, x^5 - 867 \, x^4 - 402 \, x^3 - 891 \, x^2 - 357 \, x - 680 \\ 
2034 & 35 \, x^5 - 939 \, x^4 - 181 \, x^3 + 388 \, x^2 - 841 \, x - 226 \\ 
2054 & 449 \, x^5 - 212 \, x^4 + 1097 \, x^3 - 651 \, x^2 + 1191 \, x - 478 \\ 
2171 & 688 \, x^5 - 525 \, x^4 - 635 \, x^3 + 334 \, x^2 + 181 \, x - 783 \\ 
2194 & -909 \, x^5 + 335 \, x^4 - 1136 \, x^3 - 1033 \, x^2 - 970 \, x + 557 \\ 
2353 & 780 \, x^5 - 1126 \, x^4 + 968 \, x^3 - 264 \, x^2 - 294 \, x - 107
\end{array}$$
\normalsize

b) For $\gamma = x^5 -x^4+x^3-x^2 +1$ and $p=2441$, we have $M_p=m_p(0)=5$
 ($u=2158, 2057, 724, 359, 0, 717$) for abundant solutions and an analogous array.

\medskip
c) For $\gamma = x^5 -x^3-x +1$ and $p=2087$, we also have $M_p=m_p(0)=5$
 (with $u= 1335, 950, 670, 1840, 506, 1541, 1102, 280, 1973, 60, 0$) for abundant solutions.

\medskip
d) Among the $81$ generators $\gamma$ one still find 4 cases of exceptional solutions and 4 cases of 
distinct abundant solutions.

\subsubsection{Conclusions -- Fundamental Remarks}\label{remfond} Examine the main features 
of the notions of exceptional and abundant solutions. The number $\eta$ is fixed and $p\to\infty$.

\smallskip
(i) {\it Exceptional solutions.}
If $\Delta_p^\theta(\eta)\equiv 0 \pmod p$, this generates at least $h$
solutions $z_j =\eta^j= [\eta^j]_p \in I_p$, $j = 1, \ldots, h$, and this yields $m_p(0) \geq h = O({\rm log}(p))$ 
(these solutions are also abundant). 
If we admit that $M_p=O({\rm log}(p) )$, we obtain $M_p \geq m_p(0)\geq h$.
We shall often have $M_p > m_p(0) \geq h$ taking into account that 
$M_p=m_p(u_0)$, $u_0 \in [0, p[$, and that $u_0=0$ is less probable, 
even if several $u$ realise $M_p$; moreover, $M_p>m_p(0)$, when $u_0 \ne 0$,
can be explained by the fact that if $\Delta_p^\theta(z) \equiv u_0 \pmod p$, then, in general,
$\Delta_p^\theta(\eta^k z) \equiv u_0 \pmod p$ (obvious in the case of linear $\Delta_p^\theta(\mb)$ in 
the conjugates of $\alpha_p(\mb)$; see \cite[\S\,4.2.2 ($\delta$)]{Grqf} for Fermat quotients).

\smallskip
(ii)  {\it Abundant solutions.}
If $\Delta_p^\theta(\eta) \not\equiv 0 \pmod p$ and $m_p(0) = O({\rm log}(p))$ we then have
$O({\rm log}(p))$ solutions $z'_i \in I_p$, $i = 1, \ldots, h' := m_p(0)$, where in that case, the solutions $z'_i$ 
are a priori uniformly distributed in $I_p$ (recall that from \cite{H-B}, Fermat quotients 
are uniformly distributed modulo $p$ and that this property is probably general).

\smallskip
(iii)  {\it Conclusion.} The exceptional case may be seen as the case where, 
{\it by accident}, $\eta$ is part of the solutions $z'_j$, in which case {\it we have necessarily} 
$z'_1=\eta$, $z'_2=\eta^2$, \ldots, $z'_h=\eta^h$, with additional $z'_i$, 
without one can say that the successive powers of $\eta$ establish some relations of probabilistic dependence.
Moreover, we shall obtain ``mixed cases'' (i.e., when there exists $\mu \ne \eta$ in $I_p$, $\mu \ll p$, such that 
$\Delta_p^\theta(\mu) \equiv 0 \pmod p$ giving $h' \ll O({\rm log}(p))$ solutions, in part exceptional).

\smallskip
It follows from all this, that the two cases (i) and (ii) are of similar probabilities, the exceptional case 
being less probable by definition in which case, only the consideration of the
``abondant'' case is coherent with the existence of a classical law of probability for the set of 
solutions $z'_i$ which are not subject to any condition. 

\smallskip
In other words, the ``exceptional'' case would not be particular, despite the
appearances, and it would be liable to the same probability as for Fermat quotients 
(\cite[\S\,4.3.2]{Grqf}), a probability which becomes (for instance) $O(\frac{1}{p^2})$ 
for $p > p_0$ very large, which we will analyse again.

\begin{remark}
The number $\eta$ being given, we intend to compare the probability to have a 
prime number $p$ such that $\Delta_p^\theta(\eta) \equiv 0 \pmod p$ (exceptional solutions),
with the probability to have $\Delta_p^\theta(\eta) \equiv u \pmod p$, {\it for fixed $u$ in $\N$, 
independently of~$p$} (which is the case of $u=0$). 
The numerical aspect needs to take $u$ ``fixed and small'' and to search the prime numbers $p$ such that
$\Delta_p^\theta(\eta) \equiv u \pmod p$. It is then found the same degree of rarity whatever the choice of $u$.

\smallskip
For instance if $\eta=x^2-3\,x+2$ ($x=\zeta_7+\zeta_7^{-1}$, Pr. B-7 of \cite{Grpro}), in the interval
$7< p \leq 60000001$, we get the rare pairs of solutions 

\smallskip
$(p, u)=(61, 0)$, $(5419, 0)$, $(19, 1)$, $(37, 2)$, 
$(3229, 3)$, $(43, 4)$, $(31, 5)$, $(613, 5)$, $(\emptyset, 6)$, $(79, 7)$, $(42712981, 7)$. 

\smallskip
We can use negative $u$ and we get similar results, as $(607, -1)$, $(143137, -1)$.
\end{remark}

\subsection{On the existence of a binomial law for \texorpdfstring{$m_p(u)$}{Lg}.}\label{loi} 
Besides the previous justifications, we can complete the analysis in the following 
quantitative manner which results from a very simple computation, given for the case 
of Fermat quotients $q_p(z)$, $z\in [1, p[$, as well as
for the case of local regulators $\Delta_p^\theta(z)$, $z\in I_p$, 
for the group $D_6$ and $\theta$ of degree 2 (Pr. B-8, B-12, B-13 of \cite{Grpro}, 
the last program testing more general probabilities).

\smallskip
In the two cases, we have computed the mean (under a great lot of prime numbers) 
the proportions $C/N$, where for $p$ fixed, $C$ is the number of values $u \in [0, p[$ 
such that there exists at least a $z\in [2, p-1[$ (resp. $z \in I_p$) such that 
$q_p(z) \equiv u \pmod p$ (resp. $\Delta_p^\theta(z) \equiv u \pmod p$). 

\smallskip
The very remarquable approximation of the result with $1-e^{-1} \approx 0.632120$ leads to the following 
conjecture/heuristic:

\begin{conjecture} \label{heur21} {\it Let $K/\Q$ be Galois of degree $n$ and of Galois group $G$.
We assume to study the case $f = \delta=1$ (i.e., $p$ totaly split in the field $C$ of values 
of $\varphi \div \theta$, ${\mathcal L}^\theta \simeq V_\theta$, cf. Definition \ref{defidec}) 
for $p$ and for the irreducible $p$-adic character $\theta$.

\smallskip  
Then the mean value of the proportion of $u \in [0, p[$ of the form $\Delta_p^\theta(z) \pmod p$, $z \in I_p$ 
(Definition \ref{defresidus}), is equal to $1-e^{-1} \approx 0.632120$, for $p \to \infty$. }
\end{conjecture}

The program for $C_3$ gives the value $0.632133$ and that for $D_6$ gives $0.631711$.
As we recall at the point (iv) below, it is also the probability (assuming a binomial law) of the existence
of at least one solution $z \in I_p$ to $\Delta_p^\theta(z) \equiv u \pmod p$ for fixed~$u$. 

\smallskip
Then in a complement (unpublished, accessible via \cite{Grcompl}), 
we have estimated, in various manner, the mean value 
of $M_p$ for the binomial law of probability with parameters $\Big(p-1, \hbox{$\Frac{1}{p}$}\Big)$, 
when $p\to\infty$ (we ignore if any theoretical result is known).

\subsubsection{Second principal heuristic.}\label{consheur}
The previous arguments suggest the existence of a binomial law 
with parameters $\Big(p-1, \hbox{$\Frac{1}{p}$} \Big)$, because we can consider that we realise the
$p-1$ ``random selection'' $z \in I_p$ for which we ask how many times we get the event 
$\Delta_p^\theta(z)\equiv u \pmod p$, $u\in [0, p[$ given.
 The second parameter $\Frac{1}{p}$ is an approximation of 
 ${\rm Prob}\big(\Delta_p^\theta(z)\equiv u \pmod p\big)$.
 
\smallskip
In fact one can verify that any minor modification of these parameters does not change the conclusion.

\begin{heuristic} \label{heur2} {\it Let $K/\Q$ be Galois, of degree $n$ and of Galois group $G$.
We assume to study the case $f = \delta=1$ (cf. Definition \ref{defidec}) for $p$ and for the irreducible
 $p$-adic character $\theta$ of $G$. Let $I_p$ as defned in Definition \ref{defresidus}.

\smallskip  
Let $u \in [0, p[$ be fixed. Let $m \in [0, p[$, $m \ll p$; then the probability to have at least
$m$ values $z_1, \ldots, z_m \in I_p$ such that $\Delta_p^\theta(z_j)\equiv u \pmod p$ for 
$j=1, \ldots, m$, is given by the expression

\centerline{${\rm Prob}\big(m_p(u) \geq m\big) = \Frac{1}{p^{p-1}}
 \sm_{j=m}^{p-1} \hbox{$\binom{p-1} {j}$} (p-1)^{p-1-j}$. }}
\end{heuristic}

 We resend to \cite[\S\,4.4]{Grqf} for identical calculations leading to the following facts,
from the simpler formula

\centerline{ ${\rm Prob}\big(m_p(u) \geq m \big)=
1 -\Big (1-\Frac{1}{p} \Big)^p \Frac{p}{p-1}\,\sm_{j=0}^{m-1} \Frac{1}{(p-1)^j} 
\times \hbox{$\binom{p-1} {j}$}$, }

(i) $\ \Frac{1}{p^{p-1}} \sm_{j=m}^{p-1} \hbox{$\binom{p-1} {j}$} (p-1)^{p-1-j} < 
\Frac{1}{p^m} \, \hbox{$\binom{p-1} {m}$}$ for all $m \leq p-1$.

(ii) $\ {\rm Prob}\big(m_p(u) \geq m \big) \approx 1 - 0.3678 
\times \sm_{j=0}^{m-1} \Frac{1}{(p-1)^j} \times \hbox{$\binom{p-1} {j}$}$.

\smallskip
(iii) The probability to have 0 solutions is near from $e^{-1} \approx 0.3678$. 

\medskip
(iv) The probability to have at least one solution is near from $1-e^{-1} 
\approx 0.63212$; for at least 3 (resp. 4) solutions, 
we obtain $0.0803$ (resp. $0.0189$).

\medskip
For an experimental confirmation, see the Pr. 14 of \cite{Grpro}.

\medskip
In the context of minimal $p$-divisibility, we obtain the following results, where 
$$h = \dsfrac{ {\rm log} (p-1) - {\rm log} (2\,\Gamma(K))}{{\rm log} (c_0(\eta))}, $$

with $c_0(\eta) = {\rm max}_{\sigma \in G}(\vert \eta^\sigma \vert )$ (Lemma \ref{lem3}):

\begin{lemma} \label{encadre}
(i) We have for $p\to\infty$ the inequalities (\cite[Lemma 4.6]{Grqf})

\centerline{ ${\rm exp} \Big(-1 + \Frac{1}{p} \big (h + \Frac{1}{2} \big )\, \Big)<
\Frac{\Frac{1}{p^{p-1}} \cdot\sm_{j=h}^{p-1} \hbox{$\binom{p-1} {j}$} (p-1)^{p-1-j} } {\Frac{1}{p^{h}}
\cdot \hbox{$\binom{p-1} {h}$} } \leq 1$. }
 
(ii) It follows ${\rm Prob}\big(\Delta_p^\theta(\eta) \equiv 0 \pmod p \big) < C_\infty(\eta) \times 
\Frac{1}{p^h} \, \hbox{$\binom{p-1} {h}$}$ for $p\to\infty$, where $C_\infty(\eta)$ is between 
$e^{-1} \approx 0.36788$ and $1$.
\end{lemma}

\begin{lemma}\label{theoremconv} 
 The series $\sm_{p> 2}\ \dsfrac{1}{p^h}\,\cdot  \binom{p-1} {h}$ is convergent (\cite[Lemma 4.7]{Grqf}).
\end{lemma}

So we obtain the Theorem \ref{heur3} which is modified,
compared with the case of Fermat quotient, only by the effective constant $c_0(\eta)$ 
and the term $O(1)$ which can be precised.

\section{\texorpdfstring{$p$}{Lg}-adic conjectures}\label{section8}

\subsection{Introduction}\label{prem}
 The previous general result leads to several consequences, 
 or interpretations, that we shall call {\it Conjectures} 
 insofar as we consider that, under the second principal 
 Heuristic \ref{heur2}, any situation leads to the application
 of the Borell--Cantelli principle. 
 These conjectures come from the suitable use of a $p$-adic regulator of an 
 $\eta \in K^\times$ and of its $\theta$-components, for $p \to \infty$, 
 knowing that it is always possible to suppose $\eta \in Z_K$ for the 
 Archimedean aspects of the probabilistic reasoning 
($\theta \ne 1$, cf. Lemma \ref{integer}).

\smallskip
In algebraic number theory one speaks of {\it``for almost all prime number $p$''} to
mean {\it ``all prime number $p$ except a finite set $\Sigma$''}. 
But other weaker definitions are possible
in the probabilistic number theory (cf. \cite[Chap.\,III.3.1]{T}).
The statements of this section will be given under the strong form (algebraic). 

\smallskip
Whatever the exacness or not of our heuristics,
these conjectures are given independently and many seem to be very natural and credible.

\subsection{Local interpretation of 
\texorpdfstring{$\Delta_p^{\theta}(\eta) \equiv 0 \pmod p$}{Lg}}\label{interlocale}

Let $\eta \in K^\times$ be such that the multiplicative $\Z[G]$-module $F$ generated 
by $\eta$ is of $\Z$-rank $n$ and let $\theta$ be an irreducible $p$-adic character of 
$G$ such that $\Delta_p^{\theta}(\eta) \equiv 0 \pmod p$.
From the Corollary \ref{coro25} to Theorem \ref{theo24}, this is equivalent to 
the existence of a non trivial $\theta$-relation
${U_\theta} := \sm_{\nu \in G} u (\nu)\,\nu^{-1} \in {\mathcal L}^\theta$ such that 
$\eta^{U_\theta} \in \prd_{v \div p} K_v^{\times p}$.

We shall consider this writing as a property of ``partial local $p$th power at $p$'' of $\eta$, 
according to the following definition:

\begin{definition}\label{defi28} Let $\eta \in K^\times$. We assume that the multiplicative 
$\Z[G]$-module $F$ generated by $\eta$
is of $\Z$-rank $n$. Let $p$ be a large enough prime number and let

\smallskip
\centerline{$F_{(p)} :=\Big \{\eta_0 \in F, \ \eta_0 \in \prd_{v \div p} K_v^{\times p} \Big \} . $}

We shall say that $\eta$ is a {\it partial local $p$th power at $p$} if ${\rm dim}_{\F_p}(F/F_{(p)}) < n$.
\end{definition}

In the context of this definition, we have the exact sequence
$$0 \longrightarrow {\mathcal L}(\eta) \longrightarrow \F_p[G] \longrightarrow F/F_{(p)} \to 1, $$
obtained by associating with $U \in \F_p[G]$ the element $\eta^{{(p^{n_p}-1)}\, .\, U'}$ modulo $F_{(p)}$
where $U'$ is any representative of $U$ in $\Z[G]$.

\begin{remarks}\label{rema260} Since by assumption $F$ is of $\Z$-rank $n$ and 
whithout $p$-torsion (for all $p$ large enough), we have $F/F^p \simeq \F_p[G]$; in particular
$${\rm dim}_{\F_p}(F/F^p)^{e_\theta} = f\,\varphi(1)^2, $$ 

for all $\theta$, where $f$ is the residue degree of $\theta$, cf. \S\,\ref{defi01}\,(ii).
This yields the following:

\smallskip
(i) The condition ${\rm dim}_{\F_p} (F/F_{(p)}) <n$ is equivalent to the existence of a non trivial 
$\theta$-relation $U_\theta \in e_\theta \Z_{(p)}[G]$ modulo $p$, such that 
$\eta^{U_\theta}$ is in $F_{(p)}$ and not a global $p$th power in $K^\times$ 
because $F\cap K^{\times p} = F^p$ for $p$ large enough.
Indeed, we have $F \subseteq E^S$ (group of $S$-units) where 
$S$ is a suitable finite set of prime ideals of $K$. If we suppose $p$ 
large enough such that $p$ does not divide
the orders of the torsion groups  ${\rm tor}_\Z(E^S)$ and ${\rm tor}_\Z(E^S/F)$, 
then $F$ is a direct facteur in $E^S$
and $E^S = F \oplus H$; if $\eta' \in F$ is such that $\eta' = x^p$, $x \in K^\times$, then 
$x\in E^S$ and it may be written $x =x_F \times x_H$, 
whence $x_H^p = 1$, $x_H=1$ and $\eta' = x_F^p\in F^p$.

\smallskip
(ii) We have 
$$(F/F_{(p)})^{e_\theta} \simeq e_\theta\F_p[G]/{\mathcal L}^\theta 
\simeq \F_p^{(\varphi(1) - \delta) \cdot f \cdot \varphi(1)}, $$

since the dimension is $t\, f \varphi(1)$, $0 \leq t \leq \varphi(1)$, 
which leads to the relation $t = \varphi(1) - \delta$ since 
${\mathcal L}^\theta \simeq\delta V_\theta$ is of $\F_p$-dimension $\delta f \varphi(1)$.

\smallskip
(iii) To say that $\eta \in F_{(p)}$, is to say that $F_{(p)} = F$, hence ${\mathcal L} = \F_p[G]$, of probability 
$\Frac{O(1)}{p^n}$, a case to be ignored for $n>1$ and $p \to\infty$.
\end{remarks}

In the forthcoming Sections \ref{sub35} and \ref{sub380} we shall look at the reciprocal 
aspect of these local $p$th power properties for almost $p$.

\subsection{Case of the characters of order 1 or 2}\label{sub35} 
 We return to known particular cases (see \S\,\ref{sub7}).

\subsubsection{Case of a rational}\label{sub366}
We consider $K=\Q$ with a rational $a \in \Q^\times$, $a \ne \pm 1$. If $p$ is 
an odd prime number prime to $a$, we have the elementary following result 
which is a particular case of the above (for $\theta=1$ and $U_\theta = 1$): 

\begin{lemma}\label{lemm29}
The Fermat quotient $\Frac{a^{p-1}-1}{p}$ of $a$ is zero modulo $p$ if and only if $a \in \Q_p^{\times p}$.
\end{lemma}

But we know, from a result of Silverman \cite{Si} when $a \in \N$, $a \geq 2$, 
that under the $ABC$ conjecture the set of primes $p$ 
such that $a^{p-1} \not\equiv 1 \pmod {p^2}$ is infinite.\,\footnote{\,Silverman 
proves that for all integer $a \geq 2$, the set of these prime numbers 
$p\leq x$ is of cardinal $\geq c \,{\rm log}(x)$. 
This result has been extended by Graves and Murty in \cite{GM} to the $p \equiv 1 \pmod k$, for all
 fixed $k\geq 2$, in which case, the set of these $p \leq x$ is of cardinal 
 $\geq c \, \frac{ {\rm log} (x) }{ {\rm log} ( {\rm log} (x))}$, still under the $ABC$ conjecture. }

The statistical study shows that this result is a very weak form of the reality. 
In other words, we have the following very reasonable conjectural property:

\begin{conjecture}\label{conj291}
 Let $a \in \Q^\times$; if $a \in \Q_p^{\times p}$ for all prime $p$ except a finite number then $a = \pm 1$. 
\end{conjecture}

We may consider this statement as a very particular local--global principle in comparison to those existing in 
class field theory (then purely algebraic as the ``Hasse principle'' for powers, recalled in
Proposition~\ref{prop32}). 

\smallskip
We might call it a {\it Diophantine local--global principle} in the perspective that 
``$a \in \Q_p^{\times p}$ for almost all $p$'' would be equivalent to ``$a \in \Q^{\times p}$ for almost all $p$''.

\subsubsection{Case of a unit of a quadratic field \texorpdfstring{$K=\Q(\sqrt m)$}{Lg}}\label{sub37} 
If $\eta =x+ y\sqrt m$, we have $\eta^{p^{n_p} - 1} = 1 + p\,\alpha_p(\eta)$, $\alpha_p(\eta) = u + v\sqrt m$, 
whence $\Delta^\theta_p (\eta)\equiv 2v\sqrt m \pmod p$ for $\theta \ne 1$. 
Thus $\Delta^\theta_p(\eta)\equiv 0 \!\pmod p$ if and only if $v \equiv 0 \!\pmod p$. 

Suppose $m>0$ and that $\eta$ is a unit $\varepsilon$ of $\Q(\sqrt m)$; 
we have $u \equiv 0 \pmod p$ and 
still $\Delta^\theta_p (\varepsilon)\equiv 2v\sqrt m \pmod p$; the nullity modulo $p$ of 
$\Delta^\theta_p (\varepsilon)$ 
implies $\alpha_p(\eta)\equiv 0 \pmod p$, and $\varepsilon$ is a local $p$th power. 
In a conjectural viewpoint, we are reduced to the previous situation of a rational. 
Thus it would be sufficient to prove (via a suitable form of the $ABC$ conjecture) 
that the relation $\varepsilon^{p^{n_p}-1} \not\equiv 1 \pmod {p^{2}}$ occurs for infinitely 
many $p$, to be able to state the analogue
for $\varepsilon$ of the Conjecture \ref{conj291}, then the fact that if 
$\varepsilon^{p^{n_p}-1} \equiv 1 \pmod {p^2}$
for almost all $p$, then $\varepsilon = \pm 1$.

\subsection{Generalization for the degree \texorpdfstring{$n$}{Lg}}\label{sub380}
We may suppose that the above process is valuable for the general case where 
$\eta \in K^\times$ is arbitrary and
would be ``partial local $p$th power at $p$'' (Definition \ref{defi28}) for almost all $p$.
This supposes first the analysis of the local $p$th power case in the usual sense.

\subsubsection{Conjecture about the local \texorpdfstring{$p$}{Lg}th powers}\label{sub38} The rational case 
(Conjecture \ref{conj291}) showed the reasonableness of the following kind of statements corresponding
to the writing $\eta^{p^{n_p} - 1} - 1 = p\, \alpha_p(\eta)$,
in the case (statistiquely very exceptional) where $\alpha_p(\eta) \equiv 0 \pmod p$,
equivalent to ${\mathcal L} = \F_p[G]$) of probability $\Frac{O(1)}{p^n}$ 
(Remark \ref{rema260} and \S\,\ref{HP}).

\begin{conjecture}\label{conj31} Let $K$ be any number field and let $\eta \in K^\times$. If 
$\eta \in \prd_{v \div p} K_v^{\times p}$ for all prime number $p$ except a finite number, then $\eta$ 
is a root of unity of $K$.
\end{conjecture}

This could result from a generalization of the theorem of Silverman, using here the $ABC$ conjecture 
for the number fields (see for instance the paper of Waldschmidt \cite{W} giving an important list of 
applications and consequences). But the conjecture can be formulated independently.

\smallskip
This statement is to be compared with the very classical ``Hasse principle'' for powers, much stronger, 
and which is the following (cf. e.g. \cite[II.6.3.3]{Grcft}):

\begin{proposition}\label{prop32} Let $\Pl_K$ (resp. $\Pl_p$) be the set of places 
(resp. of $p$-places) of~$K$.

\smallskip
Let $\eta \in K^{\times}$ and let $p$ be a prime number; let $\Sigma$ be a 
finite set of places of $K$. 

\smallskip
(i) If $\eta$ is a local $p$th power for all place $v \in \Pl_K \setminus \Sigma$, then $\eta \in K^{\times p}$.

\smallskip
(ii) There exist infinitely many (non effective) sets $T$ of places of $K$ such that if $\eta$ 
is a local $p$th power for all place 
$v \in T$, then $\eta \in K^{\times p}$.\,\footnote{\,The classical statements always 
suppose that $\Sigma$ is finite (to eliminate some pathological places) in order to 
use the density theorem (Chebotarev) which is expressed by means
of particular progressions having canonical densities; but to be certain that some 
of these progressions (finite in number), necessary for the proof, meet the complementary 
of $\Sigma$, this one must be ``almost all'', because as soon as an (unknown) infinite 
family would be missing, it could be that ``by accident'' 
it contains the Frobenius that we need. We see the difference that may occur 
between a general algebraic reasoning and a reasoning on significantly less strong 
assumptions involving for instance sets $\Sigma$ of zero density (\S\,\ref{prem}).}
\end{proposition}

The difference, regarding the Hasse principle, operates in two times: starting from $p$ and the set $\Pl_p$, 
we begin to say, in the Conjecture \ref{conj31}, that $\eta$ is a local $p$th power for all $v \in \Pl_p$
 (i.e., we take the infinite set $\Sigma = \Pl_K \setminus\! \Pl_p$; or else we can say that we try to take
$T = \Pl_p$), but after we suppose that this local property (a kind of ``weaker Hasse principle'') 
is true for almost all $p$, in which case $\eta$ would be conjecturally in $K^{\times p}$ for almost all 
$p$ (Diophantine local-global principle), hence a root of unity.

\smallskip
The ``ultimate'' conjecture giving a link with the theories of the $\Delta^\theta_p(\eta)$ is Conjecture 
\ref{conj34} of \S\,\ref{sub39}. Before, let us examine the general case of units of number fields
which confirms the previous analysis..

\subsubsection{Particular case of the group of units -- Spiegelungssatz}\label{grunit}
We have the following specific statement (cf. e.g. \cite[II.6.3.8]{Grcft}).

\begin{proposition}\label{prop33} Let $\eta$ be a unit of $K$ and let $p$ be a prime number; let $S_p$ be a 
finite set of places of $K$ such that the $p$-class group of $K' := K(\mu_p)$ be generated by the 
$p$-classes of the prime ideals ${\mathfrak P_{v'}}$ of $K'$ for the places $v'$ of $K'$ above $S_p$.

\smallskip
If $\eta \in K_v^{\times p}$ for all place $v\in S_p \cup \Pl_p$, then $\eta \in K^{\times p}$.
\end{proposition}

The Conjecture \ref{conj31} only concerns the set $\Pl_p$ instead of 
$S_p \cup \Pl_p$ for a well-chosen finite set $S_p$ (not sufficiant to 
have a globale $p$th power), but we assume, in the conjecture, that 
this weaker hypothesis is true for almost all $p$. 
The two systems of assumptions coincide when the $p$-class group of the field $K'$
is trivial ($S_p = \emptyset$), but we can be more precise (cf. \cite[I.6.3.1 and II.1.6.3]{Grcft}).

\smallskip
Let $\eta$ be a Minkowski unit of a totally real field $K$; we can always choose 
$\eta$ non global $\ell$th power for all prime $\ell$.
If there exists a $\theta$-relation ${U_\theta} \not \equiv 0 \pmod {p\Z_{(p)}[G]}$
for which $\eta^{U_\theta} \in \prd_{v \div p} K_v^{\times p}$ 
(i.e. $\Delta_p^\theta(\eta)\equiv 0 \pmod p$, cf. \S\,\ref{interlocale}) 
then the extension $N' := K'(\sqrt[p] {\eta^{U_\theta}})$ of $K'$ 
is unramified and $p$-split which leads, by class field theory, 
to the following information: 
let $\Cl_{K'}^{\Pl'_p}$ be the quotient of the $p$-class group
$\Cl_{K'}$ by the $p$-subgroup of classes of prime ideals ${\mathfrak P}' \div p$ in $K'$ and let
$\theta^* := \omega \theta^{-1}$, where $\omega$ is the $p$-adic Teichm\"uller character defined
from a primitive $p$th root of unity $\zeta_p$ by 
$$\hbox{$\zeta_p^s = \zeta_p^{\omega(s)}\ $ for all $s \in {\rm Gal}(K'/K)$.} $$

Then it is the $\theta^*$-component of $\Cl_{K'}^{\Pl'_p}$ which is non trivial and 
${\rm Gal}(N'/K')$ is isomorphic to a quotient of $\Big( \Cl_{K'}^{\Pl'_p} \Big)^{e_{\theta^*}}$.
This is equivalent to the existence of a $p$-split  $\theta^*$-extension $N'$ of $K'$, of degree 
a power of $p$, contained in $K'(\sqrt[p] {F})/K'$, where $F$ (independent of $p$) is the $G$-module 
generated by $\eta$. Such a situation for infinitely many $p$ seems excessive. 

\smallskip
Apart from the case of units we have another situation:
take for $K=\Q$ the example of $\eta = a\in \Q^\times$, $a \ne \pm 1$. 
Then the Proposition \ref{prop33} is no longer valid
because it only applies if the ideal $\eta\,Z_K$ is the $p$th power of an ideal, but if 
$a^{p-1} \equiv 1 \pmod {p^2}$ the extension $\Q'(\sqrt[p] a)/\Q'$ 
is unramified at $p$ (and $p$-split) but ramified at the places of $\Q'$ 
dividing $a$; if $T$ is the set of prime divisors of $a$, we must replace the 
$p$-Hilbert class field $H'$ of $\Q'$ by its generalization $H'{}^{T'}$, the 
maximal Abelian $p$-extension unramified outside the places of the set $T'$ above $T$.

\smallskip
This $p$-extension $H'{}^{T'}/\Q'$ is finite because $T$ does not contain $p$
 (it is essentially a $p$-ray class field $K'_{\mathfrak m'}$, ${\mathfrak m'}$
built on ${T'}$) and it plays a role analogous to that of $H'$; here we sall have 
$\theta^*=1^* = \omega$ and a similar analysis.

\subsection{Conjectures on the \texorpdfstring{$p$}{Lg}-adic regulators 
\texorpdfstring{${\rm Reg}_p^G(\eta)$}{Lg}}\label{sub39}
The results of \S\,\ref{interlocale} invite to propose the following conjectures 
stronger than the conjectures of \S\S\,\ref{sub35}, \ref{sub380}. 

\begin{conjecture}\label{conj341} Let $K/\Q$ be a Galois extension of degree $n$, 
of Galois group $G$.
Let $\eta \in K^\times$ be such that the multiplicative $\Z[G]$-module $F$ generated 
by $\eta$ is of $\Z$-rank $n$. Then for all $p$ large enough, $\eta$ is not a partial local 
$p$th power at $p$, in other words, we have 
$\Big \{\eta_0 \in F,\ \eta_0 \in \prd_{v \div p} K_v^{\times p} \Big\}=F^p$, 
equivalent to ${\mathcal L}(\eta)=\{0\}$.
\end{conjecture}

The following statement is in fact equivalent to the previous one.
Recall that for all irreducible $p$-adic character $\theta$ of $G$, we have 
${\rm Reg}_p^\theta(\eta) \equiv \Delta_p^\theta(\eta) \!\!\pmod p$ and that 
$${\rm Reg}_p^G(\eta) := p^{-n}\, \hbox{det} \big ({\rm log}_p(\eta^{\tau\sigma})\big )_{\sigma, \tau \in G}$$

(the normalized $p$-adic regulator of $\eta$, cf. Definitions \ref{defimodif} (i)) is factorized into
$${\rm Reg}_p^G(\eta) = \prd_\theta {\rm Reg}_p^\theta(\eta)^{\varphi(1)} \ \hbox{(Remark \ref{rema110}).}$$

\begin{conjecture}\label{conj34} Let $K/\Q$ be a Galois extension of degree $n$, of Galois group $G$.
Let $\eta \in K^\times$ be such that the multiplicative $\Z[G]$-module generated by 
$\eta$ is of $\Z$-rank $n$, and let ${\rm Reg}_p^G(\eta)$ the normalized $p$-adic regulator of $\eta$.

Then for all $p$ large enough, ${\rm Reg}_p^G(\eta)$ is a $p$-adic unit.
\end{conjecture}

\begin{remark}\label{rema35} The Conjecture \ref{conj34} implies 
the Leopoldt--Jaulent conjecture \cite{J} for all prime $p$ except 
a finite number, but it is preferable to admit this last one, very classical, 
and to say that the Conjecture \ref{conj34} is a stronger version (cf. \S\,\ref{global}, (a) and (b)).
By negation, we get that if there exist infinitely many primes $p$ 
such that  ${\rm Reg}_p^G(\eta) \equiv 0 \pmod p$,
then the $\Z$-rank of the $\Z[G]$-module generated by $\eta$ is $<n$.
\end{remark}

\subsection{Conjectures about the Abelian 
\texorpdfstring{$p$}{Lg}-ramification for real fields} \label{sub360}
Let $H^{p\rm r}$ be the maximale Abelian $p$-ramified (i.e., unramified outside $p$) 
$p$-extension of a {\it real Galois number field $K$ satisfying the Leopoldt conjecture for all $p$}. 
Let $\widehat K$ be the cyclotomic $\Z_p$-extension of $K$ and let 
${\mathcal T}_p = {\rm Gal} \big (H^{p\rm r}/\widehat K \big )$. For all $p$ large enough, 
$\vert {\mathcal T}_p\vert $ has the same $p$-adic valuation as the normalized regulator of $K$

\medskip
\centerline{$p^{1-n}\,{\mathcal R}_p(K) \sim 
\prd_{\theta \ne 1} {\rm Reg}_p^\theta(\varepsilon)^{\varphi(1)}$, } 

where $\varepsilon$ is a fixed suitable Minkowski unit of $K$ (\cite{Coa}, \cite[III.2.6.5]{Grcft}).

\smallskip
The Conjecture \ref{conj34} implies the following conjecture that we may 
state for a non necessary Galois (nor real) field because if $K$ is {\it any}
Galois field and ${\mathcal T}_p$ the torsion subgroup of the Galois group
${\rm Gal} \big (H^{p\rm r}/\widetilde K \big )$, where $\widetilde K$ is the
compositum of the $\Z_p$-extensions of $K$, for $K' \subseteq K$, 
${\mathcal T}_p(K')$ is isomorphic to a subgroup of ${\mathcal T}_p(K)$ 
under the Leopoldt conjecture
(\cite[IV, \S 2]{Grcft}), and any component ${\mathcal T}_p^\theta$, 
$\theta$ odd, is trivial for all $p$ large enough since it depends on the 
$\theta$-component of the $p$-class group of $K$ (\cite[III.2.6.1, Fig. 2.2]{Grcft}):

\begin{conjecture}\label{conj345}
The invariant $\prd_p {\mathcal T}_p$ is finite, for all number field satisfying 
the Leopoldt conjecture for all $p$.
\end{conjecture}

\begin{remarks}\label{prat}
(i) Recall that an arbitrary field $K$, such that $ {\mathcal T}_p=1$ under 
the Leopoldt conjecture, is said to be {\it $p$-rational}
and that in this case the arithmetic of $K$ becomes essentially trivial (see a synthesis of the properties in 
\cite[IV.3\,(b)]{Grcft}, \cite{MN}, and the links with the $p$-regularity whose beginnings are in
\cite{GrK2} then \cite{JN}, among an abundant subsequent bibliography on the subject
recalled in \cite{GrBP}). 

\smallskip
For $K$ real and for all $p\geq 2$, the $p$-rationality implies easily the 
Greenberg conjecture (\cite{Gre}, \cite{GrCG}), 
which clarifies the context.

\smallskip
(ii) We shall deduce, from the above, analogous properties on the residue of the $p$-adic z\^eta function 
(\cite[Appendix]{Coa}, and \cite{Se2}). When the $p$-valuation of 
$\zeta_K(2-p)$ is negative, it is equal to $-1$ (\cite[Th\'eor\`eme 6]{Se2}); in the contexte of Conjecture 
\ref{conj345}, we would have $\Frac{\zeta_K(2-p)}{\zeta_\Q(2-p)} \sim \vert {\mathcal T}_p\vert = 1$ 
for all $p$ large enough (\cite{Hat}).

\smallskip
(iii) Let $S$ be the set consisting of the $p$-places of $K$ and infinite places, 
and let $G_S(K)$ be the Galois group of the maximal $S$-ramified (i.e., unramified 
outside $S$) algebraic extension of $K$; then in a cohomological point of view, we 
would have the duality ${\rm H}^2(G_S(K), \Z_p) \simeq {\mathcal T}_p^*=1$, for all 
$p$ large enough.
\end{remarks}

\subsection{More general cohomological justifications.}\label{subBK}
 We give here some comments about results whose mathematical level largely exceeds
 any heuristic approach, but this confrontation has seemed to us very convincing.
We may refer to several papers in \cite{BK} of which \cite{Ko} and \cite{Ng}. 

\smallskip
The main central idea, related to the conjecture of Bloch--Kato, is that there exists, 
in a rather systematic way, global {\it finite} invariants whose $p$-adic specializations, 
of a cohomological nature, are the arithmetical objects 
(more or less classical) of a number field $K$ (as the $p$-class groups, 
the groups ${\mathcal T}_p$ of $p$-ramification, 
some $p$-adic regulators, certain \'etale cohomological groups,\ldots). 
This conjectural point of view is universally admitted, all the more that 
some proofs have been given quite extensively.
Let us recall briefly the main known results:

\smallskip
We start from the notation of \S\,\ref{sub360}. For $m\in \Z$, let $\Z_p(m)$ be the $G_S(K)$-module $\Z_p$ 
provided with the action defined by the character $\chi^m$, where
$\chi : G_S(K) \to \Z_p^\times$ is the character of the action of $G_S(K)$ on $\mu_{p^\infty}$. 

\smallskip
We say that $K$ is $(p, m)$-rational if ${\rm H}^2(G_S(K), \Z_p(m))$ is trivial; the usual $p$-rationality
mentioned \S\,\ref{sub360} corresponds to $m=0$ which seems to be the most delicate case. 
The finiteness of ${\rm H}^2(G_S(K), \Z_p(m))$
is equivalent to a $m$-analogue of the Leopoldt conjecture in terms of ``suitable $p$-adic regulators''.

\smallskip
The results on the finiteness of some global objects whose $p$-adic specializations are the
${\rm H}^2(G_S(K), \Z_p(m))$, for $m$ fixed, are the following ones (from private indications by 
Thong Nguyen Quang Do):

\smallskip
(i) For $m \geq 2$, this is a consequence of the ``Quillen--Lichtenbaum conjecture'' 
now Voevodsky Theorem. 
The finiteness comes from that of the ${\rm K}$-theory groups ${\rm K}_{2m-2}(Z_K)$ 
via a non trivial isomorphism of the following form ($p>2$)

\smallskip
\centerline {${\rm K}_{2m-2}(Z_K) \otimes \Z_p \simeq {\rm H}^2(G_S(K), \Z_p(m))$. }

\medskip
(ii) The case $m=1$ corresponds, under a similar form (due to the fact that the cohomology group is not finite 
because of the Brauer group), to the Gross conjecture, and the case $m<0$ is essentially unknown: 
if $m<0$ is odd, then ${\rm H}^2(G_S(K), \Z_p(m))$ is finite; the case $m<0$ even is unknown.

\smallskip
(iii) The case $m=0$ defines the framwork ``Leopoldt conjecture and torsion $p$-group ${\mathcal T}_p$,
dual of ${\rm H}^2(G_S(K), \Z_p)$'', a framework in which a similar situation is conjectured, in the line 
of the previous ``motivic work'' of Voevodsky. 

\smallskip
Thus our conjectural Diophantine approach is enforced by the deep results recalled above, 
the interest being that the notion of (normalized) $p$-adic regulator of an arbitrary algebraic number 
is more general.

\section{Conclusion}\label{concl} 
We have tried to give a maximum of justifications, in particular by the fact that when the probabilities 
of $p$-divisibility of ${\rm Reg}_p^G(\eta)$ are at most $\Frac{O(1)}{p^2}$, the heuristic principle of 
Borel--Cantelli suggests a finite number of solutions $p$ and even no solution most of the time 
since the sum of the $\Frac{1}{p^2}$ is very small

\smallskip
\centerline{$\sm_{p\geq 2} \, \Frac{1}{p^2} \approx 0.45$, 
$\ \ \  \sm_{p\geq 10^4}\, \Frac{1}{p^2} \approx 9 \times 10^{-6}$.}

\smallskip
It remains the case of minimal $p$-divisibility ${\rm Reg}_p^G(\eta) \sim p^{\varphi(1)}$ (Definition \ref{defidec}) 
which is a possible obstruction if the Heuristic \ref{heur2} is inaccurate; in that case, the ``expected number
of solutions'' $p \leq x$ would be $O(1){\rm log}_2 (x)+O(1)$ and the corresponding arithmetical
 $p$-adic invariants 
(seen in \S\,\ref{subBK}) would have, for all $p$ large enough, a minimal canonical structure of $G$-module
(e.g. ${\rm H}^2(G_S(K), \Z_p) \simeq V_\theta$ for a unique $\theta$ such that $f=\delta=1$).

\smallskip
It would be useful to have an analytical estimation of $M_p$ which precises the notions of 
exceptional and abundant solutions (cf. \S\,\ref{repetitions}; see also \cite{Grcompl}). 

\smallskip
But if there is some consistency of mathematics, then we can believe that such conjectures 
of finiteness are legitimate.

\smallskip
For instance, we can deduce from this study that the Leopoldt--Jaulent conjecture on the non nullity of
the $p$-adic regulators is an extremely weak form of the reality.

\subsection*{Acknowledgments} I would like to thank  J\'an Min\'a\v{c} for his friendly 
support about this article, Thong Nguyen Quang Do for his ``cohomological'' comments, 
and G\'erald Tenenbaum for useful information on probabilistic number theory. Finally I warmly 
thank the anonymous Referee for his careful reading and remarks.

\end{document}